\documentclass[11pt]{article}
\usepackage{myarticle-macrosShared}
\usepackage{graphicx,subfig,enumerate,bm}
\usepackage{geometry}
\usepackage{color}

\newcommand{\betaHat}{\widehat{\beta}}
\newcommand{\betaRidge}{\betaHat_{\text{ridge}}}
\newcommand{\wHat}{\widehat{w}}
\newcommand{\SigmaTilde}{\widetilde{\Sigma}}
\newcommand{\Proba}[1]{\mathbf{P}\left(#1\right)}
\newcommand{\limitScaling}{\mathfrak{s}}
\newcommand{\median}[1]{\textrm{median}\left(#1\right)}
\newcommand{\ellipticalData}{\mathfrak{X}}
\usepackage{setspace}
\newtheorem{heuristic}{Heuristic}[section]

\title{Geometric sensitivity of random matrix results: consequences for shrinkage estimators of covariance and related statistical methods}
\date{\today}

\begin{document}

\author{Noureddine El Karoui\thanks{Support from an Alfred P. Sloan research Fellowship and NSF grants DMS-0605169 and DMS-0847647 (CAREER) is gratefully acknowledged. N. El Karoui is very grateful to Professor Friedrich Goetze for his hospitality in Bielefeld (and CRC 701) in the summer of 2008.
\textbf{Contact~:} \texttt{nkaroui@stat.berkeley.edu}} \and Holger K\"osters\thanks{Research partially supported by CRC 701,
``Spectral Structures and Topological Methods in Mathematics''. This work was initiated while the second author
was visiting the Department of Statistics at UC Berkeley (in September 2009)
which he would like to thank for its hospitality.
\textbf{Contact~:} \texttt{hkoesters@math.uni-bielefeld.de}
\textbf{Key words and Phrases~: } Random matrices, shrinkage estimators, Linderberg method, concentration inequalities, Burkholder inequality, Efron-Stein inequality, regularized discriminant analysis, linear discriminant analysis, Markowitz problem, quadratic programming.
}
}
\maketitle
\begin{abstract}
Shrinkage estimators of covariance are an important tool in modern applied and theoretical statistics. They play a key role in regularized estimation problems, such as ridge regression (aka Tykhonov regularization), regularized discriminant analysis and a variety of optimization problems

In this paper, we bring to bear the tools of random matrix theory to understand their behavior, and in particular, that of quadratic forms involving inverses of those estimators, which are important in practice.

We use very mild assumptions compared to the usual assumptions made in random matrix theory, requiring only mild conditions on the moments of linear and quadratic forms in our random vectors. In particular, we show that our results apply for instance to log-normal data, which are of interest in financial applications.

Our study highlights the relative sensitivity of random matrix results (and their practical consequences) to geometric assumptions which are often implicitly made by random matrix theorists and may not be relevant in data analytic practice.
\end{abstract}
\section{Introduction}
Modern multivariate statistics is increasingly high-dimensional. It is now easy to collect many samples ($n$) with a large number of covariates ($p$) for each sample. In this paper, we will therefore study multivariate statistical problems in the ``large $n$, large $p$" setting that is increasingly popular in theoretical statistics. By this we mean that we will study certain statistics in the asymptotic setting where $n$, the number of observations, is going to infinity, and $p$, the number of predictors, is also going to infinity. Our focus will be on the situation where $p/n$ remains bounded.

The paper is mostly concerned with forms involving the inverse of a shrunken covariance matrix, or powers of this inverse as they play a key role in several important statistical problems that we review later in this introduction. As a matter of fact, these objects, in one form or another, are central in many aspects of classical regularized methods in statistics and other fields of applied mathematics. The purpose of this paper is to explain how these regularized estimators behave in the ``large $p$, large $n$" setting and derive some understanding and insights about the behavior of widely used methods that rely on them.

In classical statistics, when $p\ll n$, one can get a good estimate of the spectral properties of $\Sigma$, the population covariance matrix, by using its ``naive" counterpart, the sample covariance matrix $\SigmaHat$, with, if $\muHat$ is the sample mean of our vectors,
$$
\SigmaHat=\frac{1}{n-1}\sum_{i=1}^n (X_i-\muHat)(X_i-\muHat)\trsp\;.
$$
As is now well-known, this is not the case when $p$ is comparable to $n$, which we denote by $p \asymp n$. In that setting, even though the central limit theorem and a little bit of concentration of measure guarantee under broad assumptions that
$$
\max_{i,j} |\SigmaHat(i,j)-\Sigma(i,j)|\tendsto 0\;,
$$
(even when $p\gg n$), the eigenvalues of $\SigmaHat$ tend to be very different from those of $\Sigma$ (see \cite{imj} or the reviews \cite{JohnstoneReview07}, \cite{NekReviewHandbookChap28}). Hence, it is important to understand the performance of our standard techniques in this new asymptotic setting.

Recent papers concerned with these types of problems and their implications for concrete applications are for instance \cite{nekMarkoRisk} and \cite{nekMarkoRealizedRisk}, where the author showed that somewhat surprisingly for a broad class of covariance matrices, means and distributions for the data, one could characterize the performance of estimators as a function of the ratio $p/n$, and hence get consistent estimators for parameters, such as the efficient frontier in classical portfolio theory, that appear difficult to estimate without structural assumptions on the population parameters. In these papers, the regularization came under the form of linear constraints on the vector of interest.

As mentioned before, shrinkage estimators of covariance are fundamental objects in modern statistics, partly because of James-Stein type phenomena (\cite{haff80}) and they are very widely used.
Here are a few examples.
\begin{enumerate}
\item Classification (LDA, RDA):  when we observe data coming from two Gaussian populations, with different means $\mu_1$ and $\mu_2$, priors $\pi_1$ and $\pi_2$ but same covariance matrix $\Sigma$, the optimal classification rule is known to be Fisher's linear discriminant analysis rule: classify an observation $x$ to class 2, if
$$
x\trsp \Sigma^{-1}(\mu_1-\mu_2)> T(\mu_1,\mu_2,\Sigma,\pi_1,\pi_2)\;,
$$
where $T(\mu_1,\mu_2,\Sigma,\pi_1,\pi_2)$ is a known threshold. Naturally, we do now know $\Sigma$ in practice, so a natural method is to replace it by $\SigmaHat$. This is what is usually done in LDA (see \cite{htf09}). In \cite{FriedmanRDA89}, concerned by, among other things variance issues in LDA, Friedman proposed to use RDA, regularized discriminant analysis, where instead of using $\SigmaHat$ as an estimate of $\Sigma$, one uses $\SigmaHat+A$ or $(1-\theta)\SigmaHat+\theta A$, i.e a shrinkage estimator of covariance. This estimator has also been proposed by \cite{LedoitWolf04} in another context. It is natural to ask what happens when using these estimators in high-dimension.
\item Shrinkage estimators of covariance: a classic paper on the topic is \cite{haff80}; we also refer to \cite{anderson03}, for explanations concerning the benefit of skrinkage. In portfolio optimization, at least in the traditional mean-variance framework, similar issues arise. Hence partly motivated by this problem, \cite{LedoitWolf04} proposed to use a shrinkage estimator to solve the portfolio optimization problem and get regularized solutions. In the finance literature, there are ``finance-driven" shrinkage estimators, like the one arising in the Black-Litterman model (see \cite{MeucciBook}).
\item Regression problems: in ridge regression, where one seeks $\beta$ to optimize $\norm{Y-X\beta}+\lambda \beta\trsp\Gamma \beta$, one also encounters matrices of the form $\SigmaHat+\lambda \Gamma$, which is a shrunken version of $\SigmaHat$. The $\Gamma$ that is usually taken is $\id$, this regularization amounts to modifying the eigenvalues of $\SigmaHat$.
\end{enumerate}

In the analysis of all these methods, one needs to understand the behavior of the matrix $(\SigmaHat+A)^{-1}$ (entrywise and/or globally) as well as similar quantities involving $(\SigmaHat+A)^{-1}\Sigma_\eps (\SigmaHat+A)^{-1}$ (where $\Sigma_\eps$ is positive semidefinite) and this will be one of the focuses of the paper. It is tantalizing to use random matrix theory to do so, a program we got started on in \cite{nekMarkoRisk} and \cite{nekMarkoRealizedRisk}. However, as documented in these papers, random matrix theory has several potential pitfalls: standard random matrix models, though in appearance general, put implicitly very strong geometric constraints on the datasets they are supposed to model. In light of this, one might be wary that the remarkable results that come out of it are just consequences of this geometry, which may or may not be similar to the one a practitioner encounters in practice. Hence we feel that any analysis that is not doing a meaningful robustness analysis is sorely lacking.

As we have documented before, the geometric constraints put by classical random matrix theory on the datasets modeled by it are due to manifestations of the concentration of measure phenomenon. Hence, it seems to us that a good starting point for the analysis of shrunken covariance matrices and their applications is that of generalized elliptical distributions, where the data is modeled as
$$
\ellipticalData_i=\mu+R_i X_i\;,
$$
where $R_i$ is a random variable independent of $X_i$ and $X_i$ has some (mild) concentration properties. (This will be made clear and precise later.)

The advantage of this class of models is that it contains the Gaussian model that is popular with many researchers, though now understood to be lacking in many fundamental ways. When $\Exp{R_i^2}=1$, then $\scov{X_i}=\scov{\ellipticalData_i}$, so we can study robustness of our results in this class, since all the population parameters (which will depend on covariance and mean) will be the same.

However, by studying the model at this level of generality, we will not be able to rely on various invariance properties of the Gaussian distribution, and hence will really use only the geometric/concentration properties of the random variables of interest. One advantage of such an approach is that these properties are somewhat checkable in practice, through simple histograms for e.g norms and scalar products of points in the dataset, as has been explained before in some of the works cited above. Crucially, by showing that the results depend on the properties of $\{R_i\}_{i=1}^n$, we will able to show that even in our simple setting the geometry is key (change in $R_i$'s may mean change in the geometry) and a major contributing factor in the robustness of the results. Finally, it should be noted (see \cite{nekMarkoRisk}) that one can sometimes study the bootstrap properties of various estimators by studying the class of elliptical distributions. Hence our analysis could be used to gain insight into bootstrap properties of various estimators.

The focus of our paper will mostly be on entrywise properties of $(\SigmaHat+A)^{-1}$ or $(\SigmaHat+A)^{-1}\Sigma_\eps (\SigmaHat+A)^{-1}$ in the class of models we consider, which naturally appear in the study of the risk of certain procedures. Quadratic forms involving the sample mean are also important in practice and will be studied. Random matrix theory already handles well things like $\trace{(\SigmaHat+A)^{-1}}$, and other questions concerning only eigenvalues, so we will not spend too much time on this, though they are potentially important in the study of the risk of various estimators.

Beside shedding light on central statistical questions in multivariate analysis, our analysis also proposes what we think is a good and generic technical framework for carrying them out: namely we will do our work through invariance principles and mild concentration work. We will show that the statistics we are considering are asymptotically non-random, by showing that they are concentrated around their mean. And then we will show that the mean is the ``same" in a broad class of models by using techniques akin to the Lindeberg method. A main difficulty is then to compute the mean (in many problems it is much harder to compute the mean of a statistic than to show that e.g its variance goes to zero), but our analysis will show that it can be done for favorable distributions in the class considered, and the Gaussian distribution will then be heavily used. Importantly, our analysis is very general and shows robustness even in classes where we have not or cannot at this point compute a limit for the quantity of interest.

We should also point out that our concentration requirements on $X_i$ have purposely been kept to a minimum and hence our results extend way beyond the traditional ``linear combination of i.i.d" framework which has been popular in random matrix theory following the nice work of Bai and Silverstein (see e.g \cite{silverstein95}, \cite{silversteinbai95}). In particular, we will be able to handle (multivariate) log-normal distributions and other non-linear deformations of Gaussian random variables. Also, conditions on i.i.d-ness are essentially replaced by conditions on the mean and covariance of the random variables we deal with, as well as a little bit of concentration for linear and quadratic forms involving them. Our aim was also to show that these ``universality" results could get obtained rather simply so an effort has been made to make the proofs as simple as possible. The paper is a bit long because we treat many cases in details and at what we think is the right level of generality.

Finally, it will be noted by researchers interested in probability that some of our results can be seen as strong versions of classic random matrix results: where classic results gave results about normalized traces of certain random matrices, we will be able to have statements valid for each element of the diagonal of the matrix of interest.

In section 2, we present some of our main technical results and heuristic justification for some of the main results, which should be helpful for statisticians wanting to get a sense of where the results come from. Section 3 contains most proofs and the core technical work. Section 4 discusses some potential applications to statistics, where at this point our main results shed light on existing procedures and ``what they really do". We conclude in Section 5 and present a result of independent interest on Stieltjes transforms in the Appendix.

\section{Strategy and exposition of some results}
Our strategy is to make use of invariance principles and concentration inequalities throughout the paper. Practically, this translates into showing that the statistics we care about are concentrated around their means, that is the concentration part. In a second step, we show that this mean does not depend of the distribution of the data, as long as certain moment conditions are satisfied. To do so, we employ techniques very similar to the Lindeberg method (\cite{StroockProbaBook} and let us note that it has been perhaps ``re-popularized" by the nice work of Chatterjee in this direction, e.g \cite{ChatterjeeSimpleInvariance05}).

Throughout the paper, we will focus on model of an elliptical type, namely we observe i.i.d observations
$$
\ellipticalData_i=\mu+R_i X_i\;,
$$
where the $X_i$'s are independent and independent of $R_i$. The $R_i$'s are allowed to be dependent. Our efforts will go into relaxing distributional assumptions on $X_i$, while assuming only two moments on $R_i$ - the justification for these choices coming from applications discussed at the end of the paper. In particular, this means that we will be able to handle data with relatively heavy tails.

A main tool in our work will be a simple extension of the Efron-Stein inequality - which will allow us to characterize higher moments of the statistics we care about. This extension is likely known in martingale theory but we present a proof in the appendix for the convenience of the reader. We delay it statement and presentation to the proof section and start by highlighting some of our main results.

\subsection{A generalized version of the Efron-Stein inequality}
We will make repeated use of the following lemma, which follows from Burkholder's inequality (see \cite{BurkholderWaldLectures73}).
\begin{lemma*}
Suppose $W=h(X_1,\ldots,X_n)$, where the $X_i$'s are independent. We call ${\cal F}_j=\sigma(X_1,\ldots,X_j)$.  We also denote by $Z_m$ a (measurable) function of $(X_1,\ldots,X_{m-1},X_{m+1},\ldots,X_n)$.

Then, we have, for a constant $c$ that depends only on $k$, and for $k\geq 2$,
\begin{equation}
\Exp{\left|W-\Exp{W}\right|^k}\leq c \left(\Exp{\left[\sum_{m=1}^n \Exp{(W-W_m)^2|{\cal F}_{m-1}}\right]^{k/2}}+\sum_{m=1}^n \Exp{|W-W_m|^k}\right)\;.
\end{equation}
\end{lemma*}
The classic Efron-Stein inequality corresponds to the case where $k=2$. The advantage of using higher $k$'s is that it will for instance allow us to control $\max_{j \in J} |W_j-\Exp{W_j}|$ for $J$'s of higher cardinalities. For instance, if we can show that $\Exp{\left|W_j-\Exp{W_j}\right|^k}\leq C n^{-k/2}$ for a certain $k$, a simple union bound gives us
$$
P(\max_{j \in J} |W_j-\Exp{W_j}|>t)\leq \frac{C |J|}{(n^{1/2} t)^k}\;.
$$
Hence a bound valid of $k>2$ will allow us to handle greater $J$'s. A number of applications (involving for instance thresholding) also require control of higher moments, which will be provided by our methods.

We also note that we purposely tried to avoid deriving central limit theorems. While those are definitely interesting, we wanted to have finite sample bounds and have them be relatively robust with respect to distributional assumptions, in keeping with what we view as their potential practical usefulness.

\subsection{Quadratic forms in inverse of shrunken sample covariance matrices are essentially deterministic}
We now state an application of the previous Lemma  to forms which are at the center of our study.

\begin{theorem*}
Suppose $X_1,\ldots,X_n\in \mathbb{R}^p$ are independent. Suppose further that $\Exp{X_i}=0$ and, if $v$ is such that $\norm{v}=1$, $\Exp{|X_i\trsp v|^k}\leq b_L(k;X_i)$, where $b_L(k;X_i)$ is a deterministic function depending only on the distribution of $X_i$ and $k$. Call
$$
{\cal S}=\frac{1}{n}\sum_{i=1}^m R_i^2 X_iX_i\trsp\;,
$$
where $R_i$ are deterministic.

Call $M(t)={\cal S}+A$, and assume that for some $t>0$, $A$ is positive definite, with $A\succeq t \id_p$. Then, if $\norm{x}=1$,
$$
\Exp{|x\trsp [{\cal S}]^{-1} x-\Exp{x\trsp [{\cal S}]^{-1} x}|^k}\leq \frac{c_k}{t^{2k}}\left[\left(\sum_{i=1}^n \left[\frac{R_i^4}{n^2} b(4;X_i)\wedge t^2 \right]\right)^{k/2}+\left(\sum_{i=1}^n \left[\frac{R_i^{2k}}{n^{k}} b(2k;X_i) \wedge t^{k}\right]\right)\right]\;.
$$
\end{theorem*}

It is perhaps instructive to give an example at this point. Here are two.
\begin{itemize}
\item Suppose that $X_i$ satisfies $P(|X_i\trsp v|>t)\leq C \exp(-c t^b)$, and $X_i$ has mean 0. Then
$$
b_L(k;X_i)\leq \frac{C}{c^{k/b}} \frac{k}{b} \Gamma\left(\frac{k}{b}\right)\;.
$$
\item Suppose that $X_i$ satisfies $P(|X_i\trsp v|>t)\leq C t^{-b}$. Then if $b>(k+1)$,
$$
b_L(k;X_i)\leq C\left(1+\frac{1}{b-(k+1)}\right)\;.
$$
\end{itemize}

We note that the condition on the $X_i$'s is rather minimal: all we need is some concentration of linear forms in $X_i$, something that might seem surprising at first.

The exponential deviation inequality in our first example might look like a strong assumption. However, it is satisfied by many distributions, with quite non-linear structures which would be difficult to analyze if one did not resort to concentration of measure statements. The (centered) Gaussian copula is a good example. We give specific examples in Subsubsection \ref{subsubsec:ApplicabilityOfAssumptions}.

The result also gives us a reasonable understanding of the size of the fluctuations behavior of the quadratic forms we are interested in. Note that using the Gaussian case (at $t=0$) as a comparison, the fluctuation size of $n^{-1/2}$ seems to be the right one.

\paragraph{General strategy} The general strategy is now clear. In light of the previous theorem, if we can get a good deterministic approximation to $\Exp{S(t)^{-1}}$, we will be able to get an approximation of $x\trsp  S(t)^{-1} x$. Note that the considerable simplification here is that we are not dealing with random variables anymore. Fortunately, we can approximate this expectation using variant of methods that have been developed in the random matrix literature (specifically the part of the theory concerned with understanding limiting spectral distributions). Also, it will be possible to show that these expectations do not vary much when we change some details of the distributions - this is the essence of Lindeberg-style ideas. Hence, all we will have to do is show that the expectations in question do not change much when we replace $X_i$'s by $Y_i$'s with a different distribution (but the same covariance and mean). And then compute the expectation in a favorable case, for instance when $X_i$'s are Gaussian.

\subsection{Heuristics}\label{subsec:Heuristics}
To help readers unfamiliar with random matrix theory understand better the results, we now present heuristics that help us guess the results. Formal proofs essentially start from these conjectures and proceed to verify that they are indeed correct.

We will focus on two types of quantities:
$$
v\trsp ({\cal S}+A)^{-1} v \text{ and } v\trsp (\SigmaHat+A)^{-1} B (\SigmaHat+A)^{-1} v\;,
$$
where $A$ and $B$ are positive definite matrices.

Also, ${\cal S}=\frac{1}{n}\sum_{i=1}^n R_i^2 X_iX_i\trsp$, where $X_i=\Sigma^{1/2} Y_i$, where $Y_i$ has covariance $\id_p$, and $X_i$ (or $Y_i$) satisfies mild concentration inequalities - the details are given when we undertake a rigorous proof. At this point, the reader can safely assume that $X_i$ is ${\cal N}(0,\Sigma)$ (so $Y_i$ is ${\cal N}(0,\id_p)$. In other words, ${\cal S}$ is the ``sample" covariance matrix we would use if we knew the mean of the data.

We have the following heuristic result:
\begin{heuristic}\label{heuristic:regularizedInverse}
Under regularity conditions, we have
$$
v\trsp ({\cal S}+A)^{-1} v \simeq v\trsp (\gamma(A)\Sigma+A)^{-1} v\;,
$$
where if
$$
\alpha(A)=\frac{1}{n}\trace{\Sigma ({\cal S}+A)^{-1}}\;,
$$
$\alpha(A)$ has an asymptotically deterministic equivalent and
$$
\gamma(A)\simeq \frac{1}{n}\sum_{i=1}^n \frac{R_i^2}{1+R_i^2 \alpha(A)}\;.
$$
\end{heuristic}

\begin{proof}[Argument: ]
The key element of this argument is really the concentration of quadratic forms in $Y_i$, which allow us to replace quantities of the type $Y_i\trsp M Y_i/p$ by $\trace{M}/p=\Exp{Y_i\trsp M Y_i}/p$.

The fact that $\frac{1}{n}\trace{\Sigma ({\cal S}+A)^{-1}}$ has an asymptotically deterministic equivalent comes from standard arguments in random matrix theory (for some that rely on concentration and are just a few lines, see \cite{nekCorrEllipD}; see also Subsection \ref{subsec:CheckHeuristics}).
Let us write ${\cal S}=\sum_{i=1}^n r_i r_i\trsp$, where $r_i$ are independent. Now, we have (using an idea akin to some in \cite{silverstein95} and now classic in random matrix theory)
$$
{\cal S}({\cal S}+A)^{-1}=\id-A({\cal S}+A)^{-1}\;,
$$
and hence, using the fact that $(r_i r_i\trsp +M_i)^{-1}=M_i^{-1}-\frac{M_i^{-1} r_i r_i\trsp M_i^{-1}}{1+r_i\trsp M_i^{-1} r_i}$,
$$
A({\cal S}+A)^{-1}=\id-\sum_{i=1}^n \frac{r_ir_i\trsp M_i^{-1}}{1+r_i\trsp M_i^{-1} r_i}\;,
$$
where $M_i={\cal S}+A-r_ir_i\trsp$.

Therefore, if $v$ and $u$ are two vectors,
$$
v\trsp A({\cal S}+A)^{-1} u=v\trsp u-\sum_{i=1}^n \frac{v\trsp r_i r_i\trsp M_i^{-1} u}{1+r_i\trsp M_i^{-1} r_i}\;.
$$
Now because $Y_i$ satisfies a dimension-free concentration inequality, we have, if $M$ is a matrix independent of $Y_i$, $Y_i\trsp M Y_i/p\simeq \trace{M}/p$. Applying this heuristic in each term of the previous sum, we get,
$$
v\trsp A({\cal S}+A)^{-1} u=v\trsp u-\frac{1}{n}\sum_{i=1}^n  \frac{R_i^2 v\trsp \Sigma M_i^{-1}u}{1+R_i^2 \frac{1}{n}\trace{\Sigma M_i^{-1}}}\;.
$$
Now not much is lost by replacing $M_i$ by ${\cal S}+A$ everywhere in the previous expression. Hence, we have heuristically,
\begin{align*}
v\trsp A({\cal S}+A)^{-1} u&=v\trsp u-\left[\frac{1}{n} \sum_{i=1}^n \frac{R_i^2}{1+R_i^2 \frac{1}{n}\trace{\Sigma ({\cal S}+A)^{-1}}}\right] v\trsp \Sigma ({\cal S}+A)^{-1}u\;,\\
&=v\trsp u-\gamma(A) v\trsp \Sigma ({\cal S}+A)^{-1}u\;.
\end{align*}
Another way of rewriting this equation is simply
$$
v\trsp ({\cal S}+A)^{-1} u = v\trsp A^{-1} u - \gamma(A) v\trsp A^{-1} \Sigma ({\cal S}+A)^{-1}u\;.
$$
Now, let us call $v_k=(A^{-1}\Sigma)^k v$. Applying the previous heuristic to $v=v_k$ and $u=v$, we have if $\beta_k=v_k\trsp ({\cal S}+A)^{-1} v$, and $\alpha_k=v_k\trsp A^{-1} v$,
$$
\beta_k\simeq \alpha_k-\gamma(A) \beta_{k+1}\;.
$$
Assuming that we can use the previous approximation many times, we get
$$
\beta_0\simeq \sum_{j=0}^n (-\gamma(A))^j \alpha_j + (-\gamma(A))^{n+1} \beta_{n+1}\;.
$$
Now assuming that we can sum the series and that $(\gamma(A))^{n+1} \beta_{n+1}\tendsto 0$, we get
\begin{align*}
\beta_0\simeq \sum_{j=0}^\infty (-\gamma(A))^j \alpha_j&=v\trsp\left[\sum_{j=0}^{\infty} (-\gamma(A))^j (A^{-1}\Sigma)^j \right] A^{-1} v\\
&=v\trsp (\id+\gamma(A)A^{-1}\Sigma)^{-1}A^{-1} v=v\trsp (A+\gamma(A)\Sigma)^{-1} v\;.
\end{align*}
Note that $\beta_0=v\trsp ({\cal S}+A)^{-1} v$. Hence, it is perhaps reasonable to conjecture that
$$
v\trsp ({\cal S}+A)^{-1} v\simeq v\trsp (A+\gamma(A)\Sigma)^{-1} v\;.
$$
Note that the heuristic also gives us conjectures for approximating the value of $v\trsp ({\cal S}+A)^{-1} (A^{-1}\Sigma)^k v$, for any given $k$, as this is what we called earlier $\beta_k$.
\end{proof}

For dealing with higher powers of $({\cal S}+A)^{-1}$, we also need the following heuristic.
\begin{heuristic}\label{heuristic:squareRegularizedInverse}
Under regularity assumptions, we have
$$
v\trsp ({\cal S}+A)^{-1} B ({\cal S}+A)^{-1} v\simeq v\trsp (A+\gamma(A)\Sigma)^{-1} (B+\xi(A,B) \Sigma) (A+\gamma(A)\Sigma)^{-1} v\;,
$$
where $\gamma(A)$ is defined in Heuristic \ref{heuristic:squareRegularizedInverse} and
$$
\xi(A,B)=\left[\frac{1}{n}\sum_{i=1}^n \frac{R_i^4}{(1+R_i^2\alpha(A))^2}\right]\frac{1}{n} \trace{\Sigma({\cal S}+A)^{-1}B ({\cal S}+A)^{-1}}\;.
$$
Furthermore, $\xi(A,B)$ has an asymptotically deterministic equivalent.
\end{heuristic}

\begin{proof}[Argument : ]
Let us call $f(t)=v\trsp ({\cal S}+A(t))^{-1}v$. Then, since $([M(t)]^{-1})'=[M(t)]^{-1} M'(t) [M(t)]^{-1}$, we have
$$
f'(t)=-v\trsp ({\cal S}+A(t))^{-1} A'(t) ({\cal S}+A(t))^{-1} v\;.
$$
Now, if we consider $A(t)=A+t B$, we see that $A'(t)=B$, and therefore,
$$
f'(0)=-v\trsp ({\cal S}+A)^{-1} B ({\cal S}+A)^{-1} v\;,
$$
which is the quantity we seek to approximate.

Now recall that from Heuristic \ref{heuristic:regularizedInverse}, we gathered that
$$
v\trsp ({\cal S}+A)^{-1}v \simeq v\trsp(A+\gamma(A)\Sigma)^{-1} v\;.
$$
We might be tempted to look at this approximate equality as valid for any $A(t)$ and take the derivative with respect to $t$. Doing so, we would get, if $g(t)=v\trsp ({\cal S}+A(t))^{-1}v$,
\begin{align*}
g'(0)=-v\trsp ({\cal S}+A)^{-1} (B+\gamma(A(t))'(0)\Sigma)({\cal S}+A)^{-1} v\;.
\end{align*}
Now,
$$
\gamma(A(t))=\frac{1}{n}\sum_{i=1}^n \frac{R_i^2}{1+R_i^2 \alpha(A(t))}\;.
$$
Hence, if $h(t)=\gamma(A(t))$ and $k(t)=\alpha(A(t))=\frac{1}{n}\trace{\Sigma({\cal S}+A(t))^{-1}}$, we have
$$
h'(0)=-k'(0)\frac{1}{n}\sum_{i=1}^n \frac{R_i^4}{(1+R_i^2 \alpha(A))^2}\;.
$$
Now, $k'(t)=-\frac{1}{n}\trace{\Sigma({\cal S}+A(t))^{-1}B({\cal S}+A(t))^{-1}}$. Hence,
$$
-k'(0)=\frac{1}{n}\trace{\Sigma({\cal S}+A)^{-1}B({\cal S}+A)^{-1}}\;,
$$
and we conclude that
$$
h'(0)=\left[\frac{1}{n}\trace{\Sigma({\cal S}+A)^{-1}B({\cal S}+A)^{-1}}\right]\left[\frac{1}{n}\sum_{i=1}^n \frac{R_i^4}{(1+R_i^2 \alpha(A))^2}\right]=\xi(A,B)\;.
$$
The fact that $\xi(A,B)$ is asymptotically non-random comes from the same ideas as described in Heuristic \ref{heuristic:regularizedInverse}.
\end{proof}

In our applications, we will also need to understand quantities of the type $\muHat\trsp (\SigmaHat+ A)^{-1} \muHat$ (where $\SigmaHat={\cal S}-\muHat\muHat\trsp$) and $\muHat\trsp (\SigmaHat+ A)^{-1} v$. We naturally treat those cases below and refer the reader to that part of the paper for information about these forms. The main issue is that when dealing with $\SigmaHat$ and $\muHat$, a non-negligible interaction term between the two occurs (it is related to $\muHat\trsp ({\cal S}+A)^{-1}\muHat$) and one needs to be a bit careful to treat it.

\section{Results and proofs}
This section contains the main technical aspects of the paper. In subsection \ref{subsec:Ef-Stein}, we discuss a simple extension of the Efron-Stein inequality. The rest of this section is devoted to showing concentration and invariance of the forms we care about. The method of proof is systematic: we first show concentration (i.e control of the variance or higher moments), and then show that the mean value to which we can reduce the problem does not depend on ``details" of the distribution of the data through a Lindeberg-like argument.

\paragraph{Notations} Before we proceed, let us set some notations. We denote by $\opnorm{M}$ the operator norm (i.e largest singular value) of a matrix $M$. When dealing with several independent random variables $(X_1,\ldots,X_n)$, we use $\Expj{i}{}$ to denote expectation with respect to $X_i$ only. We often use the abbreviation psd for positive semi-definite.

\subsection{A simple extension of the Efron-Stein inequality}\label{subsec:Ef-Stein}
The strategy for our approach is to first show that the quadratic forms we care about, namely
$$
v\trsp ({\cal S}+ A)^{-1}v\;, A\succeq t \id_p\;,
$$
(and variants) are essentially deterministic asymptotically. Modern techniques can be adapted to then get (in simple cases compared to the generality level at which we will work) deterministic approximations of $v\trsp({\cal S}+ A)^{-1} v$ and we can then use those to actually compute the limit of the aforementioned quadratic form. But it is important to get a systematic way of showing that for a certain class of random matrices ${\cal S}$,
$$
v\trsp ({\cal S}+A)^{-1}v\simeq v\trsp \Exp{({\cal S}+ A)^{-1}}v\;.
$$

To do so, we propose to use (essentially) a martingale difference argument, which is not unknown in random matrix theory (\cite{bai99}, \cite{GirkoRandomDets90}, and several others), but whose role may not have been as emphasized as it perhaps should have. However, at the level of generality at which we are working, our proofs become easier if we quickly branch away from standard methods. The following lemma is essentially an $L^p$ variant of the Efron-Stein inequality (see \cite{EfronStein81}, Theorem 2, and also \cite{LugosiConcentration}, Theorem 9). It is surely known in martingale theory but we give a simple proof here for the convenience of the reader.

\begin{lemma}\label{lemma:GeneralizationEfronStein}
Suppose $W=h(X_1,\ldots,X_n)$, where the $X_i$'s are independent. We call ${\cal F}_j=\sigma(X_1,\ldots,X_j)$.  We also denote by $W_m$ a (measurable) function of $(X_1,\ldots,X_{m-1},X_{m+1},\ldots,X_n)$.

Then, we have, for a constant $c$ that depends only on $k$, and for $k\geq 2$,
\begin{equation}\label{eq:GeneralizationEfronStein}
\Exp{\left|W-\Exp{W}\right|^k}\leq c \left(\Exp{\left[\sum_{m=1}^n \Exp{(W-W_m)^2|{\cal F}_{m-1}}\right]^{k/2}}+\sum_{m=1}^n \Exp{|W-W_m|^k}\right)\;.
\end{equation}
\end{lemma}
Note that in the case $k=2$, we recover the Efron-Stein inequality
$$
\var{W}\leq \sum_{m=1}^n \Exp{(W-W_m)^2}\;,
$$
with a possibly worse constant.

In the applications we have in mind, through rank-1 update of inverses of matrices, we will easily get an approximation of $Z$ by a function that does not involve the $m$-th variable and these results will come in particularly handy.

\begin{proof}[Proof of Lemma \ref{lemma:GeneralizationEfronStein}]
We can clearly write $Z-\Exp{Z}$ as a sum of martingale differences: if
\begin{align*}
V_m&=\Exp{Z|{\cal F}_m}-\Exp{Z|{\cal F}_{m-1}},\\
Z-\Exp{Z}&=\sum_{m=1}^n V_m\;.
\end{align*}
Note also that if $Z_m$ is a (measurable) function of all the $X_i$'s except $X_m$,
$$
V_m=\Exp{Z-Z_m|{\cal F}_m}-\Exp{Z-Z_m|{\cal F}_{m-1}}\;,
$$
since $\Exp{Z_m|{\cal F}_m}=\Exp{Z_m|{\cal F}_{m-1}}$.

Now let us call $s(Z)=[\sum_{m=1}^n \Exp{V_m^2|{\cal F}_{m-1}}]^{1/2}$. Recall that Burkholder's inequality implies (see Equation 21.5 in \cite{BurkholderWaldLectures73}) that, if $\Phi$ is a non-decreasing function on $[0,\infty]$ with $\Phi(0)=0$ and $\Phi(2\lambda)\leq c_1 \Phi(\lambda)$, then
$$
\Exp{\Phi(Z)}\leq c \left(\Exp{\Phi(s(Z))}+\sum_{k=1}^n \Exp{\Phi(|V_m|)}\right)\;.
$$
As noted in \cite{BurkholderWaldLectures73}, $\Phi(x)=x^k$ satisfies the conditions needed for the inequality to hold.
Let us remind the reader that it is well known (see \cite{LugosiConcentration}, p.16) that
$$
V_m^2\leq \Exp{(Z-\Expj{m}{Z})^2|{\cal F}_m}\;,
$$
where $\Expj{m}{\cdots}$ is expectation with respect to $X_m$ only, i.e $\Expj{m}{Z}=\Exp{Z|X_1,\ldots,X_{m-1},X_{m+1},\ldots,X_n}$. Also, as noted for instance in \citet{LugosiConcentration},
$$
\Expj{m}{(Z-\Expj{m}{Z})^2}\leq \Expj{m}{(Z-Z_m)^2}\;,
$$
where $Z_m$ is any measurable function of $X_1,\ldots,X_{m-1},X_{m+1},\ldots,X_n$.  We note that
$$
\Exp{\cdot|{\cal F}_{m-1}}=\Exp{\Expj{m}{\cdot}|{\cal F}_{m-1}}\;.
$$
Therefore,
$$
\Exp{V_m^2|{\cal F}_{m-1}}\leq \Exp{[Z-\Expj{m}{Z}]^2|{\cal F}_{m-1}}\leq \Exp{\Expj{m}{[Z-\Expj{m}{Z}]^2}|{\cal F}_{m-1}}\leq \Exp{(Z-Z_m)^2|{\cal F}_{m-1}}\;,
$$
and we have
$$
s(Z)\leq \sqrt{\sum_{m=1}^n \Exp{(Z-Z_m)^2|{\cal F}_{m-1}}}\;.
$$

Hence, because $\Phi$ is non decreasing,
$$
\Exp{\Phi(s(Z))}\leq \Exp{\Phi\left[\sqrt{\sum_{m=1}^n \Exp{(Z-Z_m)^2|{\cal F}_{m-1}}}\right]}
$$

Now let us turn our attention to $\Exp{\Phi(|V_m|)}$, specifically when $\Phi(x)=x^k$. Since $V_m=\Exp{Z-Z_m|{\cal F}_m}-\Exp{Z-Z_m|{\cal F}_{m-1}}$,
$$
|V_m|^k\leq 2^{k-1} \left(\left|\Exp{Z-Z_m|{\cal F}_m}\right|^k+\left|\Exp{Z-Z_m|{\cal F}_{m-1}}\right|^k\right)\;.
$$
Also, when $k\geq 1$, $|x|^k$ is convex, so Jensen's inequality implies that
$$
\left|\Exp{Z-Z_m|{\cal F}_m}\right|^k\leq \Exp{|Z-Z_m|^k|{\cal F}_m}\;.
$$
Therefore,
$$
\Exp{|V_m|^k}\leq 2^k \Exp{|Z-Z_m|^k}
$$
Equation \eqref{eq:GeneralizationEfronStein} now follows easily.
\end{proof}

We note that if we were willing to make stronger assumptions on the data that the ones we will make, we could rely on other concentration inequalities to obtain for instance Gaussian concentration for some of the statistics we are interested in. However, since our study is a robustness study, we made the choice of making weaker assumptions and consequently to have possibly worse concentration inequalities - though of course this allows us to show that our first order results hold for a wider class of distributions.

\subsection{Setup of our study}\label{subsec:SetupOfStudy}
In all that follows we make the following assumptions, which we will casually call ``our usual assumptions".
\begin{itemize}
\item We assume that $p/n$ remains bounded away from 0 and $\infty$, i.e $p\sim n$.
\item the random variables $X_j$ and $Y_j$ which will appear below have the same covariance matrix, $\Sigma_j$, and same mean, 0.
\item $Y_j$'s are independent and so are $X_j$'s.
\item $Y_j$'s are independent of $X_j$'s
\item If $v$ is any fixed vector with norm 1, we have, for $k\geq 1$,
\begin{align}
\label{B-L}
\Exp{|X_i\trsp v|^k}\leq b_L(k;X_i)
\end{align}
\item If $M$ is any deterministic and positive semidefinite matrix with $\opnorm{M}\leq 1$,
\begin{align}
\label{B-Q}
\Exp{|X_j\trsp M X_j-\Exp{X_j\trsp M X_j}|^k}\leq b_{Q_2}(k;X_j)\;.
\end{align}
\item The matrix towards which we shrink, $A$, is such that $A\succeq t\id_p$.
\end{itemize}
Let us note that by Jensen's inequality, there is no loss in generality in assuming that $b_L(k,X_i) \leq \sqrt{b_L(2k;X_i)}$. We will assume this throughout this paper, as this will occasionally be needed
to merge certains bounds arising in our estimates, and thus to shorten our formulas.

Also we note that if $A \succeq t\id$ and $\Sigma_0\succeq 0$, for any $x\in \mathbb{R}^p$, we have
$$
x\trsp (A+\Sigma_0)^{-2} x\leq \frac{1}{t} x\trsp A^{-1} x\;,
$$
which is easily seen since $M \mapsto M^{-1}$ is monotone (and decreasing with respect to the Loewner's order), so $(A+\Sigma_0)^{-1}\preceq t^{-1} \id$; now multiplying on both sides by $(A+\Sigma_0)^{-1/2}$, the inequality (and its order) is preserved and we conclude that $(A+\Sigma_0)^{-2}\preceq t^{-1} (A+\Sigma_0)^{-1}\preceq t^{-1} A^{-1}$.

Finally, let us give some order of magnitude bounds. $b_L$ will generally be very easy to control, as it is a linear form in $X_i$. For instance, if $X_i\sim {\cal N}(0,\id_p)$, we have $X_i\trsp v\sim {\cal N}(0,\norm{v})$, so $b_L(X_i;k)$ is of order 1 for all (finite) $k$. When $X_i$ is ${\cal N}(0,\id_p)$, $X_i\trsp M X_i$ is a weighted $\chi^2$, since $X_i\trsp M X_i\equalInLaw \sum_{k=1}^p \xi_k^2 \lambda_k(M)$ where $\xi_k$ are ${\cal N}(0,1)$ and independent. Hence, we conclude that $b_{Q_2}(k;X_i)$ is of order at most $p^{k/2}$ in this case.\\
The informal bounds we will have in mind are therefore
\begin{align*}
b_{L}(k;X_i)&=\gO(1)\;,\\
\frac{b_{Q_2}(k;X_i)}{p^{k/2}}&=\gO(1)\,\left(=\frac{b_{Q_2}(k;X_i)}{n^{k/2}}\right)\;,
\end{align*}
where the last statement comes from the fact that $p\sim n$.\\
We further note that if $\Sigma$ is a covariance matrix,
\begin{align*}
b_{L}(k;\Sigma^{1/2} X_i)&\leq \opnorm{\Sigma}^{k/2} \, b_{L}(k;X_i)\;,\\
b_{Q_2}(k;\Sigma^{1/2}X_i)&\leq \opnorm{\Sigma}^{k} \, b_{Q_2}(k;X_i)\;.
\end{align*}

To bound $b_{Q_2}$ in certain situations, it will be simpler to work through an auxiliary quantity, $b_{Q_1}$. Let us define it as, if $M$ is any deterministic (psd) matrix with $\opnorm{M}\leq 1$,
$$
\Exp{|\sqrt{Y_j\trsp M Y_j}-\Exp{\sqrt{Y_j\trsp M Y_j}}|^k}\leq b_{Q_1}(k;Y_j)\;.
$$

\paragraph{Connection between $b_{Q_1}$ and $b_{Q_2}$.} $b_{Q_1}$ and $b_{Q_2}$ are of course very closely related. Also, in a concentration context, because $y\mapsto \sqrt{y\trsp M y}$ is Lipschitz with respect to Euclidian norm and convex, it is possible to derive $b_{Q_1}$ for many distributions for which it would be otherwise difficult. For instance Gaussian concentration immediately implies deviation bounds and hence bounds on $b_{Q_1}$ for e.g.\@ centered Gaussian copulas.

Let us now elaborate on the relationship between $b_{Q_1}$ and $b_{Q_2}$. Let us call $Q_M(Y)=Y\trsp M Y$, $q_M(Y)=\sqrt{Q_M(Y)}$, $\Delta_M(Y)=Q_M(Y)-\Exp{Q_M(Y)}$ and $\delta_M(Y)=\sqrt{Q_M(Y)}-\Exp{\sqrt{Q_M(Y)}}$, i.e $\delta_M(Y)=q_M(Y)-\Exp{q_M(Y)}$. Clearly,
\begin{align*}
\Delta_M(Y)&=(q^2_M(Y)-[\Exp{q_M(Y)}]^2)+[\Exp{q_M(Y)}]^2 -\Exp{Q_M(Y)}\\
&=\delta_M(Y)\left[\delta_M(Y)+2\Exp{q_M(Y)}\right] +[\Exp{q_M(Y)}]^2 -\Exp{Q_M(Y)}\\
&=\delta_M(Y)\left[\delta_M(Y)+2\Exp{q_M(Y)}\right]-\var{q_M(Y)}\;.
\end{align*}

Using convexity of $x\mapsto |x|^k$, we conclude that
\begin{align*}
\left|\Delta_M(Y)\right|^k &\leq 3^{k-1} \left[|\delta_M(Y)|^{2k}+2^k |\delta_M(Y)|^k[\Exp{q_M(Y)}]^k +[\var{q_M(y)}]^k\right]\\
&\leq 3^{k-1} \left[|\delta_M(Y)|^{2k}+2^k |\delta_M(Y)|^k[\Exp{Q_M(Y)}]^{k/2} +[\var{q_M(y)}]^k\right]
\end{align*}
Now note that $\Exp{Q_M(Y)}=\trace{M\Sigma}$ and that $\var{q_M(y)}=b_{Q_1}(2;Y)$. So after taking expectations, we have shown that
$$
b_{Q_2}(k;Y)\leq 3^{k-1} \left[b_{Q_1}(2k;Y)+2^k b_{Q_1}(k;Y)\left[\trace{M\Sigma}\right]^{k/2}+[b_{Q_1}(2;Y)]^k\right]\;.
$$

Also, it is instructive to have a sense of the parameters that impact these bounds and how they grow. In the case of normality distributed random variables, $Q_M(Y)$ is a weighted $\chi^2$ with $p$ degrees of freedom, the weights being the eigenvalues of $\Sigma^{1/2}M\Sigma^{1/2}$. In this case, we have $b_{Q_2}(2;Y)=\sup_{M:\opnorm{M}=1} 2 \trace{(\Sigma M)^2}$. When $\opnorm{M}=1$, it is easy to see that $\trace{(\Sigma M)^2}\leq \trace{\Sigma^2}$, since if $A\preceq B$, and both are positive semi-definite, then $\trace{A^2}\geq \trace{B^2}$. Hence,
$b_{Q_2}=2\trace{\Sigma^2}$.

At this point, one might be concerned about the fact that these quantities will be dependent on extreme eigenvalues of $\Sigma$. However, in some situations, we can mitigate this problem. For instance, in the case where we assume that the data are i.i.d with the same covariance $\Sigma$, it will sometime be possible to work with $Y$ having covariance $\id$, by simply replacing the shrinkage factor $A$ by $\Sigma^{-1/2}A \Sigma^{-1/2}$, and the vector $x$ at which we evaluate the shrunken matrix by $\Sigma^{-1/2} x$. This is the case for instance when considering $x\trsp (\SigmaHat+A)^{-1}x$.

\subsubsection{Meaningfulness of the assumptions and applicability}\label{subsubsec:ApplicabilityOfAssumptions}
It is of course important to check that the assumptions we make can be applied to a wide variety of situations. It is therefore instructive to give examples at this point. Here are two.
\begin{itemize}
\item Suppose that $X_i$ satisfies $P(|X_i\trsp v|>t)\leq C \exp(-c t^b)$, and $X_i$ has mean 0. Then
$$
b_L(k;X_i)\leq \frac{C}{c^{k/b}} \frac{k}{b} \Gamma\left(\frac{k}{b}\right)\;.
$$
\item Suppose that $X_i$ satisfies $P(|X_i\trsp v|>t)\leq C t^{-b}$. Then if $b>(k+1)$,
$$
b_L(k;X_i)\leq C\left(1+\frac{1}{b-(k+1)}\right)\;.
$$
\end{itemize}

We note that the condition on the $b_L(k;X_i)$'s is rather minimal: all we need is some concentration of linear forms in $X_i$. 

The exponential deviation inequality might look like a strong assumption. However, it is satisfied by many distributions, with quite non-linear structures  which would be difficult to analyze if one did not resort to concentration of measure arguments (see \cite{ledoux2001} for a very thorough reference, and see for instance \cite{nekCorrEllipD} for spelled-out examples).  For the convenience of the reader, here are some examples taken from this last reference (justifications can be found there):
\begin{itemize}
\item Gaussian random variables, with $\opnorm{\Sigma}$ bounded for instance. (Note that this can be relaxed considerably.)
\item Vectors of the type $\sqrt{p} r$ where $r$ is uniformly
distributed on the unit ($\ell_2$-) sphere is dimension $p$.
\item Vectors $X=\Gamma \sqrt{p} r$, with $r$ uniformly distributed on the unit ($\ell_2$-)sphere in $\mathbb{R}^p$
and with $\Gamma\Gamma'=\Sigma$ with e.g.\@ $\opnorm{\Sigma}$ bounded.
\item Vectors of the type $X=p^{1/b} r$, $1\leq b \leq 2$, where $r$ is ``uniformly" sampled in the 1-$\ell^b$ ball or sphere in $\mathbb{R}^p$. (See
\cite{ledoux2001}, Theorem 4.21, which refers to
\cite{SchechtmanZinn00} as the source of the theorem and explains the details of the sampling.)
\item Vectors $X$ with log-concave density of the type $\me^{-U(x)}$, with the Hessian of $U$
satisfying, for all $x$, $\textrm{Hess}(U)\geq c\id_p$ (see
\cite{ledoux2001}, Theorem 2.7.) For simplicity, though it may not be needed, one can assume that $\opnorm{\Sigma}$ remains bounded.
\item Vectors ($X$) distributed according to a (centered) Gaussian copula, with corresponding correlation matrix, $\Sigma$, having $\opnorm{\Sigma}$ bounded. In other words, if $Z\sim {\cal N}(0,R)$, $X=\Phi(Z)-1/2$, where $\Phi$ is the cdf of the standard Gaussian random variables.
\item Vectors $X=\Sigma^{1/2} Y$, where $Y$ has i.i.d bounded entries . See Corollary 4.10 in
\cite{ledoux2001} for the concentration part.
Here we crucially need the fact that the concentration of measure results we rely on are valid for convex 1-Lipschitz function (and we do not need them for all Lipschitz functions).
\item More ``exotic" examples
involving vectors sampled uniformly from certain Riemannian submanifolds of $\mathbb{R}^p$. We refer to \cite{ledoux2001} Theorems 2.4 and 3.1 for the concentration aspects for these questions.
\end{itemize}

Bounding of $b_{Q_2}$ can either be done directly or using the connection (and bound) between $b_{Q_2}$ and $b_{Q_1}$ we just made explicit. If $X_i$ satisfies a concentration inequality for convex Lipschitz functions, then bounding $b_{Q_1}$ is rather simple and this gives us a bound on $b_{Q_2}$. We now work out the details of this problem. The analysis is standard and follows along the lines of work done in e.g.\@ \cite{ledoux2001}, Chapter 1.

\paragraph{An important example: case of concentrated random variables} As a matter of fact, suppose that $X_i$ is such that for any convex and 1-Lipschitz function $f$, if $X\equalInLaw X_i$,
$$
P(|f(X)-\Exp{f(X)}|>t)\leq C \exp(-c t^b)\; \text{ or } P(|f(X)-\median{f(X)}|>t)\leq C \exp(-c t^b)
$$
Since $f_v(X)=X\trsp v$ is trivially convex and $\norm{v}$-Lipschitz, we see that if the concentration inequality is around the mean, we immediately have
$$
b_L(k;X_i)\leq \frac{C}{c^{k/b}} \frac{k}{b} \Gamma\left(\frac{k}{b}\right)\;.
$$
If we ``only" have a concentration bound around the median, then we can simply use
$$
\Exp{|X\trsp v|^k}\leq 2^{k-1}\left(\Exp{|X\trsp v-\median{X\trsp v}|^k}+|\median{X\trsp v}|^k\right)\;.
$$
The concentration inequality gives us control of the first term, while $|\median{X\trsp v}|=|\median{X\trsp v}-\Exp{X\trsp v}|$ which is also controlled (see Proposition 1.9 in \cite{ledoux2001}) or simply
$$
|\median{X\trsp v}-\Exp{X\trsp v}|\leq \Exp{|X\trsp v-\median{X\trsp v}|}=\int_{0}^{\infty} P(|X\trsp v-\median{X\trsp v}|>t) dt\leq C \int_0^{\infty} \exp(-ct^b dt)\;.
$$
This is of course nothing else than $C\Gamma(1/b)/(bc^{1/b})$, and so we have a uniform bound.

Similarly, when $M$ is a positive definite matrix with $\opnorm{M}\leq 1$, $\sqrt{X_i\trsp M X_i}$ is a convex 1-Lipschitz function (with respect to Euclidian norm for $X_i$). Using the fact that for a non-negative random variable $Z$, $\Exp{Z^k}=\int_0^{\infty} k x^{k-1} P(Z\geq x) dx$, we see that, if our concentration result is around the mean,
$$
b_{Q_1}(X_i;k)=\Exp{\left|\sqrt{X_i\trsp M X_i}-\Exp{\sqrt{X_i\trsp M X_i}}\right|^k}\leq C \int_0^{\infty} k x^{k-1} \exp(-cx^b) dx=\frac{C}{c^{k/b}} \frac{k}{b} \Gamma\left(\frac{k}{b}\right)\;.
$$
Hence, when $X_i$ satisfy a dimension-free concentration inequality, $b_{Q_1}(k;X_i)$ remains bounded uniformly in $p$ and $n$. Therefore, when $\trace{\Sigma}/n$ remains bounded as $n$ grows, so does $b_{Q_2}(k;X_i)/n^{k/2}$, thanks to the relationship between $b_{Q_1}$ and $b_{Q_2}$ we have highlighted above.

The conclusion of this short discussion is that random variables satisfying a dimension free concentration inequality and having covariance such that $\{\trace{\Sigma_i}/n\}_{i=1}^n$ remains uniformly bounded in $n$ and $p$ will have $b_{Q_2}(2;X_i)/n$ and $b_L(4;X_i)$ uniformly bounded (in $n$). Because we will express later our various bounds in terms of these quantities, this observation is very important from the point of view of the applicability of our results.

An important distribution in practice (in particular in financial applications) is the log-normal distribution. Getting bounds for $b_L$ and $b_{Q_2}$ here requires work which we now perform.

\subsubsection{The case of the log-normal distribution}
\def\ee{\mathbf{E}}
\def\pp{\mathbf{P}}

\newcommand{\diag}{\textrm{diag}}

\def\ee{\mathbf{E}}
\def\pp{\mathbf{P}}

Let $Z = (Z_1,\hdots,Z_p)$ be a random vector with a normal distribution
with parameters $\tilde\mu = (\tilde\mu_i)$ and $\tilde\Sigma = (\tilde\sigma_{ij})$.
Then the random vector $Y := (Y_1,\hdots,Y_p)$ with $Y_i := \exp(Z_i)$, $i=1,\hdots,p$,
is said to have a log-normal distribution with parameters $\tilde\mu$ and $\tilde\Sigma$
(see e.g.\@ Mardia, Kent and Bibby (1979), Chapter 2.6).
Note that the moments of the log-normal distribution are all finite, and
can be obtained from the moment generating function of the normal distribution.
Indeed, for any $t = (t_1,\hdots,t_p) \in \mathbb{N}_{0}^p$, we have
\begin{align}
\label{eq:lognormal-moment}
\ee(Y_1^{t_1} \dots Y_p^{t_p}) = \ee(\exp(t\trsp Z)) = \exp(t\trsp \tilde\mu + \tfrac12 t\trsp \tilde\Sigma t) \,.
\end{align}
Set $\tilde\mu_* := \norm{\tilde\mu}_2$ and $\tilde\sigma_*^2 := \opnorm{\tilde\Sigma}$.
Then, for any $t = (t_1,\hdots,t_p) \in \mathbb{N}_{0}^p$, we~have the estimate
\begin{align}
\label{eq:lognormal-momentbound}
\ee(Y_1^{t_1} \dots Y_p^{t_p}) \leq \exp(\|t\|_2 \tilde\mu_* + \tfrac12 \|t\|_2^2 \tilde\sigma_*^2) \,.
\end{align}
Put $X := Y - \ee(Y)$ (where the expectation is taken componentwise, of course).
In this section we will derive bounds for the constants $b_L(2r,X)$ and $b_{Q_2}(2,X)$
associated with the (centered) log-normal distribution.

In the sequel we always assume that $Z = \tilde\mu + \tilde\Sigma^{1/2} \bar{Z}$, where
$\bar{Z}$ is a $p$-dimensional Gaussian random vector with zero~mean and identity covariance.
Our derivation will be based on the following result for the Gaussian distribution
\cite[Chapter~2]{Pisier86ProbaMethodsGeomBanach}:
If $F$ is a continuously differentiable function and $\nabla F$ is the gradient of $F$
(which we always regard as a column vector), then, for~any $r \geq 1$,
$$
     \ee|F(\bar{Z}) - \ee(F(\bar{Z}))|^r
% \leq (\tfrac{\pi}{2})^r \, \ee|\nabla F(\bar{Z}) \bar{Z}'|^r
\leq K_r (\tfrac{\pi}{2})^r \, \ee\|\nabla F(\bar{Z})\|_2^{r} \,,
$$
where $K_r$ is the $r$th moment of the standard Gaussian distribution.

For any $z = (z_i) \in \mathbb{R}^p$,
let $\exp(z) := (\exp(z_i)) \in \mathbb{R}^p$ (by slight abuse of notation),
and note that this vector-valued version of the exponential function
is continuously differentiable and its Jacobian matrix $D(z)$
is diagonal with the elements $\exp(z_i)$ on the main diagonal.
With this notation, $Y = \exp(Z) = \exp(\tilde\mu + \tilde\Sigma^{1/2} \bar{Z})$,
and we get, for any $r \geq 1$,
$$
\ee|F(Y) - \ee(F(Y))|^{r} \leq K_r (\tfrac{\pi}{2})^r \, \ee\|\nabla F(Y)' D(Z) \tilde\Sigma^{1/2}\|^r \,.
$$
We now specialize this result to linear and quadratic forms.

\medskip

\textbf{Linear Forms.}
Consider the linear form $F(y) := v'y$,
where $v = (v_i)$ is a~deterministic vector with Euclidean norm $1$.
Then $\nabla F(y) = v$, and we~get, for any integer $r \geq 1$,
$$
\ee|F(Y) - \ee(F(Y))|^{2r} \leq K_{2r} (\tfrac{\pi}{2})^{2r} \opnorm{\tilde\Sigma}^{r} \, \ee( v' D(Z) D(Z) v )^{r} \,.
$$
Now, using the special structure of the diagonal matrix $D(Z)$
and the bound \eqref{eq:lognormal-momentbound}, we find that
\begin{multline*}
  \ee\left( v' D(Z) D(Z) v \right)^{r}
= \sum_{i_1} \dots \sum_{i_r} v_{i_1}^2 \dots v_{i_r}^2 \ee\left( Y^2_{i_1} \dots Y^2_{i_r} \right) \\
\leq \exp(2r \tilde\mu_* + \tfrac12 (2r)^2 \tilde\sigma_*^2) \left( \sum_i v_i^2 \right)^r
= \exp(2r \tilde\mu_* + \tfrac12 (2r)^2 \tilde\sigma_*^2) \,.
\end{multline*}
Combining these estimates, we conclude that
$$
\ee|F(Y) - \ee(F(Y))|^{2r} \leq K_{2r} (\tfrac{\pi}{2})^{2r} \tilde\sigma_*^{2r} \exp(2r \tilde\mu_* + \tfrac12 (2r \tilde\sigma_*)^2) \,.
$$
Since $v'X - \ee(v'X) = v'Y - \ee(v'Y)$, it follows that
$$
b_L(2r,X) \leq K_{2r} (\tfrac{\pi}{2})^{2r} \tilde\sigma_*^{2r} \exp(2r \tilde\mu_* + \tfrac12 (2r \tilde\sigma_*)^2) \,.
$$
In particular, if $\tilde\mu_*$ and $\tilde\sigma_*^2$ are uniformly bounded,
this is of the order $\gO(1)$.

\medskip

\textbf{Quadratic Forms.}
Consider the quadratic form $F(y) := y'My$, where $M$
is a deterministic \emph{symmetric} matrix with operator norm $1$.
Then $\nabla F(y) = 2My$, and we~get, for any integer $r \geq 1$,
$$
\ee|F(Y) - \ee(F(Y))|^{2r} \leq K_{2r} \pi^{2r} \opnorm{\tilde\Sigma}^{r} \, \ee( Y' M D(Z) D(Z) M Y )^{r} \,.
$$
Observing that $Y = D(Z) 1$, where $1$ is the vector consisting of $1$'s,
and setting $N := D(Z) M D(Z)$, it~follows that
$$
\ee|F(Y) - \ee(F(Y))|^{2r} \leq K_{2r} \pi^{2r} \opnorm{\tilde\Sigma}^{r} \, \ee( 1' N^{2r} 1 )\,.
$$
Because most of our bounds depend on $b_{Q_2}(2;X_i)$ only, let us now consider the case $r=1$.
Note that
$$
N_{i,j}=M_{i,j}e^{Z_i+Z_j}\;.
$$
So
$$
N^2_{k,l}=\sum_{j} M_{k,j}M_{j,l}e^{Z_k+2Z_j+Z_l}\;.
$$
Now $Z_k+2Z_j+Z_l=(e_k+2e_j+e_l)\trsp Z$,
so, by \eqref{eq:lognormal-moment},
\begin{align*}
\Exp{\exp^{Z_k+2Z_j+Z_l}}&=\exp((2e_j+e_k+e_l)\trsp \tilde\mu)\exp(\tfrac{1}{2}(2e_j+e_k+e_l)\trsp \tilde\Sigma(2e_j+e_k+e_l))\\
&=\exp(2\tilde\mu_j+\tilde\mu_k+\tilde\mu_l)\exp(2\tilde\Sigma_{j,j}+\tilde\Sigma_{k,k}/2+\tilde\Sigma_{l,l}/2+2\tilde\Sigma_{j,k}+2\tilde\Sigma_{j,l}+\tilde\Sigma_{k,l}) \,.
\end{align*}
Therefore,
\begin{multline*}
\Exp{N^2_{k,l}}=e^{\tilde\Sigma_{k,l}}e^{-\tilde\Sigma_{k,k}/2}e^{-\tilde\Sigma_{l,l}/2} \\
\times \sum_{j}\left(M_{k,j}\exp(\tilde\mu_j+\tilde\mu_k+\tilde\Sigma_{j,j}+\tilde\Sigma_{k,k}+2\tilde\Sigma_{j,k})\right)\left(M_{j,l}\exp(\tilde\mu_j+\tilde\mu_l+\tilde\Sigma_{j,j}+\tilde\Sigma_{l,l}+2\tilde\Sigma_{j,l})\right)\;.
\end{multline*}
Let us now write $A \circ B$ for the Hadamard product of two matrices $A$ and $B$
and $e^{\circ A}$ for the Hadamard exponential of a matrix $A$,
i.e.\@ the matrix with entries $e^{A_{i,j}}$.
Let us call $\Delta$ and $\tilde\Delta$ the diagonal matrices with entries
$e^{\tilde\Sigma_{j,j}}$ and $e^{\tilde\mu_j+\tilde\Sigma_{j,j}}$,
respectively.
Note that $M_{k,j}\exp(\tilde\mu_j+\tilde\mu_k+\tilde\Sigma_{j,j}+\tilde\Sigma_{k,k}+2\tilde\Sigma_{j,k})$
is the $k,j$ entry of the matrix
$
\tilde\Delta (M\circ e^{\circ 2\tilde\Sigma}) \tilde\Delta\;.
$
So
$$
\Exp{N^2}=\left[\Delta^{-1/2}e^{\circ \tilde\Sigma}\Delta^{-1/2}\right]\circ (\tilde\Delta (M \circ e^{\circ 2\tilde\Sigma}) \tilde\Delta)^2\;.
$$
Now recall that for any vector $x$, if $D_x$ is the diagonal matrix with $x$ on its diagonal, (see \cite{hornjohnson94}, Lemma 5.1.5),
$$
x\trsp (A\circ B) x=\trace{D_x A D_x B\trsp}\;.
$$
Hence,
\begin{align*}
  1\trsp \Exp{N^2} 1
&=\trace{\id_n \left[\Delta^{-1/2}e^{\circ \tilde\Sigma}\Delta^{-1/2}\right] \id_n (\tilde\Delta (M \circ e^{\circ 2\tilde\Sigma}) \tilde\Delta)^2} \\
&=\trace{\left[\Delta^{-1/2}e^{\circ \tilde\Sigma}\Delta^{-1/2}\right] (\tilde\Delta (M \circ e^{\circ 2\tilde\Sigma}) \tilde\Delta)^2}\;.
\end{align*}
Now the Hadamard exponential of a psd matrix is psd (see \cite{hornjohnson94}, p. 450). Recall also that for $A$ and $B$ psd matrices,
$A\circ B$ is psd (\cite{hornjohnson94}, p. 309) and
$$
\opnorm{A\circ B}= \lambda_{\max}(A\circ B)\leq \max_i a_{ii} \lambda_{\max}(B)\;,
$$
by theorem 5.3.4 in \cite{hornjohnson94}. Therefore, since $M$ is psd and $\opnorm{M}\leq 1$,
$$
\opnorm{M\circ e^{\circ 2\tilde\Sigma}}\leq \exp(2\max_j \tilde\Sigma_{j,j})\;.
$$
So
$$
\opnorm{\Delta (M\circ e^{\circ 2\tilde\Sigma}) \Delta}\leq \exp(2\max_j \tilde\mu_j + 4\max_j \tilde\Sigma_{j,j})\;.
$$
So we have, using the fact that when $A$ and $B$ are psd, $\trace{AB}\leq \lambda_{max}(B)\trace{A}$, because $A^{1/2}B A^{1/2}\preceq \lambda_{\max}(B) A$,
$$
\trace{\left[\Delta^{-1/2}e^{\circ \tilde\Sigma}\Delta^{-1/2}\right] (\Delta (M\circ e^{\circ V}\circ e^{\circ 2\tilde\Sigma}) \Delta)^2}
\leq p \exp(4 \max_j \tilde\mu_j + 8 \max_j \tilde\Sigma_{j,j}) \;.
$$
Combining the preceding estimates, we~conclude that
$$
\ee|F(Y) - \ee(F(Y))|^{2} \leq K_{2} \pi^{2} \tilde\sigma_*^{2} p \exp(4 \tilde\mu_* + 8 \tilde\sigma_*^2) \;.
$$
Now set $v := 2M\ee(Y)$ and note that $\|v\|_2^2 \leq 4\ee\| Y \|_2^2 \leq 4p \exp(2\tilde\mu_* + 2\tilde\sigma_*^2)$.
Since $X'MX - \ee(X'MX) = (Y'MY - \ee(Y'MY)) - (v'Y - \ee(v'Y))$,
it follows that
$$
b_{Q_2}(2,X) \leq K_{2} \, 4\pi^2 \tilde\sigma_*^{2} p \exp(4 \tilde\mu_* + 8 \tilde\sigma_*^2) \,.
$$
In particular, if $\tilde\mu_*$ and $\tilde\sigma_*^2$ are uniformly bounded,
this is of the order $\gO(p)$.

\subsection{On quadratic forms involving $\bm{(X\trsp D^2 X/n+A)^{-1}}$}\label{subsec:quadFormsInverse}

\subsubsection{On forms of the type $x\trsp \left(X\trsp D^2 X/n+A\right)^{-1}x$}
Throughout the proofs, we will make heavy use of the following notation: call, consistently with the notations used above,
$$
{\cal S}=\frac{1}{n}\sum_{i=1}^n R_i^2 X_iX_i\trsp\triangleq X\trsp D^2 X/n\;,
$$
where $D$ is a diagonal matrix with positive entries containing the $R_i$'s (on its $d_{i,i}$ entry) and $X$ is the $n\times p$ matrix whose $i$-th line is $X_i\trsp$.
We will use the notations
\begin{align*}
M&\triangleq {\cal S}+A\;, A\succeq t \id_p\;,\\
f(X)&\triangleq x\trsp M^{-1} x\;.
\end{align*}
To alleviate the notation, we do not show explicitly in the notations the dependence of $M$ on $A$ (and therefore, implicitly on $t$). However, our bounds will involve them, to allow us to show the impact of having a small $t$ (a small regularization), and also to show clearly how $x\trsp A^{-1}x$ affects our bounds. Similarly, because we are mostly interested in the impact of the randomness in $X_i$'s on the form $f(X)$ we keep track only of this random variable.

$\bullet$ \textbf{Concentration aspects}

\begin{theorem}\label{thm:FirstFormIsDeterministic}
Suppose $X_1,\ldots,X_n\in \mathbb{R}^p$ are independent. Suppose further that $\Exp{X_i}=0$ and, if $v$ is such that $\norm{v}=1$, $\Exp{|X_i\trsp v|^k}\leq b_L(k;X_i)$, where $b_L(k;X_i)$ is a deterministic function depending only on the distribution of $X_i$ and $k$. Call
$$
{\cal S}=\frac{1}{n}\sum_{i=1}^n R_i^2 X_iX_i\trsp\;,
$$
where $R_i$ are deterministic.

Call $M = {\cal S}+A$, and assume that $A$ is positive definite, with $A\succeq t\id_p$. We also call $f(X)=x\trsp M^{-1} x$. Then, if $\norm{x}=1$,
$$
\Exp{|f(X)-\Exp{f(X)}|^k}\leq \frac{c_k}{t^{2k}}\left[\left(\sum_{i=1}^n \left[\frac{R_i^4}{n^2} b_L(4;X_i)\wedge t^2 \right]\right)^{k/2}+\left(\sum_{i=1}^n \left[\frac{R_i^{2k}}{n^{k}} b_L(2k;X_i) \wedge t^{k}\right]\right)\right]\;.
$$
\end{theorem}
We note that the bound given in the proof below shows the actual dependence of this upper bound on $x\trsp A^{-1} x$. Also, it would be easy to handle the situation where $R_i$'s are random but independent on $X_i$'s.
\begin{proof}
We naturally apply Lemma \ref{lemma:GeneralizationEfronStein} to tackle this problem. Let us call $M_i=M-\frac{1}{n}R_i^2 X_iX_i\trsp$.

Using the classic rank-1 update formula,
$$
M^{-1}=M_i^{-1}-\frac{R_i^2}{n}\frac{M_i^{-1}X_iX_i\trsp M_i^{-1}}{1+R_i^2 X_i\trsp M_i^{-1} X_i/n}\;.
$$
Therefore, if $Z=x\trsp M^{-1} x$ and $Z_i=x\trsp M_i^{-1} x$,
$$
Z-Z_i=-\frac{R_i^2}{n}\frac{(x\trsp M_i^{-1}X_i)^2}{1+R_i^2 X_i\trsp M_i^{-1} X_i/n}\;.
$$
Hence,
$$
|Z-Z_i|\leq \left[\frac{R_i^2}{n}(x\trsp M_i^{-1}X_i)^2\right]\wedge (x\trsp M_i^{-1} x)\;,
$$
because $M_i$ is positive definite and $(x\trsp M_i^{-1}X_i)^2\leq (x\trsp M_i^{-1}x)(X_i\trsp M_i^{-1}X_i)$ by the Cauchy-Schwarz inequality.

Let us call $\Expj{i}{}$ expectation with respect to $X_i$ only.
Clearly, using our assumption on $X_i$, we have
$$
\Expj{i}{|X_i\trsp M_i^{-1} x|^k}\leq \norm{M_i^{-1} x}^k b_L(k;X_i)\;.
$$
Hence,
$$
\Expj{i}{|Z-Z_i|^k}\leq \left(\frac{R_i^2}{n}\right)^k (x\trsp M_i^{-2}x)^k b_L(2k;X_i)\wedge (x\trsp M_i^{-1} x)^k\;.
$$
Now, $M_i\succeq A\succeq t \id_p$, so $(x\trsp M_i^{-2}x)\leq t^{-1}x\trsp A^{-1} x$ and $(x\trsp M_i^{-1}x)\leq x\trsp A^{-1} x $, using the fact that $B\longmapsto -B^{-1}$ is operator monotone on Hermitian matrices (\cite{bhatia97}, p. 114).
So we finally have the bounds
\begin{align*}
\Exp{|Z-Z_i|^2|{\cal F}_{i-1}}&\leq \left(\frac{R_i^2}{n}\right)^2 t^{-2}(x\trsp A^{-1} x)^2 b_L(4;X_i)\wedge (x\trsp A^{-1} x)^2\;,\\
\Exp{|Z-Z_i|^k}&\leq \left(\frac{R_i^2}{n}\right)^{k} t^{-k} (x\trsp A^{-1} x)^k b_L(2k;X_i)\wedge (x\trsp A^{-1} x)^k\;.
\end{align*}
Now recalling Equation \eqref{eq:GeneralizationEfronStein}, we have
\begin{align*}
\Exp{|Z-\Exp{Z}|^k}&\leq c_k\left\{ \left[\sum_{i=1}^n \left(\frac{R_i^2}{n}\right)^2 \frac{(x\trsp A^{-1} x)^2}{t^2} b_L(4;X_i)\wedge (x\trsp A^{-1} x)^2\right]^{k/2}\right.
\\&+\left.\sum_{i=1}^n \left[\left(\frac{R_i^2}{n}\right)^{k} \frac{(x\trsp A^{-1} x)^k}{t^{k}} b_L(2k;X_i)\wedge (x\trsp A^{-1} x)^k\right]\right\}\;.
\end{align*}
Using the fact that $A \succeq t\id_p$ and $\norm{x}=1$, we have $x\trsp A^{-1} x\leq t^{-1}$, and this gives the result announced in the theorem.
\end{proof}

$\bullet$ \textbf{Lindeberg approach and why the limit does not depend on the distribution of $\bm{X_i}$} We are now interested in showing that for a broad class of distribution for $X_i$, the limit of
$$
x\trsp (X\trsp D^2X/n+A)^{-1} x
$$
or more precisely
$$
\Exp{x\trsp (X\trsp D^2X/n+A)^{-1} x}
$$
does not depend on the distribution of $X_i$. We have already seen that we can control the fluctuation of $x\trsp (X\trsp D^2X/n+A)^{-1} x$ around its mean for a broad class of distributions, so all we need to show is that they all have the same means.

We have the following theorem.
\begin{theorem}\label{thm:FirstLimitIsSameForEveryone}
Suppose $X_i$ are i.i.d and $Y_i$ are i.i.d and follow the assumptions mentioned above (at the beginning of Subsection \ref{subsec:SetupOfStudy}). Assume that $D$ is a deterministic diagonal matrix, whose diagonal entries are positive and denoted by $R_j$. We~assume that $A$ is a positive definite matrix with $A\succeq t\id_p$, for some $t>0$.

Then, for any given vector $x$, if $f(X)=x\trsp (X\trsp D^2 X/n+A)^{-1} x\;,$

\begin{align}
\left|\Exp{f(X)-f(Y)}\right|&\leq  \sum_{j=1}^n U_j(X_j)+U_j(Y_j) \; \; \text{where } \notag \\
U_j(X_j)&\leq \frac{R_j^4}{n^{3/2}} \left(\frac{x\trsp A^{-1}x}{t^2} \sqrt{b_L(4;X_j)}\sqrt{b_{Q_2}(2;X_j)/n}\right) \wedge \frac{R_j^2}{n} \frac{x\trsp A^{-1}x}{t} b_L(2;X_j) \;.
\label{eq:boundDifffXAndfY}
\end{align}

\end{theorem}
Let us discuss briefly this result. We see that assuming $\max_j \opnorm{\Sigma_j}$ is bounded, and making assumptions on $b_L$ and $b_{Q_2}$ that match the Gaussian situation (i.e $b_L$ and $b_{Q_2}/n$ uniformly bounded in $n$), the upper bound on the error is of the form (up to constants)
$$
\sum_{i=1}^n \frac{R_i^4}{n^{3/2}}\wedge \frac{R_i^2}{n}\;.
$$

If the $R_i$'s are given
by square-integrable i.i.d.\@ random variables (the same for each $n$),
we have
$$
\Exp{\frac{R_i^4}{n^{3/2}}\wedge \frac{R_i^2}{n}} = o(n^{-1}) \;.
$$

Hence, when this is the case, and the assumptions of our discussion are met, we have
$$
\Exp{f(X)-f(Y)}\tendsto 0\;,
$$
where $\Exp{\cdot}$ is here expectations with respect to all sources of random variables (i.e $R_i$'s, $X_i$'s and $Y_i$'s.)
Simple computations also show that if $R_i$'s are random and have $2+\eps$ moments, with $\eps\leq 2$,
$$
\Exp{\frac{R_i^4}{n^{3/2}}\wedge \frac{R_i^2}{n}}\leq \frac{K}{n^{1+\eps/4}}\;.
$$
Hence, when this is the case, we have
$$
\Exp{f(X)-f(Y)}\tendsto 0
$$
provided that $b_L$ and $b_{Q_2}$ do not grow too fast to infinity. If we are in a situation where $Y_j=\Sigma_j^{1/2} Y_0$ where $Y_0$ is such that $b_L(k;Y_0)=\gO(1)$ and $b_{Q_2}(k;Y_0)=\gO(1)$, the theorem can handle the case where $\opnorm{\Sigma_j} \ll n^{\eps/8}$ (which allows $\opnorm{\Sigma_j}$ go to infinity). Note that because we are interested in covariance matrices, we will always require $R_i$ to have at least 2 moments and so this theorem essentially covers all the cases of interests to us.

The meaning of the theorem is therefore that under these assumptions, i.e when the upper bound goes to 0 for $Y_j$ and say $X_j$ are gaussians, all we have to do is simply to understand $\Exp{f(X)}$ when $X$ is Gaussian. For this task, we can use many of the nice and well-known properties of the Gaussian distribution (which include strong concentration properties).

\begin{proof}

It is clear that $\Exp{f(X)}$ exists since $A$ is positive definite. We employ the Lindeberg approach (\cite{Lindeberg22}, and e.g.\@ \cite{StroockProbaBook}) to show that the limit does not depend on the distribution of $X_i$ (note that this technique has been used in other random matrix theoretic questions, e.g.\@ \cite{ChatterjeeSimpleInvariance05}, though the results of this paper do not seem directly applicable; note also that here all our expansions are exact whereas often in the Lindeberg method Taylor approximation arguments are used. That is why we choose to present such an approach.). Let us call
$$
Z_j=\left(Y_1,Y_2,\ldots,Y_{j-1},X_j,\ldots,X_n\right)\;,
$$
with the convention that $Z_1=(X_1,\ldots,X_n)$ and $Z_{n+1}=(Y_1,\ldots,Y_n)$. Clearly,
$$
\Exp{f(X)-f(Y)}=\sum_{j=1}^n \Exp{f(Z_j)-f(Z_{j+1})}\;.
$$
Now let us call $M_j=A+Z_j\trsp D^2 Z_j/n-R_j^2 X_j X_j\trsp/n$. Note that
$$
f(Z_j)=x\trsp (M_j+R_j^2 X_j X_j\trsp)^{-1} x\;, \; f(Z_{j+1})=x\trsp (M_j+R_j^2 Y_j Y_j\trsp)^{-1} x\;,
$$
and $M_j$ is independent of both $X_j$ and $Y_j$.
Therefore, using the fact that $(M+uu\trsp)^{-1}=M^{-1}-M^{-1}uu\trsp M^{-1}/(1+u\trsp M^{-1} u)$ (see \cite{hj}, Chapter 0), we have
$$
f(Z_j)-f(Z_{j+1})=\frac{R_j^2}{n}\left[\frac{(x\trsp M_j^{-1}Y_j)^2}{1+\frac{R_j^2}{n}Y_j\trsp M_j^{-1} Y_j}-
\frac{(x\trsp M_j^{-1}X_j)^2}{1+\frac{R_j^2}{n}X_j\trsp M_j^{-1} X_j}\right]\;.
$$
Since $Y_j$ and $X_j$ have the same covariance matrix, $\Sigma_j$, if we call $d_j=\trace{M_j^{-1}\Sigma_j}$, and $q_j(Y_j)=Y_j\trsp M_j^{-1} Y_j$, we see that
\begin{align}
\label{eq:fromQtoD}
\frac{1}{1+R_j^2 q_j(Y_j)/n}=\frac{1}{1+R_j^2d_j/n}+\tfrac{1}{n}R_j^2\delta_j(Y_j)\;,
\end{align}
where
\begin{align}
\label{eq:defDelta}
\delta_j(Y_j) := \frac{\left(d_j-q_j(Y_j)\right)}{(1+R_j^2 q_j(Y_j)/n)(1+R_j^2 d_j/n)} \,.
\end{align}
Hence, we see that
$$
\frac{(x\trsp M_j^{-1}Y_j)^2}{1+\frac{R_j^2}{n}q_j(Y_j)}=\frac{(x\trsp M_j^{-1}Y_j)^2}{1+\frac{R_j^2}{n}d_j}+
\tfrac{1}{n}R_j^2(x\trsp M_j^{-1}Y_j)^2\delta_j(Y_j)\;.
$$
Therefore,
\begin{align*}
f(Z_j)-f(Z_{j+1})&=\frac{R_j^2}{n}\left[\frac{(x\trsp M_j^{-1}Y_j)^2}{1+\frac{R_j^2}{n}d_j}-\frac{(x\trsp M_j^{-1}X_j)^2}{1+\frac{R_j^2}{n}d_j}\right]\\
&+\frac{R_j^4}{n^2}\left[(x\trsp M_j^{-1}Y_j)^2 \delta_j(Y_j)-(x\trsp M_j^{-1}X_j)^2\delta_j(X_j)\right]\\
&={\cal R}_j(1)+{\cal R}_j(2)\;.
\end{align*}
Interestingly, the first term in the above expansion, ${\cal R}_j(1)$ has mean 0, since our assumption of independence (on $X_j$'s and $Y_j$'s) guarantees that $M_j$ is independent of both $Y_j$ and $X_j$.
So we have shown that
$$
\Exp{f(X)-f(Y)}=\sum_{j=1}^n \Exp{f(Z_j)-f(Z_{j+1})-{\cal R}_j(1)} = \sum_{j=1}^n \Exp{{\cal R}_j(2)} \,.
$$

On the one hand, using the Cauchy-Schwarz inequality, we get
$$
\Expj{j}{(Y_j\trsp M_j^{-1} x)^2 |\delta_j(Y_j)|}\leq \sqrt{\Expj{j}{(Y_j\trsp M_j^{-1}x)^4}}\sqrt{\Expj{j}{\delta_j(Y_j)}^2}\;.
$$
By our assumptions \eqref{B-L} and \eqref{B-Q}, we have
$$
\Expj{j}{(Y_j\trsp M_j^{-1}x)^4}\leq (x\trsp M_j^{-2}x)^2 b_L(4;Y_j)\leq \left(\frac{x\trsp A^{-1} x}{t}\right)^2 b_L(4;Y_j)
$$
and
\begin{align}
\label{eq:deltaBound2}
\Expj{j}{\delta_j(Y_j)}^2\leq b_{Q_2}(2;Y_j)\opnorm{M_j^{-1}}^2\leq b_{Q_2}(2;Y_j) \frac{1}{t^2}\;,
\end{align}
since $M_j^{-1}\preceq A^{-1} \preceq t^{-1} \id$.
Putting everything together, and taking expectations over the other variables, we finally obtain
\begin{equation}\label{eq:firstBoundRhojYJ}
\Exp{(Y_j\trsp M_j^{-1} x)^2 |\delta_j(Y_j)|} \leq \frac{x\trsp A^{-1}x}{t^2} \sqrt{b_L(4;Y_j)}\sqrt{b_{Q_2}(2;Y_j)} \;.
\end{equation}

On the other hand, by construction, we have
\begin{align}
\label{eq:deltaBound}
\left|\tfrac{1}{n} R_j^2 \delta_j(Y_j)\right|=\left|\frac{1}{1+R_j^2d_j/n}-\frac{1}{1+R_j^2 q_j(Y_j)/n}\right|\leq 1\;,
\end{align}
because both $d_j$ and $q_j(Y_j)$ are non-negative.
Thus, we see that
$$
\tfrac{1}{n} R_j^2 \Expj{j}{(Y_j\trsp M_j^{-1} x)^2 |\delta_j(Y_j)|}\leq b_L(2;Y_j) x\trsp M_j^{-2}x \leq b_L(2;Y_j) \frac{x\trsp A^{-1}x}{t}
$$
and therefore,
\begin{equation}\label{eq:secondBoundRhojYJ}
\tfrac{1}{n} R_j^2 \Exp{(Y_j\trsp M_j^{-1} x)^2 |\delta_j(Y_j)|}\leq b_L(2;Y_j) \frac{x\trsp A^{-1}x}{t}\;.
\end{equation}

Naturally, the same bounds hold for $\Exp{(X_j\trsp M_j^{-1} x)^2 \delta_j(X_j)}$. We conclude that
\begin{align*}
\left|\Exp{f(X)-f(Y)}\right|&\leq  \sum_{j=1}^n \left[\frac{R_j^4}{n^{3/2}} \left(\frac{x\trsp A^{-1}x}{t^2} \sqrt{b_L(4;Y_j)}\sqrt{b_{Q_2}(2;Y_j)/n}\right) \wedge \frac{R_j^2}{n} \frac{x\trsp A^{-1}x}{t} b_L(2;Y_j)\right] \notag \\
&+\sum_{j=1}^n \left[\frac{R_j^4}{n^{3/2}} \left(\frac{x\trsp A^{-1}x}{t^2} \sqrt{b_L(4;X_j)}\sqrt{b_{Q_2}(2;X_j)/n}\right) \wedge \frac{R_j^2}{n} \frac{x\trsp A^{-1}x}{t} b_L(2;X_j)\right] \;,
\end{align*}
as announced in the theorem.
\end{proof}

\subsubsection{On quadratic forms involving $\bm{DX(X\trsp D^2X/n+A)^{-1}X\trsp D}$}\label{subsec:quadFormsInverseAndMean}

We are now interested in quadratic forms of the type
$$
\alpha\trsp \frac{DX}{\sqrt{n}}(X\trsp D^2 X/n+A)^{-1}\frac{X\trsp D}{\sqrt{n}}\alpha\;,
$$
which are very useful when working with both sample means and sample covariance matrices. $\alpha$ here will be a vector with norm bounded away from zero and from infinity in most cases.
Hence, we will focus without loss of generality on the case $\norm{\alpha}=1$.

Our strategy is once again to use the Lindeberg method in connection with Efron-Stein type variance bounds.

Before we turn to the technical aspects of the questions, let us make a bit more explicit our motivation. Let us call, if $\ellipticalData_i=\mu+R_i X_i$, $D$ a diagonal matrix containing the $R_i$'s, and $1$ is an $n$-dimensional vectors having 1 in all its entries,
$$
\SigmaHat=\frac{1}{n} \ellipticalData\trsp \ellipticalData-\muHat_\ellipticalData\muHat_\ellipticalData\trsp=\frac{1}{n}X\trsp D^2 X-\frac{1}{n^2}X\trsp D\trsp 11\trsp DX\;.
$$
$\SigmaHat$ is naturally the covariance matrix of our data (we assume that we observe the $\ellipticalData_i$'s). Without loss of generality, we can assume that $\mu=0$ and do so from now on in this discussion.
Let us call $\muHat=X\trsp D\trsp 1/n$, the mean of the vectors $R_i X_i$'s.
Suppose we are interested in
$$
\muHat_\ellipticalData\trsp (\SigmaHat+A)^{-1} \muHat_\ellipticalData=(\mu+\muHat)\trsp (\SigmaHat+A)^{-1} (\mu+\muHat)\;.
$$
These quantities occur naturally in various optimization problems, as well as in theoretical investigations of classification problems. Calling as before
$$
M=X\trsp D^2 X/n+A\;, \text{ we see that } \SigmaHat+A=M-\muHat\muHat\trsp\;,
$$
and hence, using the rank-1 update formula,
$$
\muHat\trsp (\SigmaHat+A)^{-1} \muHat=1-\frac{1}{1-\muHat\trsp M^{-1}\muHat}\;.
$$
Spelling out $M$ and $\muHat$, we see that
$$
\muHat\trsp M^{-1}\muHat=\alpha\trsp \frac{DX}{\sqrt{n}}(X\trsp D^2 X/n+A)^{-1}\frac{X\trsp D}{\sqrt{n}}\alpha\;,
$$
with $\alpha=1/\sqrt{n}$. Hence our motivation for understanding these problems.\\
Naturally, we will also be interested in
$$
\mu\trsp (\SigmaHat+A)^{-1}\mu=\mu\trsp M^{-1} \mu-\frac{(\mu\trsp M^{-1}\muHat)^2}{1-\muHat\trsp M^{-1} \muHat}\;.
$$
and
$$
\mu\trsp (\SigmaHat+A)^{-1}\muHat=\frac{\muHat\trsp M^{-1}\mu}{1-\muHat\trsp M^{-1} \muHat}\;.
$$

$\bullet$ \textbf{Lindeberg Approach}\\
We are now interested in
$$
g(\alpha;X)=\alpha\trsp \frac{D X}{\sqrt{n}}(X\trsp D^2 X/n+A)^{-1}\frac{X\trsp D}{\sqrt{n}}\alpha\;.
$$
The entries of $D$ are assumed to be deterministic and non-negative at this point. It is clear that this can be done without loss of generality, since $(D\alpha)_i=d_{i,i}\alpha_i$ (so negative signs in $D$ could be handled by changing the corresponding signs in $\alpha$, which would not affect $\norm{\alpha}$).

Let us observe that
\begin{align}
\label{eq:alphaMalphaBound}
|g(\alpha;X)| \leq \|\alpha\|^2 \,.
\end{align}

Indeed, setting
\begin{equation}\label{eq:DefM}
M\triangleq (X\trsp D^2 X/n+A) \succcurlyeq 0\;,
\end{equation}
we have, since $M \succeq (X\trsp D^2 X/n)$,
$$
DX M^{-1} X\trsp D\preceq \id_n\;,
$$
since $\id_n$ is greater in the Loewner order than any projection matrix.

\begin{theorem}
\label{thm:ThirdLimitIsSameForEveryone}
Suppose $X_i$ are i.i.d and $Y_i$ are i.i.d and follow the assumptions mentioned above (see Subsection \ref{subsec:SetupOfStudy}). Assume that $D$ is a deterministic diagonal matrix, whose diagonal entries are positive and denoted by $R_j$. We assume that $A$ is a positive definite matrix with $A\succeq t\id_p$, for some $t>0$.
Let us call, for a deterministic vector $\alpha$ with $\norm{\alpha}=1$ (without loss of generality),
$$
g(\alpha;X)=\alpha\trsp \frac{D X}{\sqrt{n}}(X\trsp D^2 X/n+A)^{-1}\frac{X\trsp D}{\sqrt{n}}\alpha\;.
$$
Then
\begin{align}
\left|\Exp{g(\alpha;X)-g(\alpha;Y)}\right|&\leq \sum_{i=1}^n U(X_i;R_i;\alpha_i)+U(Y_i;R_i;\alpha_i)\label{eq:LindebergForQuadFormsOfTheMeanType}\;,
\end{align}
where $U_i(X_i;R_i;\alpha_i)$ are deterministic quantities depending only on the distribution of $X_i$. We have, for a numerical constant $K$ that does not depend on the distribution of $X_i$ and $Y_i$, and not on $n$ or $p$ either,
\begin{align*}
\sum_{i=1}^n U(X_i;R_i;\alpha_i) \leq K \sum_{i=1}^n \left( \frac{R_i^2}{\sqrt{n} t}\sqrt{b_{Q_2}(2;X_i)/n} \wedge 1 \right)\left(\alpha_i^2+\frac{R_i^2}{nt} \sqrt{b_L(4;X_i)}\right)\;.
\end{align*}
\end{theorem}

Once again when the $R_i$'s are random (but independent of $\{X_i\}_{i=1}^n$ and $\{Y_i\}_{i=1}^n$, it is clear that under minimal assumptions on the existence of moments for $R_i$, the right hand side will converge to 0.
Suppose for the moment that $b_{Q_2}(2;X_i)/n$ and $b_L(4;X_i)$ are uniformly bounded
and that the $R_i$ are random and uniformly square-integrable.
Then we have
$$
\Exp{\sum_{i=1}^n \left( \frac{R_i^2}{\sqrt{n}} \wedge 1 \right) \alpha_i^2}
\leq
\tfrac{1}{\sqrt{n}} \sum_{i=1}^{n} \alpha_i^2 \Exp{R_i^2}
=
\gO(n^{-1/2})
$$
and
$$
\Exp{\sum_{i=1}^n \left( \frac{R_i^2}{\sqrt{n}} \wedge 1 \right) \frac{R_i^2}{n}}
=
o(1) \,,
$$
so that the upper bound converges to zero in $R_i$-probability (and also in expectation when the expectation is taken over $R_i$'s, $X_i$'s and $Y_i$'s). Let us now prove this theorem.

\begin{proof}

Let
$$
M := X\trsp D^2 X / n + A\; \text{ and }\; m := X\trsp D \alpha / \sqrt{n}\;.
$$

Also, let $M_i$ and $m_i$ be the corresponding functionals for $X_{(i)}$,
where $X_{(i)} := \sum_{j \ne i} e_j X_j\trsp$. In other words,
$X_{(i)}$ is obtained from $X$ by setting the $i$th row to zero.
Clearly, we have
$$
M = M_i + \tfrac{1}{n} R_i^2 X_iX_i\trsp\; \text{ and } m = m_i + \tfrac{1}{\sqrt{n}} \alpha_i R_i X_i\;.
$$
Note that $X_{(i)}$ is independent of $X_i$ and so are $M_i$ and $m_i$.
After computing the rank-1 perturbation for
$(X\trsp D^2 X/n+A)^{-1}$, we get that
$$
g(\alpha;X)=\tfrac{1}{n}\alpha\trsp (D X_{(i)}+R_i  e_i X_i\trsp)\left[M_i^{-1}-\frac{R_i^2}{n}\frac{M_i^{-1}X_iX_i\trsp M_i^{-1}}{1+R_i^2 \frac{X_i\trsp M_i^{-1} X_i}{n}}\right](X_{(i)}\trsp D+R_i X_i e_i\trsp)\alpha\;.
$$
A straightforward calculation shows that,
if $g_i(\alpha;X)=\tfrac{1}{n}\alpha\trsp D X_{(i)} M_i^{-1} X_{(i)}\trsp D\alpha$, we have the key estimate
\begin{equation}
\label{eq:alphaMalphaDecomposition}
g(\alpha;X)=g_i(\alpha;X)+\alpha_i^2-\frac{1}{1+\frac{R_i^2}{n}q_i(X_i)}(\alpha_i-\frac{R_i}{\sqrt{n}}\zeta_i)^2\;.
\end{equation}
where
$$
\zeta_i(X_i)=X_i\trsp M_i^{-1} m_i\; \text{ and  }q_i(X_i)=X_i\trsp M_i^{-1} X_i\;.
$$

We are now interested in $g(\alpha;X)-g(\alpha;Y)$. Calling $Z_j=\left(Y_1,Y_2,\ldots,Y_{j-1},X_j,\ldots,X_n\right)$, we write as before
$$
\Exp{g(\alpha;X)-g(\alpha;Y)}=\sum_{j=1}^n \Exp{g(\alpha;Z_j)-g(\alpha;Z_{j+1})}\;.
$$
It should be noted that the expansion we just got for $g(\alpha;X)$ as a function of $X_j$ also holds if we replace $X$ by $Z_j$.

With our decomposition \eqref{eq:alphaMalphaDecomposition} above, we immediately see that
$$
g(\alpha;Z_i)-g(\alpha;Z_{i+1})=\frac{1}{1+\frac{R_i^2}{n}q_i(X_i)}(\alpha_i-\frac{R_i}{\sqrt{n}}\zeta_i(X_i))^2
-\frac{1}{1+\frac{R_i^2}{n}q_i(Y_i)}(\alpha_i-\frac{R_i}{\sqrt{n}}\zeta_i(Y_i))^2\;,
$$
where now $M_i$ and $m_i$ are computed from $Z_i$ instead of $X$.
Note that $\Expj{i}{\zeta_i(X_i)}=0=\Expj{i}{\zeta_i(Y_i)}$ and $\Expj{i}{\zeta^2_i(X_i)}=\Expj{i}{\zeta_i^2(Y_i)}$ because the two have the same covariance. \\
Now let us call $$\psi_i(X_i)=(\alpha_i-\frac{R_i}{\sqrt{n}}\zeta_i(X_i))^2\;,$$
and let us define $q_i(X_i)$, $d_i$ and $\delta_i(X_i)$
as in the proof of Theorem \ref{thm:FirstFormIsDeterministic}.
Then, using Equation \eqref{eq:fromQtoD}, we have
$$
\frac{\psi_i(X_i)}{1+\frac{R_i^2}{n}q_i(X_i)}=\frac{\psi_i(X_i)}{1+\frac{R_i^2}{n}d_i}+\frac{R_i^2}{n} \psi_i(X_i) \delta_i(X_i)
$$
and therefore
$$
g(\alpha;Z_i)-g(\alpha;Z_{i+1})
=\frac{\psi_i(X_i)-\psi_i(Y_i)}{1+\frac{R_i^2}{n}d_i}
+\frac{R_i^2}{n} \left( \psi_i(X_i) \delta_i(X_i) - \psi_i(Y_i) \delta_i(Y_i) \right)
$$
So we clearly see that
$$
\Expj{i}{g(\alpha;Z_i)-g(\alpha;Z_{i+1})}=\frac{R_i^2}{n}\Expj{i}{\psi_i(X_i)\delta_i(X_i)-\psi_i(Y_i)\delta_i(Y_i)}\;.
$$
Recall that we have shown earlier that
$$
\Expj{i}{\delta_i(X_i)^2}\leq \frac{b_{Q_2}(2;X_i)}{t^2} \text{ and } \frac{R_i^2}{n}|\delta_i(X_i)|\leq 1\;.
$$
Recall also that $\zeta_i=X_i\trsp M_i^{-1} m_i$. It is clear from \eqref{eq:alphaMalphaBound} that
$$
\norm{M_i^{-1} m_i}\leq \norm{\alpha}/\sqrt{t} = 1/\sqrt{t}
$$
Hence,
$$
\Expj{i}{|\zeta_i(X_i)|^k}\leq t^{-k/2} b_L(k;X_i)\;.
$$
Using H\"older's inequality, we therefore see that
$$
\Expj{i}{\left|\psi_i(X_i)\delta_i(X_i)\right|}\leq \frac{K}{t} \sqrt{b_{Q_2}(2;X_i)} \sqrt{\alpha_i^4 +\frac{R_i^4}{n^2} b_L(4;X_i)/t^2}\leq \frac{K}{t} \sqrt{b_{Q_2}(2;X_i)} \left(\alpha_i^2+\frac{R_i^2}{n} \sqrt{b_L(4;X_i)/t^2}\right)\;.
$$
By Equation \eqref{eq:deltaBound}, we also have
$$
\tfrac{1}{n} R_i^2 |\psi_i(X_i) \delta_i(X_i)| \leq |\psi_i(X_i)| \,,
$$
whence
$$
\Expj{i}{\tfrac1n R_i^2 |\psi_i(X_i) \delta_i(X_i)|} \leq 2\left( \alpha_i^2 + \frac{R_i^2}{n}\Expj{i}{\zeta_i(X_i)}^2 \right) \leq 2\left( \alpha_i^2 + \frac{R_i^2}{n} b_L(2;X_i) / t \right) \leq 2 \left( \alpha_i^2 + \frac{R_i^2}{n} \sqrt{b_L(4;X_i) / t^2} \right) \,,
$$
since $b_L(2;X_i) \leq \sqrt{b_L(4;X_i)}$ by the Cauchy-Schwarz inequality.
Since similar estimates hold for $\psi_i(Y_i) \delta_i(Y_i)$,
it finally follows that
\begin{multline*}
\left|\Exp{g(\alpha;X)-g(\alpha;Y)}\right|\leq K \sum_{i=1}^n \bigg[ \left( \frac{R_i^2}{n^{1/2} t} \sqrt{b_{Q_2}(2;X_i)/n} \wedge 1 \right) \left(\alpha_i^2+\frac{R_i^2}{nt} \sqrt{b_L(4;X_i)}\right) \\+ \left( \frac{R_i^2}{n^{1/2}t} \sqrt{b_{Q_2}(2;Y_i)/n} \wedge 1 \right) \left(\alpha_i^2+\frac{R_i^2}{nt} \sqrt{b_L(4;Y_i)}\right) \bigg] \;.
\end{multline*}

\end{proof}

$\bullet$ \textbf{Efron-Stein aspects} We now turn to the Efron-Stein aspects of the problem, namely we show that our statistic has small variance.
\begin{theorem}\label{thm:ThirdFormIsDeterministic}
Suppose $X_i$ are i.i.d and $Y_i$ are i.i.d and follow the assumptions mentioned above. Assume that $D$ is a deterministic diagonal matrix, whose diagonal entries are positive and denoted by $R_j$. We assume that $A$ is a positive definite matrix with $A\succeq t\id_p$, for some $t>0$.
Let us call, for a deterministic vector $\alpha$ with $\norm{\alpha}=1$ (without loss of generality),
$$
g(\alpha;X)=\alpha\trsp \frac{D X}{\sqrt{n}}(X\trsp D^2 X/n+A)^{-1}\frac{X\trsp D}{\sqrt{n}}\alpha\;.
$$

Then we have, for a certain constant $K$,
$$
\var{g(\alpha;X)}\leq K \sum_{i=1}^n \left[\left(\alpha_i^4 \frac{R_i^4}{n}\frac{b_{Q_2}(2;X_i)}{nt^2}+\frac{R_i^4}{n^2} b_L(4;X_i)\frac{1}{t^2}+\alpha_i^2 \frac{R_i^2}{n} \frac{b_L(2;X_i)}{t}\right)\wedge 1\right]\label{eq:varQuadFormsOfTheMeanType}\;.
$$
\end{theorem}
Before we turn to the proof, let us show that when $R_i$'s are independent and have two moments, the upper bound converges to 0 in ($R_i$-) probability, when $b_L(4;X_i)$ and $b_{Q_2}(2;X_i)/n$ remain bounded as $n$ grows. Using the Marcienkiewicz-Zygmund strong law of large numbers, we know that
$$
\sum_{i=1}^n \frac{R_i^4}{n^2}\tendsto 0 \text{ in probability}.
$$
Now, suppose for the moment that $b_{Q_2}(2;X_i)/n$ and $b_L(4;X_i)$ are uniformly bounded
and that the $R_i$ are random. Since
\begin{align*}
  \Exp{\tfrac1n \alpha_i^4 R_i^4 \wedge 1}
&= \Exp{(\tfrac1n \alpha_i^4 R_i^4 \wedge 1)\bm{1}_{\{ \alpha_i^4R_i^4/n\geq1 \}}}
+ \Exp{(\tfrac1n \alpha_i^4 R_i^4 \wedge 1)\bm{1}_{\{ \alpha_i^4R_i^4/n  < 1 \}}} \\
&\leq \Proba{\alpha_i^2R_i^2/\sqrt{n}\geq 1} + \Exp{\tfrac{1}{\sqrt{n}} \alpha_i^2 R_i^2}
\leq 2\alpha_i^2\Exp{R_i^2}/\sqrt{n}
\end{align*}
and $\sum_{i=1}^{n} \alpha_i^2 = 1$, we see that
$$
\sum \alpha_i^4 \frac{R_i^4}{n} \tendsto 0 \text{ in } R_i-\text{probability}\;.
$$

\begin{proof}
A little bit of care is needed to handle the situation where $\norm{\alpha}_4^4$ is not small - otherwise the result could be obtained in a slightly easier fashion with slightly coarser bounds.
Recall that
$$
g(\alpha;X)-g_i(\alpha;X)=\alpha_i^2 -\frac{(\alpha_i-R_i\zeta_i(X_i)/\sqrt{n})^2}{1+R_i^2q_i(X_i)/n}\;,
$$
and therefore
$$
g(\alpha;X)-g_i(\alpha;X)=\alpha_i^2 (1-\frac{1}{1+\frac{R_i^2}{n}q_i})+\frac{1}{1+\frac{R_i^2}{n}q_i}(R_i^2/n\zeta_i^2-2\alpha_iR_i/\sqrt{n} \zeta_i)\;.
$$
Thus, if we set $T := g(\alpha;X)$ and $T_i := g_i(\alpha;X) - \alpha_i^2 (1-\frac{1}{1+\frac{R_i^2}{n}d_i})$,
which does not depend on $X_i$, we have
$$
T-T_i=\alpha_i^2 \frac{R_i^2}{n} \frac{\delta_i(X_i)}{(1+\frac{R_i^2}{n}q_i)(1+\frac{R_i^2}{n}d_i)}+\frac{1}{1+\frac{R_i^2}{n}q_i}(R_i^2/n\zeta_i^2-2\alpha_iR_i/\sqrt{n} \zeta_i)\;.
$$
So, using the bounds used in the proof of the previous theorem,
$$
\Expj{i}{|T-T_i|^2}\leq K \left(\alpha_i^4 \frac{R_i^4}{n^2}b_{Q_2}(2;X_i) \frac{1}{t^2}+\frac{R_i^4}{n^2} b_L(4;X_i)\frac{1}{t^2}+\alpha_i^2 \frac{R_i^2}{n} \frac{b_L(2;X_i)}{t}\right)\;.
$$
Using the fact that $0\leq T=g(\alpha;X)\leq 1$ and $0\leq g_i(\alpha;X)\leq 1$, we also have $|T-T_i|\leq 1+\alpha_i^2$, and hence
$$
\Expj{i}{|T - T_i|^2} \leq 2 \Expj{i}{|g(\alpha;X) - g_i(\alpha;X)|^2 + \alpha_i^4} \leq 4 \,.
$$
Thus, the Efron-Stein inequality gives us
$$
\var{g(\alpha;X)}\leq K \sum_{i=1}^n \left( \left(\alpha_i^4 \frac{R_i^4}{n^2}b_{Q_2}(2;X_i) \frac{1}{t^2}+\frac{R_i^4}{n^2} b_L(4;X_i)\frac{1}{t^2}+\alpha_i^2 \frac{R_i^2}{n} \frac{b_L(2;X_i)}{t} \right) \wedge 1 \right) \;.
$$
\end{proof}

$\bullet$ \textbf{Gaussian computations}
To understand the form we care about, it is now sufficient to compute its mean in a simple case. We naturally turn to the Gaussian case for this final task.

We now compute $\Exp{g(\alpha;X)}$ when the $X_i$'s are independent with (mean 0) normal distribution and possibly different covariance. Let us call
$$
P_R = \frac{1}{n} DXM^{-1}X\trsp D\trsp\;,
$$
with $M=X\trsp D^2 X/n +A $. $P_R$ is a $n\times n$ matrix.
We have the following result.
\begin{lemma}
Suppose that $X_i$ are independent normally distributed random variables, with mean 0 and covariance $\Sigma_i$. Then $\Exp{P_R}$ is diagonal and
$$
\Exp{g(\alpha;X)}=\sum_{i=1}^n \alpha_i^2 \Exp{P_R(i,i)}\;,
$$
where
$$
P_R(i,i)=1-\frac{1}{1+\frac{R_i^2}{n}X_i\trsp M_i^{-1} X_i}\;,
$$
and $M_i=\frac{1}{n}\sum_{j\neq i} R_j^2 X_j X_j\trsp +A$.
\end{lemma}

A particularly interesting case is that where $X_j$ are exchangeable (so for instance, we now allow the covariance $\Sigma_j$ to be random with a certain prior, and conditional on $\Sigma_j$, $X_j$'s are ${\cal N}(0,\Sigma_j)$ - the resulting random variables being exchangeable), and so are $R_i^2$ (which are assumed independent of $X_i$'s). Then we have (if $\Exp{\cdot}$ is expectation with respect to all sources of randomness)  $\Exp{P_R(i,i)}=\Exp{P_R(j,j)}$, for all $(i,j)$. In this case, we also have
$$
\Exp{\alpha\trsp P_R \beta}=\alpha\trsp \beta \left(1-\Exp{\frac{1}{1+\frac{R_1^2}{n}X_1\trsp M_1^{-1} X_1}}\right)\;.
$$
Therefore, if $\alpha\trsp \beta=0$, we have
$$
\Exp{\alpha\trsp DX(X\trsp D^2X /n+A)^{-1} X\trsp D \beta}=0\;.
$$

Another very interesting case is the situation where $R_i$'s are non-random (or random but independent of $X_i$'s) and $X_i$'s are i.i.d. Then,
$$
\Exp{g(\alpha;X)}=\norm{\alpha}_2^2-\sum_{i=1}^n \Exp{\frac{\alpha_i^2}{1+\frac{R_i^2}{n}X_i\trsp M_i(A)^{-1} X_i}}\;.
$$
Now, when $\norm{\Sigma_i}$ is not too large (i.e $\lo(p^{1/2-\eta})$, $\eta>0$, it is easy to see (by concentration of Gaussian random variables, see \cite{ledoux2001} and \cite{nekCorrEllipD} for details of the application) that $X_i\trsp M_i(A)^{-1} X_i/p$ is concentrated around its mean, which is $\trace{M_i(A)^{-1}\Sigma_i/p}$. When $\Sigma_j=\Sigma$, this quantity has a limit as $n$ and $p$ tend to $\infty$ with $p/n\tendsto \rho$, and this limit is known (see e.g.\@ \cite{mp67,silversteinbai95}). As a matter of fact, then
$$
\trace{M_i(A)^{-1}\Sigma} = \trace{X_0\trsp D_i^2 X_0/n+\Sigma^{-1/2}A \Sigma^{-1/2}}\;,
$$
where $X_0$ are i.i.d ${\cal N}(0,\id_p)$ and $D_i=D-R_ie_ie_i\trsp$. Calling $L$ this limit (which naturally depends on the distribution of $R$'s), we have
$$
g(\alpha;X)\simeq 1-\sum_{i=1}^n \frac{\alpha_i^2}{1+\frac{p}{n} R_i^2  L}\;.
$$

\begin{proof}
Notice that
$$
P_R(i,j)=\frac{1}{n} R_iR_j  X_i\trsp (M\trsp M/n+A)^{-1} X_j\;.
$$
Now changing $X_i$ into $-X_i$ does not affect the term $M\trsp M/n+A=\frac{1}{n}\sum_{i=1}^n R_i^2 X_iX_i\trsp +A$, but changes the sign of $P_R(i,j)$. On the other hand, $\{X_1,\ldots,X_{i-1},X_i,X_{i+1},\ldots,X_n\}\equalInLaw \{X_1,\ldots,X_{i-1},-X_i,X_{i+1},\ldots,X_n\}$. So we conclude that
$$
P_R(i,j)\equalInLaw -P_R(i,j) \text{ when } i\neq j\;.
$$
Now it is easy to check that in the positive semi-definite ordering, $P_R^2 \preceq P_R$. So $\opnorm{P_R}\leq 1$. So in particular, all of its entries are less than 1 in absolute value and therefore have moments.

So we have shown that when $X_i$ are independent mean 0 Gaussian variables,
$$
\Exp{P_R(i,j)}=0 \text{ if } i\neq j\;.
$$
And we therefore have the proved the lemma. (The description of the diagonal comes from using rank-1 update formulas.)
\end{proof}

\subsubsection{On $\bm{n^{-1/2} \alpha\trsp D X (X\trsp D^2 X/n+A)^{-1} x}$} \label{subsubsec:quadFormsInverseMeanAndPopMean}
These forms naturally occur in the study of quadratic forms involving both the sample mean and the sample covariance matrix as we explained at the beginning of Subsubsection \ref{subsec:quadFormsInverseAndMean}, hence our interest in them.

Therefore, for our applications, we also need results about the quantity
$$
h(\alpha;X) := n^{-1/2} \alpha\trsp D X (X\trsp D^2 X/n+A)^{-1} x\;.
$$
where $\alpha$ and $x$ are deterministic vectors, whose norm we will generally assume (without loss of generality) to be 1.

Note that if $M=X\trsp D^2X/n+A$, $\opnorm{M^{-1/2} (X\trsp D^2 X/n) M^{-1/2}}\leq 1$, and hence,
\begin{align}
\label{eq:alphaMxBound}
|h(\alpha;X)| \leq \|\alpha\| \, \sqrt{x\trsp M^{-1} x} \leq \|\alpha\| \, \sqrt{x\trsp A^{-1} x} \leq 1 / \sqrt{t} \,,
\end{align}
which follows from the Cauchy-Schwarz inequality, and \eqref{eq:alphaMalphaBound}.

\paragraph{Concentration}

Our first aim is to show that $h(\alpha;X)$ is also essentially deterministic.

\begin{theorem}
Under our usual assumptions (stated in Subsection \ref{subsec:SetupOfStudy}), we have
\label{thm:SecondFormIsDeterministic}
$$
\Exp{h(\alpha;X)-\Exp{h(\alpha;X)}}^2
\leq
K \sum_{j=1}^{n} \left[ \left( \tfrac{1}{n} \alpha_i^2 R_i^2 b_L(2,X_i) \frac{x\trsp A^{-1} x}{t} + \tfrac{1}{n^2} R_i^4 b_L(4,X_i) \frac{x\trsp A^{-1} x}{t^2} \right) \wedge \left( x\trsp A^{-1} x \right) \right] \;.
$$
\end{theorem}

\begin{proof}
This is an application of the Efron-Stein inequality.
Let $M$ and $M_i$ be defined as in the proof of Theorem \ref{thm:FirstFormIsDeterministic},
and let $m := n^{-1/2} X\trsp D \alpha$ and $m_i := n^{-1/2} X_{(i)}\trsp D \alpha$,
where $X_{(i)}$ is defined as in the proof of Theorem~\ref{thm:ThirdLimitIsSameForEveryone}.
Using the rank-1 perturbation formula once more, we get
$$
h(\alpha;X) := m\trsp M^{-1} x = (m_i\trsp + n^{-1/2} \alpha_i R_i X_i\trsp)\left[M_i^{-1}-\frac{R_i^2}{n}\frac{M_i^{-1}X_iX_i\trsp M_i^{-1}}{1+R_i^2 \frac{X_i\trsp M_i^{-1} X_i}{n}}\right]x\;.
$$
A straightforward calculation shows that,
if $h_i(\alpha;X)=m_i' M_i^{-1} x$,
\begin{align}
\label{eq:alphaMxDecomposition}
h(\alpha;X)=h_i(\alpha;X)+\frac{\varphi_i(X_i)}{1+\tfrac{1}{n}R_i^2 q_i(X_i)} \,,
\end{align}
where
\begin{align}
\label{eq:defPhi}
\varphi_i(X_i) := \left( n^{-1/2} \alpha_i R_i \, X_i\trsp M_i^{-1} x - \tfrac1n R_i^2 \, X_i\trsp M_i^{-1} m_i \, X_i\trsp M_i^{-1} x \right) \,.
\end{align}
Note that $h_i(\alpha;X)$ is independent of $X_i$ here.
Thus, the Efron-Stein inequality yields
$$
\var{h(\alpha;X)} \leq \sum_{i=1}^{n} \var{h(\alpha;X) - h_i(\alpha;X)} \leq \sum_{i=1}^{n} \Exp{\frac{\varphi_i(X_i)}{1+\tfrac{1}{n}R_i^2 q_i(X_i)}}^2 \,.
$$
Now, on the one hand, using \eqref{B-L}, we have
$$
\Exp{X_i\trsp M_i^{-1} x}^2 \leq b_L(2,X_i) \, x\trsp A^{-1} x /t \;,
$$
$$
\Exp{X_i\trsp M_i^{-1} m_i \, X_i\trsp M_i^{-1} x}^2 \leq \sqrt{b_L(4,X_i) \, (x\trsp A^{-1} x /t)^2} \sqrt{b_L(4,X_i)/t^2} \;,
$$
and therefore
\begin{align}
\label{eq:phiBound2}
\Exp{\frac{\varphi_i(X_i)}{1+\tfrac{1}{n}R_i^2 q_i(X_i)}}^2
\leq \Exp{\varphi_i(X_i)}^2
\leq 2 \left( \tfrac{1}{n} \alpha_i^2 R_i^2 b_L(2,X_i) \frac{x\trsp A^{-1} x}{t} + \tfrac{1}{n^2} R_i^4 b_L(4,X_i) \frac{x\trsp A^{-1} x}{t^2} \right) \;.
\end{align}
On the other hand, it follows from \eqref{eq:alphaMxDecomposition} and \eqref{eq:alphaMxBound} that
$$
\Exp{\frac{\varphi_i(X_i)}{1+\tfrac{1}{n}R_i^2 q_i(X_i)}}^2
\leq 2 \left( \Exp{h(\alpha;X)}^2 + \Exp{h_i(\alpha;X)}^2 \right)
\leq 4 x\trsp A^{-1} x \,.
$$
The proof is completed by combining these estimates.
\end{proof}

$\bullet$ \textbf{Lindeberg approach}\\

Our next aim is to show that the limit of $h(\alpha;X)$ does not depend
on the distribution of the $X_i$.

\begin{theorem}
Under our usual assumptions (stated in Subsection \ref{subsec:SetupOfStudy}), we have
\label{thm:SecondLimitIsSameForEveryone}
\begin{multline*}
\left|\Exp{h(\alpha;X)-h(\alpha;Y)}\right| \leq
K \sum_{j=1}^{n} U_j(X_j)+U_j(Y_j)\;, \text{ with }\\
U_j(X_j)\leq
K \left[ \frac{R_i^2}{nt} \sqrt{b_Q(2;X_i)} \wedge 1 \right] \cdot \left[ \left( \tfrac{1}{\sqrt{n}} |\alpha_i| R_i \sqrt{b_L(2,X_i)} \sqrt{\frac{x\trsp A^{-1} x}{t}} + \tfrac{1}{n} R_i^2 \sqrt{b_L(4,X_i)} \sqrt{\frac{x\trsp A^{-1} x}{t^2}} \right) \right] \;.
\end{multline*}
\end{theorem}

\begin{proof}
We use the notation from the proof of Theorem \ref{thm:FirstLimitIsSameForEveryone}.
Using the decomposition \eqref{eq:alphaMxDecomposition} with $X$ replaced by $Z_j,Z_{j+1}$
and observing that $h_j(\alpha;Z_j) = h_j(\alpha;Z_{j+1})$, we get
\begin{align*}
   \Exp{h(\alpha;X)-h(\alpha;Y)}
&= \sum_{j=1}^{n} \Exp{h(\alpha;Z_j)-h(\alpha;Z_{j+1})} \\
&= \sum_{i=1}^{n} \Exp{\frac{\varphi_i(X_i)}{1+\tfrac{1}{n}R_i^2 q_i(X_i)}-\frac{\varphi_i(Y_i)}{1+\tfrac{1}{n}R_i^2 q_i(Y_i)}} \,,
\end{align*}
where $\varphi_i(X_i)$ is defined as in \eqref{eq:defPhi},
but with $X$ replaced by $Z_i$.
Next, using \eqref{eq:fromQtoD}, we have
\begin{align*}
  \frac{\varphi_i(X_i)}{1+\tfrac{1}{n}R_i^2 q_i(X_i)}
- \frac{\varphi_i(Y_i)}{1+\tfrac{1}{n}R_i^2 q_i(Y_i)}
&= \left( \frac{\varphi_i(X_i)}{1+\tfrac{1}{n}R_i^2 d_i}
- \frac{\varphi_i(Y_i)}{1+\tfrac{1}{n}R_i^2 d_i} \right)
+ \frac{R_i^2}{n} \Big( \varphi_i(X_i) \delta_i(X_i) - \varphi_i(Y_i) \delta_i(Y_i) \Big) \,.
\end{align*}
Since $X_i$ and $Y_i$ both have mean $0$ and covariance $\Sigma_i$,
it follows that
\begin{align*}
\Expj{i}{\frac{\varphi_i(X_i)}{1+\tfrac{1}{n}R_i^2 q_i(X_i)}
- \frac{\varphi_i(Y_i)}{1+\tfrac{1}{n}R_i^2 q_i(Y_i)}}
= \Expj{i}{\frac{R_i^2}{n} \Big( \varphi_i(X_i) \delta_i(X_i) - \varphi_i(Y_i) \delta_i(Y_i) \Big)} \,.
\end{align*}
Now, on the one hand, using Cauchy-Schwarz inequality
as well as \eqref{eq:deltaBound2} and \eqref{eq:phiBound2},
we have
\begin{align*}
      \Expj{i}{\left| \varphi_i(X_i) \delta_i(X_i) \right|}
&\leq \left( \Expj{i}{\delta_i(X_i)}^2 \Expj{i}{\varphi_i(X_i)}^2 \right)^{1/2} \\
&\leq \frac{K}{t} \sqrt{b_Q(2;X_i)} \left( \tfrac{1}{n} \alpha_i^2 R_i^2 b_L(2,X_i) \frac{x\trsp A^{-1} x}{t} + \tfrac{1}{n^2} R_i^4 b_L(4,X_i) \frac{x\trsp A^{-1} x}{t^2} \right)^{1/2} \\
&\leq \frac{K}{t} \sqrt{b_Q(2;X_i)} \left( \tfrac{1}{\sqrt{n}} |\alpha_i| R_i \sqrt{b_L(2,X_i)} \sqrt{\frac{x\trsp A^{-1} x}{t}} + \tfrac{1}{n} R_i^2 \sqrt{b_L(4,X_i)} \sqrt{\frac{x\trsp A^{-1} x}{t^2}} \right) \,.
\end{align*}
On the other hand, using \eqref{eq:deltaBound}, we get
\begin{align*}
     \tfrac{1}{n} R_i^2 \Expj{i}{\left| \varphi_i(X_i) \delta_i(X_i) \right|}
\leq \Expj{i}{|\varphi_i(X_i)|}
\leq \tfrac{1}{\sqrt{n}} |\alpha_i| R_i {b_L(1,X_i)} \sqrt{\frac{x\trsp A^{-1} x}{t}} + \tfrac{1}{n} R_i^2 {b_L(2,X_i)} \sqrt{\frac{x\trsp A^{-1} x}{t^2}} \,.
\end{align*}
We now use that $b_L(k,X_i) \leq \sqrt{b_L(2k,X_i)}$.
Combining these estimates, we get
\begin{align*}
      \tfrac{1}{n} R_i^2 \Exp{\left| \varphi_i(X_i) \delta_i(X_i) \right|}
&\leq K \left( \frac{R_i^2}{nt} \sqrt{b_Q(2;X_i)} \wedge 1 \right) \left( \tfrac{1}{\sqrt{n}} |\alpha_i| R_i \sqrt{b_L(2,X_i)} \sqrt{\frac{x\trsp A^{-1} x}{t}} + \tfrac{1}{n} R_i^2 \sqrt{b_L(4,X_i)} \sqrt{\frac{x\trsp A^{-1} x}{t^2}} \right) \,.
\end{align*}
Since similar estimates hold for $\varphi_i(Y_i) \delta_i(Y_i)$,
this completes the proof.
\end{proof}

$\bullet$ \textbf{Gaussian computations}\\

Consider the case where the $X_i$ are independent normal random vectors
with mean zero and covariance $\Sigma_i$.
Then we clearly have $X \equalInLaw -X$ and therefore
$$
h(\alpha,X) \equalInLaw h(\alpha,-X) = - h(\alpha,X) \,.
$$
But this means that we must have $\Exp{h(\alpha,X)} = 0$. (Recall that $|h(\alpha;X)|\leq 1/\sqrt{t}$ by Equation \eqref{eq:alphaMxBound}, so the existence of the equation is not a problem.)

\subsection{Forms in $M^{-1}\Sigma_{\eps}M^{-1}$, $\Sigma_\eps\succeq 0$}
In a variety of situations, we will need to work with quantities of the type
$$
x\trsp M^{-1}\Sigma_{\eps}M^{-1} x\;.
$$
These quantities will occur when we study $\var{x\trsp M^{-1} \eps}$ where $\eps$ has mean 0 and covariance $\Sigma_{\eps}$. So these quantities will appear when we investigate the risk of various estimators (or asset allocations). This is why we restrict ourselves to $\Sigma_{\eps}\succeq 0$, though our proofs would go through with minor adjustments if $\Sigma_{\eps}$ was allowed to be more general.

We will also need to understand
$$
\muHat\trsp M^{-1}\Sigma_{\eps} M^{-1} \muHat\;,
$$
if we want to understand the risk properties of certain portfolio allocations.

Hence our problem is the following: in all the forms where before $M^{-1}$ was involved, we now want to work with $M^{-1} \Sigma_{\eps}M^{-1}$ instead. Our idea - somewhat similar to the one developed in \cite{nekMarkoRealizedRisk} - is the following: consider
$$
M_u=X\trsp D^2 X/n+A+u\Sigma_{\eps}\;, \text{ and } M_0=M\;.
$$
We remark that
$$
\left.\frac{\partial}{\partial u} M_u^{-1}\right|_{u=0}=-M^{-1} \Sigma_{\eps} M^{-1}\;.
$$

Hence, at least formally, our previous proofs will go through; the only thing we have to do is replace $A$ by $A+u\Sigma_{\eps}$ and take a derivative with respect to $u$ so we can get the decompositions that will help us make our methods work.

\subsubsection{Forms in $x\trsp M^{-1} \Sigma_{\eps} M^{-1} x$}\label{subsubsec:FormsInxMinvSigmaepsMinvx}
It is natural to study these forms in a variety of contexts, for instance when $x=\mu$. We have the following theorem, which holds under what we now call our ``usual assumptions", namely $A \succeq t\id$, $X_i$ are i.i.d with mean 0 and covariance $\Sigma_i$, and so are $Y_i$'s, though $X_i$ and $Y_i$ have different distributions.
\begin{theorem}\label{thm:FormsInxMinvSigmaepsMinvx}
Let $M=\frac{1}{n}\sum_{i=1}^n R_i^2 X_iX_i \trsp+A$ and
$$
F(X)=x\trsp M^{-1} \Sigma_{\eps} M^{-1} x\;.
$$
Let us call
$$
b(A,\Sigma_\eps)=\opnorm{A^{-1/2}\Sigma_\eps A^{-1/2}}\;.
$$
Then
$$
\var{F(X)}\leq K (x\trsp A^{-1} x)^{2} b(A,\Sigma_\eps)^{2}\sum_{i=1}^n \left[\left(\frac{R_i^{4}}{n^{2}}\frac{1}{t^2} b_L(4;X_i)\wedge 1\right)\right]\;.
$$
Also,
\begin{align*}
&\left|\Exp{F(X)-F(Y)}\right| \leq \sum_{i=1}^n U_i(X_i)+U_i(Y_i)\\
&U_i(X_i) \leq K\, b(A;\Sigma_\eps)\frac{x\trsp A^{-1} x}{t}\left\{ \left[\frac{R_i^4}{n^{3/2}}\sqrt{\frac{b_{Q_2}(2;X_i)}{n}} \sqrt{b_{L}(4;X_i)} \frac{1}{t}\right]\wedge \frac{R_i^2}{n} b_L(2;X_i) \right\}
\end{align*}
\end{theorem}
As is explained in the proof of the theorem, when $R_i$'s are i.i.d and uniformly square integrable, the upper bound goes to zero, provided $\sqrt{\frac{b_{Q_2}(2;X_i)}{n}}$ and $b_L(4;X_i)$ remain bounded.

It should be noted that when $X_i$ are i.i.d with covariance $\Sigma$, we have found in Heuristic \ref{heuristic:squareRegularizedInverse} and its proof a deterministic equivalent for $F(X)$. Naturally, our theorem shows that doing computations in the Gaussian case is enough to understand $\Exp{F(Y)}$ for $Y$ with a variety of distributions - that is the essence of Lindeberg-style results.

\begin{proof}

Let us call
$$
f(X)=f(X;A)=x\trsp M^{-1} x=x\trsp (X\trsp D^2 X/n+A)^{-1}x\;.
$$
Call $f_i(X)$ the same quantity where $D_{i,i}=R_i$ is replaced by $D_{i,i}=0$ (or equivalently $X_i$ is replaced by $0$).

Our key estimate was
$$
f(X)-f_i(X)=-\frac{R_i^2}{n}\frac{(x\trsp M_i^{-1} X_i)^2}{1+\frac{R_i^2}{n} X_i\trsp M_i^{-1} X_i}\;.
$$
This equality is true if $A$ is replaced by $A+u\Sigma_{\eps}$. Now we can take the derivative of this expression with respect to $u$. Call
$$
F(X)=x\trsp M^{-1} \Sigma_{\eps} M^{-1} x\;,
$$
and $F_i$ the same quantity when $X_i$ is replaced by $0$.
We have
$$
F(X)-F_i(X)=- \left.\frac{\partial}{\partial u}\right|_{u=0}\left[f(X;A+u\Sigma_{\eps})-f_i(X;A+u\Sigma_{\eps})\right]\;.
$$
Recall the notations $q_i(X_i)=X_i\trsp M_i^{-1} X_i$, $d_i=\trace{M_i^{-1}\Sigma_i}$.
After taking the derivative, we get
$$
-(F(X)-F_i(X))=2\frac{R_i^2}{n}\frac{x\trsp M_i^{-1}X_i\, x\trsp M_i^{-1} \Sigma_{\eps} M_i^{-1} X_i}{1+\frac{R_i^2}{n} q_i}-\frac{R_i^4}{n^2}\frac{(x\trsp M_i^{-1}X_i)^2\, X_i\trsp M_i^{-1} \Sigma_{\eps} M_i^{-1} X_i}{(1+\frac{R_i^2}{n} q_i)^2}\;.
$$

Let us call
\begin{align*}
\tilde{q}_i(X_i)&=X_i\trsp M_i^{-1}\Sigma_\eps M_i^{-1} X_i\;,\\
\tilde{d}_i(X_i)&=\trace{\Sigma_i M_i^{-1}\Sigma_\eps M_i^{-1}}\;,\\
E_1&=2\frac{R_i^2}{n}\frac{x\trsp M_i^{-1}X_i\, x\trsp M_i^{-1} \Sigma_{\eps} M_i^{-1} X_i}{1+\frac{R_i^2}{n} q_i}\;,\\
E_2&=\frac{R_i^4}{n^2}\frac{(x\trsp M_i^{-1}X_i)^2\, \tilde{q}_i(X_i)}{(1+\frac{R_i^2}{n}q_i)^2}
\end{align*}

$\bullet$ \textbf{Control of $E_1$}
By using the fact that $|v\trsp M_i^{-1} X_i|\leq \sqrt{v\trsp M_i^{-1} v}\sqrt{X_i\trsp M_i^{-1} X_i}$, we see that
$$
|E_1|\leq 2\sqrt{x\trsp M_i^{-1} x}\sqrt{x\trsp M_i^{-1/2}(M_i^{-1/2}\Sigma_\eps M_i^{-1/2})^2 M_i^{-1/2}x}
$$

\paragraph{On $\mathbf{M_i^{-1}\Sigma_\eps M_i^{-1}}$}
We first note that $M_i^{-1/2}M_i^{-1/2}\Sigma_\eps M_i^{-1/2}M_i^{-1/2}\preceq M_i^{-1} \opnorm{M_i^{-1/2}\Sigma_\eps M_i^{-1/2}}$. Now $\opnorm{M_i^{-1/2}\Sigma_\eps M_i^{-1/2}}=\lambda_{\max}(M_i^{-1/2}\Sigma_\eps M_i^{-1/2})=\lambda_{\max}(\Sigma_\eps^{1/2}M_i^{-1}\Sigma_\eps^{1/2})$ by e.g.\@ similarity. Now $\Sigma_\eps^{1/2}M_i^{-1}\Sigma_\eps^{1/2}\preceq
\Sigma_\eps^{1/2}A^{-1}\Sigma_\eps^{1/2}$, so
$$
\opnorm{M_i^{-1/2}\Sigma_\eps M_i^{-1/2}}\leq \lambda_{\max}(\Sigma_\eps^{1/2}A^{-1}\Sigma_\eps^{1/2})=\opnorm{\Sigma_\eps^{1/2}A^{-1}\Sigma_\eps^{1/2}}=b(A;\Sigma_{\eps})\;.
$$
We therefore also have
$$
\opnorm{M_i^{-1}\Sigma_\eps M_i^{-1}}\leq \frac{b(A;\Sigma_\eps)}{t}\;.
$$

We will also repeatedly need to control $\norm{M_i^{-1}\Sigma_{\eps}M_i^{-1}v}$ for a fixed vector $v$. Call $u=M_i^{-1}\Sigma_{\eps}M_i^{-1}v$. Clearly
$$
u\trsp u=v\trsp M_i^{-1/2} (M_i^{-1/2}\Sigma_\eps M_i^{-1/2}) M_i^{-1} (M_i^{-1/2}\Sigma_\eps M_i^{-1/2}) M_i^{-1/2}v\;.
$$
Now, using our bounds on $\opnorm{(M_i^{-1/2}\Sigma_\eps M_i^{-1/2})}$, $\opnorm{M_i^{-1}}\leq t$ and the fact that $\opnorm{\cdot}$ is submultiplicative, we have
$$
\opnorm{(M_i^{-1/2}\Sigma_\eps M_i^{-1/2}) M_i^{-1} (M_i^{-1/2}\Sigma_\eps M_i^{-1/2})}\leq \frac{b^2(A;\Sigma_{\eps})}{t}\;.
$$
So
$$
\norm{M_i^{-1}\Sigma_{\eps}M_i^{-1}v}^2=u\trsp u \leq \frac{b^2(A;\Sigma_{\eps})}{t} v\trsp M_i^{-1} v\leq \frac{b^2(A;\Sigma_{\eps})}{t} v\trsp A^{-1} v
$$
Finally, we conclude that
$$
\norm{M_i^{-1}\Sigma_{\eps}M_i^{-1}v}\leq \frac{b(A;\Sigma_{\eps})}{\sqrt{t}}\norm{M_i^{-1}v}\leq \frac{b(A;\Sigma_{\eps})}{\sqrt{t}} \sqrt{v \trsp A^{-1} v}\;.
$$

Using the previous bounds, we clearly then have
$$
|E_1|\leq 2x\trsp A^{-1} x \; b(A,\Sigma_\eps)\;.
$$
On the other hand, using the Cauchy-Schwarz inequality, we see that
$$
\Expj{i}{|E_1|^k}\leq
K \frac{R_i^{2k}}{n^{k}}(x\trsp A^{-1} x)^{k} \frac{b(A,\Sigma_\eps)^{k}}{t^k} b_L(2k;X_i)\;.
$$

Hence,
$$
\Exp{|E_1|^k}\leq K (x\trsp A^{-1} x)^{k} b(A,\Sigma_\eps)^{k}\left[\frac{R_i^{2k}}{n^{k}}\frac{1}{t^k} b_L(2k;X_i)\wedge 1\right]
$$

$\bullet$ \textbf{Control of $E_2$}
Writing
$$
-E_2=\frac{R_i^2}{n}\frac{(x\trsp M_i^{-1}X_i)^2}{(1+\frac{R_i^2}{n}q_i)} \frac{R_i^2}{n}\frac{\tilde{q}_i(X_i)}{(1+\frac{R_i^2}{n}q_i)}
$$
we remark that
$$
0\leq \frac{R_i^2}{n}\frac{\tilde{q}_i(X_i)}{(1+\frac{R_i^2}{n}q_i)} \leq b(A;\Sigma_\eps) \; \text{ and } |E_2|\leq \frac{R_i^2}{n}\frac{(x\trsp M_i^{-1}X_i)^2}{(1+\frac{R_i^2}{n}q_i)} b(A;\Sigma_\eps)\;.
$$
Hence, we can conclude that
$$
|E_2|\leq x\trsp A^{-1} x \; \; b(A,\Sigma_\eps)\;.
$$
We also have the inequalities
$$
\Expj{i}{|E_2|^k}\leq \frac{R_i^{2k}}{n^k}b_L(2k;X_i) (x\trsp M_i^{-2}x)^k  b(A;\Sigma_\eps)^k \leq
\frac{R_i^{2k}}{n^k}b_L(2k;X_i) (x\trsp A^{-1}x)^k  \frac{b(A;\Sigma_\eps)^k}{t^k}\;.
$$
Therefore,
$$
\Exp{|E_2|^k}\leq K (x\trsp A^{-1} x)^k \; \; b(A,\Sigma_\eps)^k \left(\frac{R_i^{2k}}{t^k n^k}b_L(2k;X_i)\wedge 1\right)\;.
$$
$\bullet$ \textbf{Efron-Stein aspects}\\
Using the Efron-Stein inequality, we have
$$
\var{F(X)}\leq\sum_{i=1}^n \var{F(X)-F_i(X)}\leq K (x\trsp A^{-1} x)^{2} b(A,\Sigma_\eps)^{2}\sum_{i=1}^n \left[\left(\frac{R_i^{4}}{n^{2}}\frac{1}{t^2} b_L(4;X_i)\wedge 1\right)\right]\;.
$$
Hence, when $R_i$'s have 2 moments, $\var{F(X)}\tendsto 0$ in $R_i$-probability.

$\bullet$ \textbf{Lindeberg aspects}\\
We go a bit fast here. $M_i$ is now computed from data $X_1,\ldots,X_{i-1},0,Y_{i+1},\ldots,Y_n$.
Recall that
$$
E_1(X_i)=2\frac{R_i^2}{n}\frac{x\trsp M_i^{-1}X_i\, x\trsp M_i^{-1} \Sigma_{\eps} M_i^{-1} X_i}{1+\frac{R_i^2}{n} q_i(X_i)}
$$
Let us show that, when $Y_i$ and $X_i$ have the same covariance $\Sigma_i$ and mean 0, we can control
$$
\sum_{i=1}^n \Exp{E_1(X_i)-E_1(Y_i)}\;.
$$
We call $N_i=x\trsp M_i^{-1}X_i\, x\trsp M_i^{-1} \Sigma_{\eps} M_i^{-1} X_i$ and (as before) $d_i=\trace{\Sigma_i M_i^{-1}}$. We have
$$
E_1(X_i)=2\frac{R_i^2}{n}\frac{N_i(X_i)}{1+\frac{R_i^2}{n} d_i}+2\frac{R_i^2}{n} N_i(X_i) \frac{R_i^2/n(d_i-q_i(X_i))}{(1+\frac{R_i^2}{n}d_i)(1+\frac{R_i^2}{n}q_i(X_i))}\;.
$$
Note that $\Expj{i}{N_i(X_i)}=\Expj{i}{N_i(Y_i)}$, so to control $\Exp{E_1(X_i)-E_1(Y_i)}$, we just need to understand the second term, namely
$$
{\cal R}_1(X_i)=2\frac{R_i^2}{n} N_i(X_i) \frac{R_i^2/n(d_i-q_i(X_i))}{(1+\frac{R_i^2}{n}d_i)(1+\frac{R_i^2}{n}q_i(X_i))}\;.
$$
Our studies in Subsubsection \ref{subsec:quadFormsInverse} show that
$$
\delta_i(X_i)=\frac{(d_i-q_i(X_i))}{(1+\frac{R_i^2}{n}d_i)(1+\frac{R_i^2}{n}q_i(X_i))}
$$
is such that
$$
|R_i^2/n\delta_i(X_i)|\leq 1\;.
$$
On the other hand, we have essentially given bounds earlier for $\Expj{i}{|N_i|^k}$ (see the work on $\Expj{i}{|E_i|^k}$), so we have
$$
\Expj{i}{|{\cal R}_1(X_i)|}\leq K \frac{R_i^2}{n} b_L(2;X_i) \frac{x\trsp A^{-1} x}{t} b(A;\Sigma_\eps)\;.
$$
Furthermore,
$$
|{\cal R}_1(X_i)|\leq K \frac{R_i^4}{n^2} |N_i| |d_i-q_i(X_i)|\;.
$$
So
$$
\Expj{i}{|{\cal R}_1(X_i)|}\leq K \frac{R_i^4}{n^2} \sqrt{\Expj{i}{|d_i-q_i(X_i)|^2}}\sqrt{\Expj{i}{N_i^2}}\;.
$$
Using our bounds on $\Expj{i}{N_i^k}$ and those on $\Expj{i}{|d_i-q_i(X_i)|^2}$, we get
$$
\Expj{i}{|{\cal R}_1(X_i)|}\leq K \frac{R_i^4}{n^2} \frac{x\trsp A^{-1} x}{t} b(A;\Sigma_\eps)\frac{\sqrt{b_{Q_2}(2;X_i)}}{t}\sqrt{b_L(4;X_i)}\;.
$$
We conclude that
$$
|\Exp{{\cal R}_1(X_i)}|\leq K \frac{x\trsp A^{-1} x}{t} b(A;\Sigma_\eps) \left[\frac{R_i^4}{n^{3/2}} \sqrt{\frac{b_{Q_2}(2;X_i)}{n}}\frac{\sqrt{b_L(4;X_i)}}{t}\wedge \frac{R_i^2}{n} b_L(2;X_i) \right]\;,
$$
and similarly for $Y_i$. We have shown that
$$
\left|\Exp{E_1(X_i)-E_1(Y_i)}\right|\leq K \frac{x\trsp A^{-1} x}{t} b(A;\Sigma_\eps) \left[\frac{R_i^4}{n^{3/2}} \sqrt{\frac{b_{Q_2}(2;X_i)}{n}}\frac{\sqrt{b_L(4;X_i)}}{t}\wedge \frac{R_i^2}{n} b_L(2;X_i) \right]\;,
$$
and we can therefore control
$$
\left|\sum_{i=1}^n \Exp{E_1(X_i)-E_1(Y_i)}\right|\;.
$$
$\bullet$ \textbf{About $\mathbf{E_2}$}\\
We now turn to the $E_2$ part of the problem. The strategy is to replace
\begin{align*}
\tilde{q}_i(X_i)&=X_i\trsp M_i^{-1} \Sigma_\eps M_i^{-1} X_i
\text{ by the ``equivalent" (and independent of $X_i$)}\\
\tilde{d}_i&=\trace{\Sigma_i M_i^{-1} \Sigma_\eps M_i^{-1}}\;,
\end{align*}
and similarly to replace $q_i(X_i)=X_i\trsp M_i^{-1}X_i$ by $d_i(X_i)=\trace{\Sigma_i M_i^{-1}}$. Hence the first term is going to have the same mean for both $X_i$ and $Y_i$ and we just have to work on the remainders.
Let us call
$$
\Delta_i(X_i)=\frac{1}{(1+\frac{R_i^2}{n}q_i(X_i))^2}-\frac{1}{(1+\frac{R_i^2}{n}d_i(X_i))^2}
$$
and let us remark that, with the $\delta_i(X_i)$ notation we just recalled, we have
$$
\Delta_i(X_i)=\frac{R_i^2}{n} \delta_i(X_i) \left[\frac{1}{1+\frac{R_i^2}{n}q_i(X_i)}+\frac{1}{1+\frac{R_i^2}{n}d_i(X_i)}\right]\;.
$$
With this notation, we have
\begin{align*}
E_2(X_i)&=\frac{R_i^4}{n^2} \frac{(x\trsp M_i^{-1}X_i)^2 \tilde{d}_i}{(1+\frac{R_i^2}{n} d_i)^2}
+\frac{R_i^4}{n^2}(x\trsp M_i^{-1}X_i)^2 \frac{ \tilde{q}_i(X_i)-\tilde{d}_i}{(1+\frac{R_i^2}{n} d_i)^2}
+\Delta_i \frac{R_i^4}{n^2} (x\trsp M_i^{-1} X_i)^2 \tilde{q}_i(X_i)\;\\
&\triangleq {\cal M}_i(X_i)+{\cal R}_{2,1}(X_i)+{\cal R}_{2,2}(X_i)\;.
\end{align*}
Note that by construction $\Expj{i}{{\cal M}_i(X_i)}=\Expj{i}{{\cal M}_i(Y_i)}$, so to bound
$\Exp{E_2(X_i)-E_2(Y_i)}$, all we will have to do is bound $\Exp{|{\cal R}_{2,1}(X_i)|}$ and $\Exp{|{\cal R}_{2,2}(X_i)|}$ .
Before we turn to this task, let us recall that
$$
|E_2(X_i)|\leq \frac{R_i^2}{n}(x\trsp M_i^{-1} X_i)^2\, b(A,\Sigma_\eps)\;.
$$
In other respects, if $A$ and $B$ are positive semi-definite (psd) matrices and $\opnorm{B}\leq C$, then $\trace{AB}\leq C \trace{A}$ (because when $A$ and $B$ are psd, $A^{1/2} B A^{1/2}\preceq \opnorm{B} A$). Therefore,
$$
\tilde{d}_i\leq b(A;\Sigma_\eps) d_i \text{ and } {\cal M}_i(X_i)\leq \frac{R_i^2}{n}(x\trsp M_i^{-1} X_i)^2\, b(A,\Sigma_\eps)\;.
$$
Hence,
$$
|{\cal R}_{2,1}(X_i)+{\cal R}_{2,2}(X_i)|\leq K \frac{R_i^2}{n}(x\trsp M_i^{-1} X_i)^2\, b(A,\Sigma_\eps)\;.
$$
Let us now work more precisely on ${\cal R}_{2,1}(X_i)$ and ${\cal R}_{2,2}(X_i)$\;. \\
$\bullet$ \textbf{On $\mathbf{{\cal R}_{2,1}(X_i)}$.}\\
Note that, using $\opnorm{M_i^{-1}\Sigma_\eps M_i^{-1}}\leq b(A;\Sigma_\eps)/t$, we have
$$
\Expj{i}{|\tilde{q}_i(X_i)-\tilde{d}_i(X_i)|^2}\leq \frac{b^2(A;\Sigma_\eps)}{t^2} b_{Q_2}(2;X_i)\;.
$$
Using the Cauchy-Schwarz inequality in connection with the previous remark, we get
\begin{align*}
\Expj{i}{|{\cal R}_{2,1}(X_i)|}&\leq \frac{R_i^4}{n^{3/2}} (x\trsp M_i^{-2}x)\frac{b(A;\Sigma_\eps)}{t} \sqrt{\frac{b_{Q_2}(2;X_i)}{n}}\sqrt{b_L(4;X_i)}\\
&\leq \frac{R_i^4}{n^{3/2}} \frac{(x\trsp A^{-1}x)}{t^2}b(A;\Sigma_\eps) \sqrt{\frac{b_{Q_2}(2;X_i)}{n}}\sqrt{b_L(4;X_i)}\;.
\end{align*}
$\bullet$ \textbf{On $\mathbf{{\cal R}_{2,2}(X_i)}$.}\\
Let us first note that
$$
\frac{\tilde{q}_i(X_i)}{q_i(X_i)}\leq b(A;\Sigma_\eps)  \text{ and }
\left|\Delta_i\right|\leq K \frac{R_i^2}{n}\frac{|d_i-q_i|}{(1+\frac{R_i^2}{n}d_i)(1+\frac{R_i^2}{n}q_i(X_i))}\;.
$$

Hence,
$$
|\Delta_i \tilde{q}_i(X_i)|\leq K b(A;\Sigma_\eps)\frac{|d_i-q_i|}{(1+\frac{R_i^2}{n}d_i)}\;.
$$
We can therefore conclude that
$$
|{\cal R}_{2,2}(X_i)|\leq K \frac{R_i^4}{n^2} b(A;\Sigma_\eps)\frac{|d_i-q_i|}{(1+\frac{R_i^2}{n}d_i)}(x\trsp M_i^{-1} X_i)^2\;.
$$
Using the Cauchy-Schwarz inequality we also get
\begin{align*}
\Expj{i}{|{\cal R}_{2,2}(X_i)|}&\leq K \frac{R_i^4}{n^{3/2}} \frac{b(A;\Sigma_\eps)}{t} (x\trsp M_i^{-2}x) \sqrt{\frac{b_{Q_2}(2;X_i)}{n}} \sqrt{b_{L}(4;X_i)}\;\\
&\leq K \frac{R_i^4}{n^{3/2}} b(A;\Sigma_\eps)  \frac{x\trsp A^{-1} x}{t^2}\sqrt{\frac{b_{Q_2}(2;X_i)}{n}} \sqrt{b_{L}(4;X_i)}\;.
\end{align*}

We conclude that if $U_i=\Exp{|{\cal R}_{2,1}(X_i)+{\cal R}_{2,2}(X_i)|}$,
$$
\left|\Exp{E_2(X_i)-E_2(Y_i)}\right|\leq U_i(X_i)+U_i(Y_i)\;,
$$
where
$$
U_i(X_i)\leq K\, b(A;\Sigma_\eps)\frac{x\trsp A^{-1} x}{t} \left[\frac{R_i^4}{n^{3/2}}\sqrt{\frac{b_{Q_2}(2;X_i)}{n}} \sqrt{b_{L}(4;X_i)} \frac{1}{t}\right]\wedge \frac{R_i^2}{n} b_L(2;X_i) \;.
$$

Finally, putting everything together we have shown that
$$
\left|\Exp{F(X)-F(Y)}\right|\leq K \sum_{i=1}^n U_i(X_i)+U_i(Y_i)\;.
$$
We conclude that when $R_i$ are independent and have 2 moments, the upper bound goes to zero in $R_i$-probability, provided $b_L(4;X_i)$ and $b_{Q_2}(2;X_i)/n$ remain uniformly bounded (we have already analyzed similar series previously).
\end{proof}

\newcommand{\zetaTilde}{\tilde{\zeta}}
\newcommand{\qTilde}{\tilde{q}}
\newcommand{\dTilde}{\tilde{d}}
\subsubsection{Forms in $\alpha\trsp \frac{DX\trsp}{\sqrt{n}} M^{-1} \Sigma_{\eps} M^{-1}\frac{X\trsp D}{\sqrt{n}} \alpha$}\label{subsubsec:formsinDXMinvSigmaEpsMinvXtrspD}

In the analysis of quantities of the type
$$
\muHat\trsp (\SigmaHat+A)^{-1}\Sigma_\eps(\SigmaHat+A)^{-1} \muHat
$$
we will naturally have to understand quantities of the type, if $M=X\trsp D^2 X/n+A$,
$$
G(\alpha;X)=\alpha\trsp \frac{DX}{\sqrt{n}}M^{-1}\Sigma_\eps M^{-1}\frac{X\trsp D}{\sqrt{n}}\alpha\;.
$$

We work under our usual assumptions, and in particular $A\succeq t\Id$.

We have the following theorem.

\begin{theorem}\label{thm:formsinDXMinvSigmaEpsMinvXtrspD}
Under the usual assumptions of this paper, when $\norm{\alpha}=1$, we have, for $K$ a constant,
\begin{gather*}
\var{G(\alpha;X)} \leq \sum_{i=1}^n V_i, \text{ with }\\
V_i\leq K b^2(A;\Sigma_\eps)\left\{\alpha_i^4 \left[\frac{R_i^4}{n^2} \frac{b_{Q_2}(2;X_i)}{t^2} \wedge 1\right]\right.\\
+\left[
\left(\alpha_i^4 \frac{R_i^4}{n^2}b_{Q_2}(2;X_i) \frac{1}{t^2}+\frac{R_i^4}{n^2} b_L(4;X_i)\frac{1}{t^2}+\alpha_i^2 \frac{R_i^2}{n} \frac{b_L(2;X_i)}{t} \right) \wedge 1 \right. \\
\left.\left.+\frac{R_i^{4}}{n^2}b_L(4;X_i)  \frac{1}{t^2}+\alpha_i^2 \frac{R_i^2}{n} b_L(2;X_i) \frac{1}{t}\right]\right\}\;.
\end{gather*}
Furthermore,
\begin{gather*}
\left|\Exp{G(\alpha;X)-G(\alpha;Y)}\right|\leq \sum_{i=1}^n U_{i,1}(X_i)+U_{i,2}(X_i)+U_{i,1}(Y_i)+U_{i,2}(Y_i)\;, \text{ where }\\
U_{i,1}(X_i)\leq
K \frac{b(A;\Sigma_\eps)}{\sqrt{t}} \left(\frac{R_i^2}{\sqrt{n}} \sqrt{\frac{b_{Q_2}(2;X_i)}{nt}} \wedge 1\right)
\left[\frac{|\alpha_i| R_i}{\sqrt{n}} \sqrt{b_L(2;X_i)}+\frac{R_i^2}{n} \sqrt{b_L(4;X_i)} \frac{1}{\sqrt{t}}\right]\;.\\
U_{i,2}(X_i)\leq K b(A;\Sigma_\eps) \left[(\alpha_i^2+\frac{R_i^2}{n}\frac{b_L(2;X_i)}{t})\right.\\
\wedge \sqrt{\frac{b_{Q_2}(2;X_i)}{n}} \left[\frac{R_i^2}{\sqrt{n}} \left(\frac{\alpha_i^2}{t}+\frac{R_i^2}{n}\frac{1}{t^2} \sqrt{b_L(4;X_i)}\right)
+ \frac{1}{t^{3/2}}\left(\frac{R_i^4}{n^{3/2}} \frac{1}{t^{1/2}} \sqrt{b_L(4;X_i)}+\frac{|\alpha_i| R_i^{3}}{n} \sqrt{b_L(2;X_i)}\right)\right]\;.
\end{gather*}
\end{theorem}

It is shown in the course of the proof that the upper bounds go to zero in probability
when $R_i$'s are i.i.d and uniformly square integrable
and $b_L(4;X_i)$ as well as $\sqrt{b_{Q_2}(2;X_i)/n}$ remain uniformly bounded.

We note that in the Gaussian case (i.e $X_i$ are ${\cal N}(0,\Sigma_i)$), by the symmetry trick we have now used several times, it is clear that the off-diagonal elements of the matrix
$$
\frac{DX}{\sqrt{n}}M^{-1}\Sigma_\eps M^{-1}\frac{X\trsp D}{\sqrt{n}}
$$
have mean 0. Hence, to understand $\Exp{G(\alpha;X)}$, all that is needed is to understand the diagonal entries of
$$
\frac{DX}{\sqrt{n}}M^{-1}\Sigma_\eps M^{-1}\frac{X\trsp D}{\sqrt{n}}\;.
$$
If we further assume that $X_i$ have the same $\Sigma$, computations similar to the ones done in Subsubsection \ref{subsec:quadFormsInverseAndMean} (also using our derivative trick) and fairly
standard random matrix results yield a reasonably simple expression. In the interest of space, and since this is a very simple problem, we do not state in more details the deterministic equivalent.

\begin{proof}

We use the same trick as in the previous subsection, namely calling
$$
g(\alpha;X;A+u\Sigma_\eps)=\alpha\trsp \frac{DX}{\sqrt{n}} (X\trsp D^2 X/n+A+u\Sigma_\eps)^{-1}\frac{X\trsp D}{\sqrt{n}}\alpha\;,
$$
we see that
$$
G(\alpha;X)=-\left.\frac{\partial}{\partial u}\right|_{u=0} g(\alpha;X;A+u\Sigma_\eps)\;.
$$
Hence we can use the refined understanding of $g$ we have developed earlier to study $G$.

In particular, the key equation in the study of $g$ was
$$
g(\alpha;X;A+u\Sigma_\eps)-g_i(\alpha;X;A+u\Sigma_\eps)=\alpha_i^2 -\frac{(\alpha_i-R_i\zeta_i(X_i)/\sqrt{n})^2}{1+\frac{R_i^2}{n} q_i(X_i)}\;.
$$
with, if $D_i=D-R_i e_i e_i\trsp$, (i.e $D_i$ is $D$ where we replace the $(i,i)$ entry by a 0)
\begin{align*}
q_i(X_i;u)&=X_i\trsp \left[M_i(A+u\Sigma_\eps)\right]^{-1} X_i\\
\zeta_i(X_i;u)&=X_i\trsp \left[M_i(A+u\Sigma_\eps)\right]^{-1} m_i, \;\; m_i=\frac{X\trsp D_i \alpha}{\sqrt{n}}\;.
\end{align*}

Hence, if $M_i=X\trsp D_i^2 X/n+A$,
\begin{align*}
\left.\frac{\partial}{\partial u}\right|_{u=0} q_i(X_i)&=-X_i\trsp M_i^{-1} \Sigma_\eps M_i^{-1} X_i\triangleq - \tilde{q}_i(X_i)\;,\\
\left.\frac{\partial}{\partial u}\right|_{u=0} \zeta_i(X_i)&=-X_i\trsp M_i^{-1} \Sigma_\eps M_i^{-1}m_i\triangleq -\tilde{\zeta}_i(X_i)\;.
\end{align*}
Hence, if $G_i(\alpha;X)$ is the same statistic as $G(\alpha;X)$ where $X_i$ is replaced by $0$ (and hence it does not depend on $X_i)$, we have
$$
G(\alpha;X)-G_i(\alpha;X)=2\frac{R_i}{\sqrt{n}} \frac{\zetaTilde_i(R_i\zeta_i/\sqrt{n}-\alpha_i)}{1+\frac{R_i^2}{n} q_i(X_i)}-\frac{R_i^2}{n} \qTilde_i \left(\frac{(\alpha_i-R_i\zeta_i(X_i;u)/\sqrt{n})}{1+\frac{R_i^2}{n} q_i(X_i;u)}\right)^2\;.
$$

In preparation for Lindeberg-style work below, we note that if $Y_i$ and $X_i$ have mean 0 and the same covariance $\Sigma_i$,
\begin{align*}
\Expj{i}{\zeta_i(X_i)}&=\Expj{i}{\zeta_i(Y_i)}\; \; \;
&\Expj{i}{\zeta^2_i(X_i)}&=\Expj{i}{\zeta^2_i(Y_i)}\\
\Expj{i}{\zeta_i(X_i)\zetaTilde_i(X_i)}&=\Expj{i}{\zeta_i(Y_i)\zetaTilde_i(Y_i)}\; \; \;
&\Expj{i}{\zetaTilde^2_i(X_i)}&=\Expj{i}{\zetaTilde^2_i(Y_i)}\;.
\end{align*}
Recall also that $\norm{M_i^{-1/2}m_i}\leq 1$, $\norm{M_i^{-1}m_i}\leq t^{-1/2}$, so $\norm{M_i^{-1}\Sigma_\eps M_i^{-1} m_i}\leq b(A;\Sigma_\eps)/\sqrt{t}$.
We therefore have the estimates
$$
\Exp{|\zetaTilde_i(X_i)|^k}\leq \left(\frac{b(A;\Sigma_\eps)}{\sqrt{t}}\right)^k b_L(k;X_i)\; \text{ and }
\Exp{|\zeta_i(X_i)|^k}\leq b_L(k;X_i)t^{-k/2}\;.
$$

Let us call, if $\psi_i(X_i)=(\alpha_i-\frac{R_i}{\sqrt{n}}\zeta_i(X_i))^2$,
\begin{align*}
E_1(X_i)&=\frac{R_i}{\sqrt{n}}\frac{\zetaTilde_i(X_i)(R_i/\sqrt{n} \zeta_i-\alpha_i)}{1+\frac{R_i^2}{n} q_i(X_i)}\\
E_2(X_i)&=\frac{R_i^2/n \, \,\qTilde_i(X_i)}{1+\frac{R_i^2}{n}q_i(X_i)}\frac{\psi_i(X_i)}{1+\frac{R_i^2}{n} q_i(X_i)}\;.
\end{align*}
Clearly,
$$
G(\alpha;X)-G_i(\alpha;X)=2E_1(X_i)-E_2(X_i)\;.
$$

$\bullet$ \textbf{\large{Efron-Stein aspects}}\\
The aim here is to find $Z_{i,1}$, independent of $X_i$ such that we can control $\Exp{|E_1(X_i)-Z_{i,1}|^2}$ and similarly for $E_2(X_i)$, we will try to find a $Z_{i,2}$ such that we control
$\Exp{|E_2(X_i)-Z_{i,2}|^2}$. This will give us control of $\var{G(\alpha;X)}$.\\
\textbf{1) Controlling $\mathbf{E_1(X_i)}$}
Let us call
$$
T_{1,i}=\frac{R_i}{\sqrt{n}}\frac{\alpha_i\zetaTilde_i(X_i)}{1+\frac{R_i^2}{n} q_i(X_i)}\;.
$$
Clearly, $T_{1,i}^2\leq \alpha_i^2 R_i^2/n \zetaTilde^2_i(X_i)$ and therefore
$$
\Exp{T_{1,i}^2}\leq \alpha_i^2 \frac{R_i^2}{n} b_L(2;X_i) \frac{b^2(A;\Sigma_\eps)}{t}\;.
$$
This term will not cause problem in our analysis as $\sum_{i=1}^n \Exp{T_{1,i}^2}$ will clearly go to zero when $R_i$'s
have two moments and $b(A;\Sigma_\eps)$ as well as $b_L(k;X_i)$ remain bounded. (Recall that $\norm{\alpha}=1$.)

Let us call
$$
T_{2,i}=\frac{R_i^2}{n} \frac{\zeta_i\zetaTilde_i}{1+\frac{R_i^2}{n}q_i(X_i)}\;.
$$
Clearly,
$$
\Expj{i}{|T_{2,i}|^k}\leq \frac{R_i^{2k}}{n^k} \sqrt{\Expj{i}{\zeta_i^{2k}}\Expj{i}{\zetaTilde_i^{2k}}}\leq \frac{R_i^{2k}}{n^k}b_L(2k;X_i)  \frac{b^k(A;\Sigma_\eps)}{t^k}\;.
$$
In particular,
$$
\Exp{|T_{2,i}|^2}\leq  \frac{R_i^{4}}{n^2}b_L(4;X_i)  \frac{b^2(A;\Sigma_\eps)}{t^2}\;.
$$
(Since, when $R_i$'s are i.i.d and have two moments,
$\sum_i R_i^4/n^2\tendsto 0$ a.s, this terms is again not going to create any problems when we try to control the variance of $G$.)

So we have shown that
$$
\Expj{i}{E_1(X_i)^2}\leq 2\left[\frac{R_i^{4}}{n^2}b_L(4;X_i)  \frac{1}{t}+\alpha_i^2 \frac{R_i^2}{n} b_L(2;X_i)\right]\frac{b^2(A;\Sigma_\eps)}{t}\;.
$$

\textbf{2) Controlling $\mathbf{E_2(X_i)}$}
Recall the decomposition from the proof of Theorem \ref{thm:ThirdFormIsDeterministic}
$$
\frac{\psi_i(X_i)}{1+\frac{R_i^2}{n} q_i(X_i)}-\frac{\alpha_i^2}{1+\frac{R_i^2}{n}d_i}=\alpha_i^2 \frac{R_i^2}{n} \frac{q_i(X_i)-d_i}{(1+\frac{R_i^2}{n}q_i)(1+\frac{R_i^2}{n}d_i)}+\frac{1}{1+\frac{R_i^2}{n}q_i}(R_i^2/n\zeta_i^2-2\alpha_iR_i/\sqrt{n} \zeta_i)\;.
$$
Let us call
$$
\Delta_{2,i}(X_i)\triangleq \alpha_i^2 \frac{R_i^2}{n} \frac{q_i(X_i)-d_i}{(1+\frac{R_i^2}{n}q_i)(1+\frac{R_i^2}{n}d_i)}+\frac{1}{1+\frac{R_i^2}{n}q_i}(R_i^2/n\zeta_i^2-2\alpha_iR_i/\sqrt{n} \zeta_i)\;.
$$
Let us note that since $|\frac{R_i^2/n\, \qTilde_i(X_i)}{1+\frac{R_i^2}{n}q_i(X_i)}|\leq b(A;\Sigma_\eps)$, we will have, using the work
we did in the proof of Theorem \ref{thm:ThirdFormIsDeterministic},
$$
\Expj{i}{\left[\frac{R_i^2/n\, \qTilde_i(X_i)}{1+\frac{R_i^2}{n}q_i(X_i)} \Delta_{2,i}(X_i)\right]^2}\leq K b^2(A;\Sigma_\eps)\left[
\left(\alpha_i^4 \frac{R_i^4}{n^2}b_{Q_2}(2;X_i) \frac{1}{t^2}+\frac{R_i^4}{n^2} b_L(4;X_i)\frac{1}{t^2}+\alpha_i^2 \frac{R_i^2}{n} \frac{b_L(2;X_i)}{t} \right) \wedge 1 \right] \;.
$$
As we saw then, these terms will not cause any problem in our eventual control of the variance.

So we just need to focus on understanding
$$
{\cal R}_{2,i}(X_i)=\frac{R_i^2/n\, \qTilde_i(X_i)}{1+\frac{R_i^2}{n} q_i(X_i)}\frac{\alpha_i^2}{1+\frac{R_i^2}{n}d_i}\;.
$$
Now recall that we called $\frac{R_i^2}{n} \delta_i(X_i)=1/(1+\frac{R_i^2}{n}q_i(X_i))-1/(1+\frac{R_i^2}{n}d_i)$; with this notation, we have
$$
\frac{R_i^2/n\, \qTilde(X_i)}{1+\frac{R_i^2}{n}q_i(X_i)}=\frac{R_i^2/n\, \tilde{d}_i}{1+\frac{R_i^2}{n}d_i}+\frac{R_i^2}{n} \frac{\qTilde_i(X_i)-\tilde{d}_i}{1+\frac{R_i^2}{n} q_i(X_i)}+\frac{R_i^2}{n} \tilde{d}_i \frac{R_i^2}{n}\delta_i(X_i)\;.
$$

Hence,
$$
{\cal R}_{2,i}(X_i)-\frac{R_i^2/n\, \tilde{d}_i}{1+\frac{R_i^2}{n}d_i}\frac{\alpha_i^2}{1+\frac{R_i^2}{n}d_i}=
\frac{\alpha_i^2}{1+\frac{R_i^2}{n}d_i}\left[\frac{R_i^2}{n} \frac{\qTilde_i(X_i)-\tilde{d}_i}{1+\frac{R_i^2}{n} q_i(X_i)}+\frac{R_i^2}{n} \tilde{d}_i \frac{R_i^2}{n}\delta_i(X_i)\right]\;.
$$

Since $\tilde{d}_i(X_i)\leq b(A;\Sigma_\eps) d_i$, we have
$$
\frac{R_i^2}{n} \frac{\tilde{d}_i}{1+\frac{R_i^2}{n} d_i}\leq b(A;\Sigma_\eps)\;.
$$

We have seen that
$$
\Expj{i}{|\qTilde_i(X_i)-\tilde{d}_i(X_i)|^2}\leq \opnorm{M_i^{-1}\Sigma_\eps M_i^{-1}}^2 b_{Q_2}(2;X_i)\leq \frac{b^2(A;\Sigma_\eps)}{t^2}b_{Q_2}(2;X_i)\;.
$$
Furthermore, we saw previously that
$$
\Expj{i}{\delta_i^2(X_i)}\leq \frac{1}{t^2} b_{Q_2}(2;X_i)\;.
$$
So we have
$$
\Expj{i}{\left[{\cal R}_{2,i}(X_i)-\frac{R_i^2 \tilde{d}_i}{1+\frac{R_i^2}{n}d_i}\frac{\alpha_i^2}{1+\frac{R_i^2}{n}d_i}\right]^2}\leq
K\alpha_i^4 \frac{R_i^4}{n^2} \frac{b_{Q_2}(2;X_i)}{t^2} b^2(A;\Sigma_\eps)\;.
$$
On the other hand,
$$
\left|{\cal R}_{2,i}(X_i)-\frac{R_i^2/n\, \tilde{d}_i}{1+\frac{R_i^2}{n}d_i}\frac{\alpha_i^2}{1+\frac{R_i^2}{n}d_i}\right|\leq K b(A;\Sigma_\eps) \alpha_i^2 \;.
$$
So we conclude that
$$
\Expj{i}{\left[{\cal R}_{2,i}(X_i)-\frac{R_i^2 \tilde{d}_i}{1+\frac{R_i^2}{n}d_i}\frac{\alpha_i^2}{1+\frac{R_i^2}{n}d_i}\right]^2}\leq
K\alpha_i^4 b^2(A;\Sigma_\eps) \left[\frac{R_i^4}{n^2} \frac{b_{Q_2}(2;X_i)}{t^2}\wedge 1 \right]\;.
$$

Hence,
\begin{gather*}
\Expj{i}{\left[E_2(X_i)-\frac{R_i^2 \tilde{d}_i}{1+\frac{R_i^2}{n}d_i}\frac{\alpha_i^2}{1+\frac{R_i^2}{n}d_i}\right]^2}\leq
K b^2(A;\Sigma_\eps) \left\{\alpha_i^4 \left[\frac{R_i^4}{n^2} \frac{b_{Q_2}(2;X_i)}{t^2} \wedge 1\right]\right.\\
\left.+\left[
\left(\alpha_i^4 \frac{R_i^4}{n^2}b_{Q_2}(2;X_i) \frac{1}{t^2}+\frac{R_i^4}{n^2} b_L(4;X_i)\frac{1}{t^2}+\alpha_i^2 \frac{R_i^2}{n} \frac{b_L(2;X_i)}{t} \right) \wedge 1 \right] \right\}\;.
\end{gather*}

So if
$$
Z_i=G_i(\alpha;X)-\frac{R_i^2 \tilde{d}_i}{1+\frac{R_i^2}{n}d_i}\frac{\alpha_i^2}{1+\frac{R_i^2}{n}d_i}\;,
$$
and $Z=G(\alpha;X)$,
we have shown that
\begin{gather*}
\Expj{i}{|Z-Z_i|^2}\leq K b^2(A;\Sigma_\eps)\left\{\alpha_i^4 \left[\frac{R_i^4}{n^2} \frac{b_{Q_2}(2;X_i)}{t^2} \wedge 1\right]\right.\\
+\left[
\left(\alpha_i^4 \frac{R_i^4}{n^2}b_{Q_2}(2;X_i) \frac{1}{t^2}+\frac{R_i^4}{n^2} b_L(4;X_i)\frac{1}{t^2}+\alpha_i^2 \frac{R_i^2}{n} \frac{b_L(2;X_i)}{t} \right) \wedge 1 \right. \\
\left.\left.+\frac{R_i^{4}}{n^2}b_L(4;X_i)  \frac{1}{t^2}+\alpha_i^2 \frac{R_i^2}{n} b_L(2;X_i) \frac{1}{t}\right]\right\}\;.
\end{gather*}
This bound is sufficient to allow us to apply the Efron-Stein inequality, as we saw earlier: as soon as the
$R_i$'s have two moments and are i.i.d, the (sum over $i$ of the) upper bound goes to zero.

$\bullet$ \textbf{Lindeberg aspects}\\
\textbf{1) Controlling }$\mathbf{\Exp{E_1(X_i)-E_1(Y_i)}}$\\
To alleviate the notation, we make a slight abuse of notation and change the meaning of $M_i$ compared to what was used
in the previous part of the proof: because we are now in the Lindeberg setting,
the matrix $M_i$ is (as usual) computed by using $(X_1,\ldots,X_{i-1},0,Y_{i+1},\ldots,Y_n)$, but it is still independent of $X_i$ and $Y_i$.

Recall that
$$
E_1(X_i)=\frac{R_i}{\sqrt{n}}\frac{\zetaTilde_i(X_i)(R_i/\sqrt{n} \zeta_i-\alpha_i)}{1+\frac{R_i^2}{n} q_i(X_i)}\triangleq \frac{N_i(X_i)}{1+\frac{R_i^2}{n} q_i(X_i)}\;.
$$

Calling as usual
$$
\frac{R_i^2}{n}\delta_i(X_i)=\frac{1}{1+\frac{R_i^2}{n} q_i(X_i)}-\frac{1}{1+\frac{R_i^2}{n} d_i}\;,
$$
we have
$$
E_1(X_i)=\frac{N_i(X_i)}{1+\frac{R_i^2}{n}d_i}+\frac{R_i^2}{n} N_i(X_i)\delta_i(X_i)\;.
$$
Of course $\Exp{N_i(X_i)}=\Exp{N_i(Y_i)}$ when $X_i$ and $Y_i$ have mean 0 and the same covariance, $\Sigma_i$.
Therefore, since $d_i=\trace{\Sigma_iM_i^{-1}}=\Expj{i}{X_i\trsp M_i^{-1} X_i}=\Expj{i}{Y_i\trsp M_i^{-1} Y_i}$, we have
$$
\Exp{\frac{N_i(X_i)}{1+\frac{R_i^2}{n}d_i}}=\Exp{\frac{N_i(Y_i)}{1+\frac{R_i^2}{n}d_i}}\;.
$$
The only question now is to try to control the remainder term
$$
{\cal R}_1(X_i)=\frac{R_i^2}{n} N_i(X_i)\delta_i(X_i)=E_1(X_i)\frac{R_i^2}{n}\frac{(d_i-q_i(X_i))}{1+\frac{R_i^2}{n} d_i(X_i)} \;.
$$

We have, after using Cauchy-Schwarz and our usual bounds,
$$
\Expj{i}{\left|{\cal R}_1(X_i)\right|}\leq \frac{R_i^2}{\sqrt{n}} \sqrt{\frac{b_{Q_2}(2;X_i)}{nt}} \sqrt{\Expj{i}{E_1^2(X_i)}}\;.
$$
Using the results we got earlier on $\sqrt{\Expj{i}{E_1^2(X_i)}}$, we have
$$
\Expj{i}{\left|{\cal R}_1(X_i)\right|}\leq K \frac{R_i^2}{\sqrt{n}} \sqrt{\frac{b_{Q_2}(2;X_i)}{nt}}
\left(\frac{R_i^2}{n} \frac{1}{\sqrt{t}} \sqrt{b_L(4;X_i)}+|\alpha_i| \frac{R_i}{\sqrt{n}} \sqrt{b_L(2;X_i)}\right)
\frac{b(A;\Sigma_\eps)}{\sqrt{t}}\;.
$$
On the other hand,
$$
\left|{\cal R}_1(X_i)\right|\leq |N_i(X_i)|\;.
$$
From its definition, we see that
$$
\Expj{i}{|N_i(X_i)|}\leq  \frac{|\alpha_i| R_i}{\sqrt{n}} \frac{b(A;\Sigma_\eps)}{\sqrt{t}} b_L(1;X_i)
+\frac{R_i^2}{n} b_L(2;X_i) \frac{b(A;\Sigma_\eps)}{t}\;.
$$
Hence,
\begin{align*}
\Expj{i}{\left|{\cal R}_1(X_i)\right|}&\leq
K \frac{R_i^2}{\sqrt{n}} \frac{b(A;\Sigma_\eps)}{\sqrt{t}} \sqrt{\frac{b_{Q_2}(2;X_i)}{nt}} \left(\frac{R_i^2}{n} \frac{\sqrt{b_L(4;X_i)}}{\sqrt{t}}+|\alpha_i| \frac{R_i}{\sqrt{n}} \sqrt{b_L(2;X_i)}\right)\\
&\wedge
\frac{b(A;\Sigma_\eps)}{\sqrt{t}}\left( \frac{|\alpha_i| R_i}{\sqrt{n}}  b_L(1;X_i)+\frac{R_i^2}{n} b_L(2;X_i) \frac{1}{\sqrt{t}}\right)\;,
\end{align*}
and
\begin{align*}
\Expj{i}{\left|{\cal R}_1(X_i)\right|}&\leq K \frac{b(A;\Sigma_\eps)}{\sqrt{t}} \left(\frac{R_i^2}{\sqrt{n}} \sqrt{\frac{b_{Q_2}(2;X_i)}{nt}} \wedge 1\right)
\left[\frac{|\alpha_i| R_i}{\sqrt{n}} \sqrt{b_L(2;X_i)}+\frac{R_i^2}{n} \sqrt{b_L(4;X_i)} \frac{1}{\sqrt{t}}\right]\;.
\end{align*}

At this point, we would like to show that the control we have is sufficient for the Lindeberg method to work when $R_i$'s are i.i.d and
have two moments. For this, it is sufficient to show that
$$
\Exp{\sum_{i=1}^n |\alpha_i| (R_i/\sqrt{n}\wedge R_i^3/n)}\tendsto 0\;.
$$
We will simply show that $\Exp{(R_i/\sqrt{n}\wedge R_i^3/n)}=\lo(n^{-1/2})$. We note that, since $R_i\geq 0$,
$$
\Exp{R_i\wedge n^{-1/2} R_i^3}=\Exp{R_i 1_{R_i\geq n^{1/4}}}+\Exp{R_i^3n^{-1/2}1_{R_i\leq n^{1/4}}}\leq \Exp{R_i 1_{R_i\geq n^{1/4}}}+n^{-1/4}\Exp{R_i^2}\;.
$$
Since $R_i$ has two moments (and hence one), the monotone convergence theorem guarantees  that
$$
\Exp{R_i\wedge n^{-1/2} R_i^3}=\lo(1)\;.
$$
We now remark that since $\norm{\alpha}=1$, $\norm{\alpha}_1\leq \sqrt{n}$. Therefore,
$$
\Exp{\sum_{i=1}^n |\alpha_i| (R_i/\sqrt{n}\wedge R_i^3/n)}=\norm{\alpha}_1 \Exp{(R_i/\sqrt{n}\wedge R_i^3/n)}=\lo(\norm{\alpha}_1/n^{1/2})=\lo(1)\;.
$$
\textbf{2) Controlling $\mathbf{\Exp{E_2(X_i)-E_2(Y_i)}}$}\\
Recall the notation
$$
\Delta_i=\frac{1}{(1+\frac{R_i^2}{n}q_i(X_i))^2}-\frac{1}{(1+\frac{R_i^2}{n}d_i)^2}\;.
$$
Let us write
$$
\psi_i(X_i)=\alpha_i^2 - 2\alpha_i \frac{R_i}{\sqrt{n}} \zeta_i+ \frac{R_i^2}{n} \zeta_i^2=\alpha_i^2 - \Gamma_i\;.
$$
By definition,
$$
E_2(X_i)=\frac{R_i^2\tilde{q}_i(X_i)/n}{(1+\frac{R_i^2}{n}q_i(X_i))^2}  \psi_i(X_i)\;.
$$
Therefore,
\begin{align*}
E_{2}(X_i)&=\frac{\alpha_i^2 \qTilde_i(X_i)-\Gamma_i \tilde{d}_i}{(1+\frac{R_i^2}{n}d_i)^2}\frac{R_i^2}{n}+\frac{R_i^2}{n}\frac{\Gamma_i(\tilde{d}_i-\tilde{q}_i)}{(1+\frac{R_i^2}{n}d_i)^2}
+\Delta_i(X_i) \psi_i(X_i) \qTilde_i(X_i) \frac{R_i^2}{n}\;,\\
&={\cal M}_2(X_i)+{\cal R}_{2,1}(X_i)+{\cal R}_{2,2}(X_i)\;.
\end{align*}
It is clear that when $X_i$ and $Y_i$ have the same covariance $\Sigma_i$ and mean 0, $\Expj{i}{{\cal M}_2(X_i)}=\Expj{i}{{\cal M}_2(Y_i)}$. Hence, in controlling $\Exp{E_2(X_i)-E_2(Y_i)}$, all we will have to do is control
$$
\Expj{i}{|{\cal R}_{2,1}(X_i)+{\cal R}_{2,2}(X_i)|}\;.
$$
\textbf{a) Controlling $\mathbf{\Exp{|{\cal R}_{2,1}(X_i)|}}$}\\
We have, using the Cauchy-Schwarz inequality:
\begin{align*}
\Expj{i}{\zeta_i^2 |\tilde{d}_i-\qTilde_i(X_i)|}\leq \frac{b(A;\Sigma_\eps)}{t^2} \sqrt{b_L(4;X_i)}\sqrt{b_{Q_2}(2;X_i)}\;,\\
\Expj{i}{|\zeta_i| |\tilde{d}_i-\qTilde_i(X_i)|}\leq \frac{b(A;\Sigma_\eps)}{t^{3/2}}\sqrt{b_L(2;X_i)}\sqrt{b_{Q_2}(2;X_i)}\;.
\end{align*}
Therefore,
$$
\frac{R_i^4}{n^2} \Expj{i}{\frac{\zeta_i^2 |\tilde{d}_i-\qTilde_i(X_i)|}{(1+\frac{R_i^2}{n}d_i)^2}}\leq \frac{R_i^4}{n^{3/2}} \frac{b(A;\Sigma_\eps)}{t^2} \sqrt{b_L(4;X_i)}\sqrt{b_{Q_2}(2;X_i)/n}\;.
$$
and
$$
|\alpha_i|\frac{R_i^3}{n^{3/2}} \Expj{i}{\frac{|\zeta_i| |\tilde{d}_i-\qTilde_i(X_i)|}{(1+\frac{R_i^2}{n}d_i)^2}}\leq \frac{|\alpha_i| R_i^{3}}{n} \frac{b(A;\Sigma_\eps)}{t^{3/2}}\sqrt{b_L(2;X_i)}\sqrt{b_{Q_2}(2;X_i)/n}
$$
We conclude that
$$
\Expj{i}{|{\cal R}_{2,1}(X_i)|}\leq K \frac{b(A;\Sigma_\eps)}{t^{3/2}}\sqrt{b_{Q_2}(2;X_i)/n}\left(\frac{R_i^4}{n^{3/2}} \frac{1}{t^{1/2}} \sqrt{b_L(4;X_i)}+\frac{|\alpha_i| R_i^{3}}{n} \sqrt{b_L(2;X_i)}\right)\;.
$$
\textbf{b) Controlling $\mathbf{\Exp{|{\cal R}_{2,2}(X_i)|}}$}\\
Since
$$
\left|\frac{R_i^2}{n}\qTilde_i(X_i)\Delta_i(X_i)\right|\leq K b(A;\Sigma_\eps) \frac{R_i^2}{n} \frac{|q_i-d_i|}{1+\frac{R_i^2}{n}d_i}\;,
$$
we have
$$
\Expj{i}{\left|\Delta_i(X_i)\psi_i(X_i)\frac{R_i^2}{n}\qTilde_i(X_i)\right|}\leq K b(A;\Sigma_\eps) \Expj{i}{(\alpha_i^2+\frac{R_i^2}{n}\zeta_i^2) \frac{R_i^2}{n} \frac{|q_i-d_i|}{1+\frac{R_i^2}{n}d_i}}\;,
$$
Hence,
$$
\Expj{i}{\left|\Delta_i(X_i)\psi_i(X_i)\frac{R_i^2}{n}\qTilde_i(X_i)\right|}\leq K b(A;\Sigma_\eps)  \frac{R_i^2}{\sqrt{n}}\sqrt{b_{Q_2}(2;X_i)/n} \left(\frac{\alpha_i^2}{t}+\frac{R_i^2}{n}\frac{1}{t^2} \sqrt{b_L(4;X_i)}\right)\;.
$$
\textbf{c) Controlling $|\Expj{i}{E_{2}(X_i)-E_2(Y_i)}|$}\\
We note that $|E_2(X_i)|\leq \psi_i(X_i)b(A;\Sigma_\eps)$ and therefore
$$
\Expj{i}{|E_2(X_i)|}\leq K b(A;\Sigma_\eps) (\alpha_i^2+\frac{R_i^2}{n}\frac{b_L(2;X_i)}{t})\;.
$$
We can finally conclude that
$$
|\Expj{i}{E_{2}(X_i)-E_2(Y_i)}|\leq \Phi_1(X_i)+\Phi_1(Y_i)\;,
$$
where
\begin{gather*}
\Phi_1(X_i)\leq K b(A;\Sigma_\eps) \left[(\alpha_i^2+\frac{R_i^2}{n}\frac{b_L(2;X_i)}{t})\right.\\
\wedge \sqrt{\frac{b_{Q_2}(2;X_i)}{n}} \left[\frac{R_i^2}{\sqrt{n}} \left(\frac{\alpha_i^2}{t}+\frac{R_i^2}{n}\frac{1}{t^2} \sqrt{b_L(4;X_i)}\right)\right.\\
\left.+ \frac{1}{t^{3/2}}\left(\frac{R_i^4}{n^{3/2}} \frac{1}{t^{1/2}} \sqrt{b_L(4;X_i)}+\frac{|\alpha_i| R_i^{3}}{n} \sqrt{b_L(2;X_i)}\right)\right]\;.
\end{gather*}
This expression is somewhat unseemly, however, assuming that $b_L$, $b_{Q_2}/n$ and $b(A;\Sigma_\eps)$ stay bounded
we see that it is of the form
$$
\Phi_1(X_i)\leq K \left(\frac{R_i^2}{\sqrt{n}}(\frac{R_i^2}{n}+\alpha_i^2+\frac{|\alpha_i|R_i}{\sqrt{n}})\right)\wedge (\alpha_i^2+\frac{R_i^2}{n})\leq K \left(\alpha_i^2+\frac{R_i^2}{n}\right)\left(\frac{R_i^2}{\sqrt{n}}\wedge 1\right)\;.
$$
We have already seen how to control this expression when $R_i$ are i.i.d and uniformly square integrable
in the proof of Theorem \ref{thm:SecondFormIsDeterministic}.
So we conclude that when this is the case $\sum_{i=1}^n \Phi_i(X_i)$ will tend to 0 (for instance in $R_i$-probability).

\end{proof}

\subsubsection{Forms in $\frac{1}{\sqrt{n}}\alpha\trsp DX\trsp M^{-1} \Sigma_{\eps} M^{-1}x$}\label{subsubsec:formsinAlphaDXMinvSigmaEpsMinvx}
The third and last situation we need to consider are forms of the type
$$
H(\alpha;X)=\frac{1}{\sqrt{n}}\alpha\trsp DX\trsp M^{-1} \Sigma_{\eps} M^{-1}x\;,
$$
where as usual
$$
M=\frac{1}{n}X\trsp D^2 X+A\;.
$$
We work under our usual assumptions, and in particular $A\succeq t\Id$.

We have the following theorem.

\begin{theorem}\label{thm:FormsInformsinAlphaDXMinvSigmaEpsMinvx}
Under the usual assumptions of this paper (see Subsection \ref{subsec:SetupOfStudy}), we have, for $K$ a constant,
\begin{align*}
&\var{H(\alpha;X)} \leq \sum_{i=1}^n V_i, \text{ with }\\
&V_i\leq K b^2(A;\Sigma_\eps)\frac{x\trsp A^{-1} x}{t}\left[\frac{R_i^2}{n}\alpha_i^2 b_L(2;X_i)+\frac{R_i^4}{n^2}b_L(4;X_i)\frac{1}{t}\right]\;.\\
&\left|\Exp{H(\alpha;X)-H(\alpha;Y)}\right|\leq \sum_{i=1}^n U_i(X_i)+U_i(Y_i)\;, \text{ where }\\
&U_i(X_i)\leq K b(A;\Sigma_\eps)\sqrt{\frac{x\trsp A^{-1}x}{t}}\left[\frac{1}{\sqrt{n}}|\alpha_i|R_i\sqrt{b_L(2;X_i)}+\frac{R_i^2}{n}\sqrt{\frac{b_L(4;X_i)}{t}}\right]
\left[\frac{R_i^2}{n} \frac{\sqrt{b_{Q_2}(2;X_i)}}{t}\wedge 1\right]
\end{align*}
\end{theorem}
The proof of the theorem uses the same ideas as before and will rely on the work of Subsubsection \ref{subsubsec:quadFormsInverseMeanAndPopMean}.

We also note that by the same symmetry arguments as before, in the Gaussian case, we trivially have $\Exp{H(\alpha;X)}=0$.

\begin{proof}
Naturally, $H(\alpha;X)$ is closely related to
$$
h(\alpha;X)=\frac{1}{\sqrt{n}}\alpha\trsp DX\trsp M^{-1}x
$$
which we studied earlier.
Recall that we got the key decomposition
$$
h(\alpha;X;A)=h(\alpha;X)=h_i(\alpha;X)+\frac{\varphi_i(X_i)}{1+\frac{R_i^2}{n}q_i(X_i)}\;,
$$
where $h_i$ did not involve $X_i$ and
\begin{align*}
\varphi_i(X_i)&=\frac{1}{\sqrt{n}}\alpha_i R_i X_i\trsp M_i^{-1} x-\frac{R_i^2}{n}X_i\trsp M_i^{-1}m_i X_i\trsp M_i^{-1} x\;,\\
q_i(X_i)&=X_i\trsp M_i^{-1} X_i\;.
\end{align*}

As before, we can deduce $H$ from $h(\alpha;X;A+u\Sigma_\eps)$ by taking the derivative of the latter with respect to $u$ and appropriately modifying the sign.

We call
\begin{align*}
H_i&=-\frac{\partial h_i(\alpha;X;A+u\Sigma_\eps)}{\partial u}\\
\Upsilon_i(X_i)&\triangleq -\frac{\partial \varphi_i(X_i;A+u\Sigma_\eps)}{\partial u}\\
&=\frac{R_i}{\sqrt{n}}\alpha_i X_i\trsp M_i^{-1}\Sigma_\eps M_i^{-1} x-\frac{R_i^2}{n}\left[X_i\trsp M_i^{-1} \Sigma_\eps M_i^{-1} m_i X_i\trsp M_i^{-1} x+X_i\trsp M_i^{-1}m_i X_i\trsp M_i^{-1}\Sigma_\eps M_i^{-1} x\right]\\
\qTilde_i(X_i)&=X_i\trsp M_i^{-1}\Sigma_\eps M_i^{-1} X_i \;.
\end{align*}

The new ``key equality" is
$$
H(\alpha;X)=H_i(\alpha;X)-\frac{\Upsilon_i}{1+\frac{R_i^2}{n}q_i(X_i)}+\frac{R_i^2}{n}\qTilde_i(X_i)\frac{\varphi_i(X_i)}{(1+\frac{R_i^2}{n}q_i(X_i))^2}\;.
$$

We are now in a position to do our usual analysis with the Efron-Stein inequality and the Lindeberg approach.

$\bullet$ \textbf{Efron-Stein aspects}
Because $H_i(\alpha;X)$ does not involve $X_i$, we clearly have
$$
\var{H(\alpha;X)}\leq \sum_{i=1}^n \var{H(\alpha;X)-H_i(\alpha;X_i)}\;.
$$
Now, clearly,
$$
V_i=\var{H(\alpha;X)-H_i(\alpha;X_i)}\leq K\left[\Exp{\left(\frac{\Upsilon_i(X_i)}{1+\frac{R_i^2}{n}q_i(X_i)}\right)^2}+\frac{R_i^4}{n^2}\Exp{\frac{\qTilde_i(X_i)^2 \varphi_i^2(X_i)}{(1+\frac{R_i^2}{n}q_i(X_i))^2}}\right]\;.
$$
If $v=M_i^{-1}\Sigma_\eps M_i^{-1} x$, we have seen that $\norm{v}\leq b(A;\Sigma_\eps) \sqrt{x\trsp A^{-1} x/t}$.So we conclude that
$$
\Exp{\left(\frac{R_i}{\sqrt{n}}\alpha_i X_i\trsp M_i^{-1}\Sigma_\eps M_i^{-1} x\right)^2}\leq \frac{R_i^2}{n} \alpha_i^2 b_L(2;X_i) b^2(A;\Sigma_\eps) \frac{x\trsp A^{-1} x}{t}\;.
$$
Recall now that $\norm{M_i^{-1}m_i}\leq \frac{\norm{\alpha}}{\sqrt{t}}\leq \frac{1}{\sqrt{t}}$ and $\norm{M_i^{-1}\Sigma_\eps M_i^{-1} m_i}\leq \frac{b(A;\Sigma_\eps)}{\sqrt{t}}$. Therefore,
$$
\Exp{\left[X_i\trsp M_i^{-1} \Sigma_\eps M_i^{-1} m_i X_i\trsp M_i^{-1} x+X_i\trsp M_i^{-1}m_i X_i\trsp M_i^{-1}\Sigma_\eps M_i^{-1} x\right]^2}\leq K b_L(4;X_i) \frac{b^2(A;\Sigma_\eps)}{t}\frac{x\trsp A^{-1} x}{t}\;.
$$
We finally have
$$
\Exp{\left(\frac{\Upsilon_i(X_i)}{1+\frac{R_i^2}{n}q_i(X_i)}\right)^2}\leq K \frac{b^2(A;\Sigma_\eps) x\trsp A^{-1} x}{t}\left[\frac{R_i^2}{n}\alpha_i^2 b_L(2;X_i)+\frac{R_i^4}{n^2}b_L(4;X_i)\frac{1}{t}\right]\;.
$$

For the second part of this simple variance bounding exercise, we first remind the reader that
$$
\left|\frac{\frac{R_i^2}{n} \qTilde_i(X_i)}{1+\frac{R_i^2}{n}q_i(X_i)}\right|\leq b(A;\Sigma_\eps)\;.
$$
Hence, we simply need to bound
$$
\Exp{\left(\frac{\varphi_i(X_i)}{1+\frac{R_i^2}{n}q_i(X_i)}\right)^2}\;,
$$
something we have essentially already done, and we get easily
$$
\Exp{\left(\frac{\varphi_i(X_i)}{1+\frac{R_i^2}{n}q_i(X_i)}\right)^2}\leq K\left[\alpha_i^2 \frac{R_i^2}{n} b_L(2;X_i) \frac{x\trsp A^{-1} x}{t} +\frac{R_i^4}{n^2}\frac{x\trsp A^{-1}x}{t^2} b_L(4;X_i)\right]\;.
$$

We have therefore shown that
$$
V_i\leq K b^2(A;\Sigma_\eps)\frac{x\trsp A^{-1} x}{t}\left[\frac{R_i^2}{n}\alpha_i^2 b_L(2;X_i)+\frac{R_i^4}{n^2}b_L(4;X_i)\frac{1}{t}\right]\;.
$$

We note that when $R_i$ are i.i.d uniformly square integrable, the Marcienkiewicz-Zygmund law of large numbers guarantees that $\sum_{i=1}^n V_i\tendsto 0$ for instance in probability.

We now turn to Lindeberg-type questions.\\
$\bullet$ \textbf{Lindeberg aspects} As usual, we will go a bit fast here. Essentially the previous decomposition can still be used, but it should now be understood that the $M_i$ matrix we are dealing with involves both $\{X_m\}_{m<i}$ and $\{Y_k\}_{k>i}$, instead of just $\{X_j\}_{j=1}^n$ or $\{Y_j\}_{j=1}^n$. However, the key fact is that $M_i$ is independent of both $X_i$ and $Y_i$. Hence, we have for instance
$$
\Exp{\varphi_i(X_i)}=\Exp{\varphi_i(Y_i)}
$$
and
$$
\Exp{\Upsilon_i(X_i)}=\Exp{\Upsilon_i(Y_i)}\;.
$$

Let us call
$$
T_1(X_i)=\frac{\Upsilon_i(X_i)}{1+\frac{R_i^2}{n}q_i(X_i)}
$$
and
$$
T_2(X_i)=\frac{R_i^2}{n}\frac{\qTilde_i(X_i)\varphi_i(X_i)}{(1+\frac{R_i^2}{n}q_i(X_i))^2}\;.
$$
It is clear that if we can control
$$
\sum_{i=1}^n \left|\Exp{T_1(X_i)-T_1(Y_i)}-\Exp{T_2(X_i)-T_2(Y_i)}\right|
$$
we will have control over $|\Exp{H(\alpha;X)-H(\alpha;Y)}|$.
We recall that we have already showed that
$$
\left|H(\alpha;X)\right|\leq K \frac{\norm{x}b(A;\Sigma_\eps)}{\sqrt{t}}\;.
$$
\\

$\bullet$ \textbf{Control of }$\mathbf{\Exp{T_1(X_i)-T_1(Y_i)}}$\\
As usual, we use the fact that
\begin{align*}
T_1(X_i)&=\frac{\Upsilon_i(X_i)}{1+\frac{R_i^2}{n}d_i}+\Upsilon_i(X_i) \frac{R_i^2}{n}\frac{d_i-q_i(X_i)}{(1+\frac{R_i^2}{n}q_i(X_i))(1+\frac{R_i^2}{n}d_i)}\\
&\triangleq T_{1,1}(X_i)+T_{1,2}(X_i)\;.
\end{align*}
Naturally, since $X_i$ and $Y_i$ have mean 0 and the same covariance,
$$
\Exp{\frac{\Upsilon_i(X_i)}{1+\frac{R_i^2}{n}d_i}}=\Exp{\frac{\Upsilon_i(Y_i)}{1+\frac{R_i^2}{n}d_i}}\;,
$$
so all that is left to do is control $\Exp{|T_{1,2}(X_i)|}$. To do so, we can use Cauchy-Schwarz and recall that
$$
\Exp{\left(\frac{\Upsilon_i(X_i)}{1+\frac{R_i^2}{n}q_i(X_i)}\right)^2}\leq K \frac{b^2(A;\Sigma_\eps) x\trsp A^{-1} x}{t}\left[\frac{R_i^2}{n}\alpha_i^2 b_L(2;X_i)+\frac{R_i^4}{n^2}b_L(4;X_i)\frac{1}{t}\right]\;.
$$
and
$$
\Exp{(d_i-q_i)^2}\leq \frac{b_{Q_2}(2;X_i)}{t^2}\;.
$$
Hence,
\begin{align*}
\Exp{|T_{1,2}(X_i)|}&\leq K \frac{R_i^2}{n}\sqrt{\frac{b^2(A;\Sigma_\eps) x\trsp A^{-1} x}{t}\left[\frac{R_i^2}{n}\alpha_i^2 b_L(2;X_i)+\frac{R_i^4}{n^2}b_L(4;X_i)\frac{1}{t}\right]}\frac{\sqrt{b_{Q_2}(2;X_i)}}{t}\\
&\leq K \frac{R_i^2}{n}\frac{b(A;\Sigma_\eps)\sqrt{x\trsp A^{-1} x}}{t^{3/2}} \left[\frac{R_i}{\sqrt{n}}|\alpha_i|\sqrt{b_L(2;X_i)}+\frac{R_i^2}{n}\sqrt{\frac{b_L(4;X_i)}{t}}\right]\sqrt{b_{Q_2}(2;X_i)}\;.
\end{align*}

In other respects, let us note that
$$
\Exp{|T_1(X_i)|}\leq \Exp{\left|\Upsilon_i(X_i)\right|}\;.
$$
We have
\begin{align*}
\Exp{\left|\Upsilon_i(X_i)\right|}&\leq \frac{R_i}{\sqrt{n}}b(A;\Sigma_\eps) \sqrt{\frac{x\trsp A^{-1} x}{t}}\left[|\alpha_i|b_L(1;X_i)+\frac{R_i}{\sqrt{n}}\frac{2b_L(2;X_i)}{\sqrt{t}}\right]\;,\\
&\leq K \frac{R_i}{\sqrt{n}}b(A;\Sigma_\eps) \sqrt{\frac{x\trsp A^{-1} x}{t}}
\left[|\alpha_i|\sqrt{b_L(2;X_i)}+\frac{R_i}{\sqrt{n}}\frac{\sqrt{b_L(4;X_i)}}{\sqrt{t}}\right]
\end{align*}
Hence,
$$
\left|\Exp{T_1(X_i)-T_1(Y_i)}\right|\leq \Psi_i(X_i)+\Psi_i(Y_i)\;,
$$
where
$$
\Psi_i(X_i)=K \frac{R_i}{\sqrt{n}}b(A;\Sigma_\eps) \sqrt{\frac{x\trsp A^{-1} x}{t}}
\left[|\alpha_i|\sqrt{b_L(2;X_i)}+\frac{R_i}{\sqrt{n}}\frac{\sqrt{b_L(4;X_i)}}{\sqrt{t}}\right]
\left(1\wedge \frac{R_i^2}{\sqrt{n}t}\sqrt{\frac{b_{Q_2}(2;X_i)}{n}}\right)\;.
$$

$\bullet$ \textbf{Control of }$\mathbf{\Exp{T_2(X_i)-T_2(Y_i)}}$\\
Recall that
$$
T_2(X_i)=\frac{R_i^2}{n} \frac{\qTilde_i(X_i)\varphi_i(X_i)}{(1+\frac{R_i^2}{n}q_i(X_i))^2}
$$
Clearly, using the notation $\Delta_i(X_i)=1/(1+R_i^2/n q_i(X_i))^{2}-1/(1+R_i^2/n d_i)^{2}$, we have
$$
\frac{\qTilde_i(X_i)}{(1+\frac{R_i^2}{n}q_i(X_i))^2}=\frac{\dTilde_i}{(1+\frac{R_i^2}{n}d_i)^2}+
\frac{\qTilde_i(X_i)-\dTilde_i}{(1+\frac{R_i^2}{n}q_i(X_i))^2}+\dTilde_i \Delta_i(X_i)\;.
$$
Now $\Exp{\varphi_i(X_i)}=\Exp{\varphi_i(Y_i)}$, so to control $\Exp{T_2(X_i)-T_2(Y_i)}$, all we need to do is control
\begin{align*}
T_{2,1}(X_i)&=\frac{R_i^2}{n} \frac{\qTilde_i(X_i)-\dTilde_i}{(1+\frac{R_i^2}{n}q_i(X_i))^2} \varphi_i(X_i)\\
T_{2,2}(X_i)&=\frac{R_i^2}{n}\dTilde_i \Delta_i(X_i)\varphi_i(X_i)\;.
\end{align*}
Recall that
$$
\frac{R_i^2}{n} \left|\dTilde_i \Delta_i(X_i)\right|\leq K b(A;\Sigma_\eps)\frac{R_i^2}{n}\frac{|q_i(X_i)-d_i|}{1+\frac{R_i^2}{n} q_i(X_i)}
$$
Hence,
$$
\Exp{|T_{2,2}(X_i)|}\leq K b(A;\Sigma_\eps)\frac{R_i^2}{n}\Exp{|q_i(X_i)-d_i||\varphi_i(X_i)|}\;,
$$
and we have already gotten a bound on $\Exp{|q_i(X_i)-d_i||\varphi_i(X_i)|}$, so we get
$$
\Exp{|T_{2,2}(X_i)|}\leq K \frac{R_i^2}{n}b(A;\Sigma_\eps)\sqrt{\frac{x\trsp A^{-1}x}{t}}\frac{\sqrt{b_{Q_2}(2;X_i)}}{t}
\left[\frac{1}{\sqrt{n}}|\alpha_i|R_i\sqrt{b_L(2;X_i)}+\frac{R_i^2}{n}\sqrt{\frac{b_L(4;X_i)}{t}}\right]\;.
$$
Similarly, using the fact that $\sqrt{\Exp{|\qTilde_i(X_i)-\dTilde_i|^2}}\leq \sqrt{b_{Q_2}(2;X_i)}b(A;\Sigma_\eps)/t$, we see that
$$
\Exp{|T_{2,1}(X_i)|}\leq K \frac{R_i^2}{n} \frac{\sqrt{b_{Q_2}(2;X_i)}b(A;\Sigma_\eps)}{t}
\left[\frac{1}{\sqrt{n}}|\alpha_i|R_i\sqrt{b_L(2;X_i)}+\frac{R_i^2}{n}\sqrt{\frac{b_L(4;X_i)}{t}}\right]
\sqrt{\frac{x\trsp A^{-1}x}{t}}\;.
$$

On the other hand,
$$
\left|\Exp{T_2(X_i)}\right|\leq b(A;\Sigma_\eps)\Exp{|\varphi_i(X_i)|}
$$
and we have already seen that
$$
\Exp{|\varphi_i(X_i)|}\leq K \sqrt{\frac{x\trsp A^{-1}x}{t}}\frac{R_i}{\sqrt{n}} \left(|\alpha_i|\sqrt{b_L(2;X_i)}+\frac{R_i}{\sqrt{n}} \frac{\sqrt{b_L(4;X_i)}}{\sqrt{t}}\right)\;.
$$
So we conclude that
$$
\left|\Exp{T_2(X_i)-T_2(Y_i)}\right|\leq U_i(X_i)+U_i(Y_i)
$$
where
$$
U_i(X_i)=K b(A;\Sigma_\eps)\sqrt{\frac{x\trsp A^{-1}x}{t}}\left[\frac{1}{\sqrt{n}}|\alpha_i|R_i\sqrt{b_L(2;X_i)}+\frac{R_i^2}{n}\sqrt{\frac{b_L(4;X_i)}{t}}\right]
\left[\frac{R_i^2}{n} \frac{\sqrt{b_{Q_2}(2;X_i)}}{t}\wedge 1
\right]\;.
$$
$\bullet$\textbf{Putting everything together}
Since $K$ can be chosen so that $\Psi_i(X_i)=U_i(X_i)$, we conclude that
$$
\left|\Exp{H(\alpha;X)-H(\alpha;Y)}\right|\leq 2 \sum_{i=1}^n U_i(X_i)+U_i(Y_i)\;.
$$
\end{proof}

\def\ee{\mathbf{E}}
\def\pp{\mathbf{P}}

\subsection{Checking the heuristics}\label{subsec:CheckHeuristics}
In Subsection \ref{subsec:Heuristics}, we gave some heuristics
to compute an asymptotically deterministic equivalent of forms
like $x\trsp (X\trsp D^2X/n+A)^{-1}x$
and $x\trsp (X\trsp D^2X/n+A)^{-1}\Sigma_{\eps} (X\trsp D^2X/n+A)^{-1} x$
in the case where all the $X$'s have the same covariance.
We now prove them rigorously.

Of course, the centerpiece of our analysis is the fact
that this only need to be done in the Gaussian case.
The proof is somewhat involved, since at the level
of generality at which we operate, we cannot seem to rely
on invariance properties of the Gaussian distribution
which were recently systematically exploited
in \cite{nekMarkoRisk}, \cite{nekMarkoRealizedRisk}
and have been a mainstay of multivariate statistics
(\cite{anderson03}, \cite{EatonBookReprint83}, \cite{ChikuseBook03}).
As is often the case,
computing the limit (or a deterministic equivalent)
of the quantities we are interested in is in fact
at least as difficult as showing that the limit does
not depend on the particulars of the distributions we consider,
or bounding the variance (or higher central moments).

It should be noted that our Lindeberg style results are valid
for families when each $X_i$ has a different $\Sigma_i$.
The limits we are investigating here are for the case
(mostly encountered or assumed in practice)
where all the $X_i$'s have the same $\Sigma$.

Let us begin by clarifying our assumptions and by introducing some notation:
We assume throughout this subsection that the rows $X_j\trsp$ of the matrix $X$ are
independent Gaussian random vectors with mean $0$ and (identical) covariance $\Sigma$.
Due to the concentration properties of the Gaussian distribution
(see e.g \cite{ledoux2001}), or using the properties of normal and weighted-$\chi^2$ random variables, this implies that for any $r \geq 1$,
\begin{align}
\label{eq:moment-vx}
\Exp{|v' X_j|^r} \leq K_r \norm{v}_2^r \opnorm{\Sigma}^{r/2}
\end{align}
for any deterministic vector $v$ and
\begin{align}
\label{eq:moment-xbx}
\Exp{|X_j' B X_j - \trace{\Sigma B}|^r} \leq K_r p^{r/2} \opnorm{B}^r \opnorm{\Sigma}^{r}
\end{align}
for any deterministic matrix $B$, where $K_r$ is a numerical constant.
\pagebreak[1]

Given a matrix $C \succcurlyeq 0$, put
$$
M_C := (A + C) \,,
$$
where $A \succcurlyeq t \Id$ is our regularizing matrix as above.
Note that $\opnorm{M_C^{-1}} \leq t$. In the special case where
$C = \mathcal{S}(j) := \mathcal{S} - \tfrac1n R_j^2 X_j X_j\trsp$,
we simply write $M_j$ instead of $M_{S(j)}$.
We now recall the classic rank-1 update formula which will again be used repeatedly in this part of the paper.
\begin{align}
\label{eq:rank1update}
M_S^{-1} = M_j^{-1} - \frac{\tfrac{1}{n} R_j^2 \, M_j^{-1} X_j X_j\trsp M_j^{-1}}{1 + \tfrac{1}{n} R_j^2 \, X_j\trsp M_j^{-1} X_j} \,.
\end{align}

Unless otherwise mentioned, $B$ is always a deterministic
positive semidefinite matrix in the sequel. \linebreak[2]
For $j=1,\hdots,n$, let
$$
q_j := X_j\trsp M_j^{-1} X_j \,, \quad
d_j := \trace{\Sigma M_j^{-1}} \,, \quad
\tilde{q}_j := X_j\trsp M_j^{-1} B M_j^{-1} X_j \,, \quad
\tilde{d}_j := \trace{\Sigma M_j^{-1} B M_j^{-1}} \,.
$$
In this subsection, we will usually replace $q_j$ and $\tilde{q}_j$
with the fully deterministic quantities $\ee(d_j)$ and $\ee(\tilde{d}_j)$
(instead of $d_j$ and $\tilde{d}_j$).
Using the fact that $B$, $\Sigma$ and $M_j^{-1}$ are positive definite,
it is easy to see that
\begin{align}
\label{eq:ch-bounds-1}
0 \leq \frac{1}{1+\tfrac1n R_j^2 \, q_j} \leq 1 \,, \quad
0 \leq \frac{1}{1+\tfrac1n R_j^2 \, \ee(d_j)} \leq 1 
\end{align}
and
\begin{align}
\label{eq:ch-bounds-2}
0 \leq \frac{\tfrac1n R_j^2 \, \tilde{q}_j}{1+\tfrac1n R_j^2 \, q_j} \leq \frac{\opnorm{B}}{t} \,, \quad
0 \leq \frac{\tfrac1n R_j^2 \, \ee(\tilde{d}_j)}{1+\tfrac1n R_j^2 \, \ee(d_j)} \leq \frac{\opnorm{B}}{t} \,. 
\end{align}

The following lemma provides some additional estimates
which will be used in this subsection:

\begin{lemma}
\label{lemma:replacement}
Suppose that the above-mentioned assumptions are satisfied.
\begin{enumerate}[(a)]
\item
We have
$$
|\trace{\Sigma M_S^{-1} - \Sigma M_j^{-1}}| \leq \opnorm{\Sigma} t^{-1}
$$
and
$$
|\trace{\Sigma M_S^{-1} B M_S^{-1} - \Sigma M_j^{-1} B M_j^{-1}}| \leq 2 \opnorm{B} \opnorm{\Sigma} t^{-2} \,.
$$
\item
For fixed $r \geq 1$, we have
$$
\Exp{\left|\trace{\Sigma M_S^{-1}} - \Exp{\trace{\Sigma M_S^{-1}}}\right|^r} \leq K_r' n^{r/2} \opnorm{\Sigma}^r t^{-r}
$$
and
$$
\Exp{\left|\trace{\Sigma M_S^{-1} B M_S^{-1}} - \Exp{\trace{\Sigma M_S^{-1} B M_S^{-1}}}\right|^r} \leq K_r'' n^{r/2} \opnorm{B}^r \opnorm{\Sigma}^r t^{-2r} \,,
$$
where $K_r'$ and $K_r''$ are constants depending only on $r$.
\item
We have
$$
\Exp{\left|\frac{1}{1+\tfrac1n R_j^2 \, q_j} - \frac{1}{1+\tfrac1n R_j^2 \, \ee(d_j)}\right|} \leq K' \left( 1 \wedge \tfrac{1}{n} R_j^2 (\sqrt{p} + \sqrt{n}) \opnorm{\Sigma} t^{-1} \right)
$$
and
$$
\Exp{\left|\frac{\tfrac{1}{n} R_j^2 \, \tilde{q}_j}{(1+\tfrac1n R_j^2 \, q_j)^2} - \frac{\tfrac{1}{n} R_j^2 \, \ee(\tilde{d}_j)}{(1+\tfrac1n R_j^2 \, \ee(d_j))^2} \right|} \leq K'' \opnorm{B} t^{-1} \left( 1 \wedge \tfrac{1}{n} R_j^2 (\sqrt{p}+\sqrt{n}) \opnorm{\Sigma} t^{-1} \right) \,,
$$
where $K'$ and $K''$ are numerical constants.
\item
For any square-integrable random variables $Z_j$ such that $\Exp{Z_j}^2 \leq L^2$, we~have
$$
\sum_{j=1}^{n} \tfrac1n R_j^2 \, \Exp{\left| \left(\frac{1}{1+\tfrac1n R_j^2 \, q_j} - \frac{1}{1+\tfrac1n R_j^2 \, \ee(d_j)}\right) Z_j \right|}  = \gO(LU)
$$
and
$$
\sum_{j=1}^{n} \tfrac1n R_j^2 \, \Exp{\left| \left(\frac{\tfrac1n R_j^2 \, \tilde{q}_j}{(1+\tfrac1n R_j^2 \, q_j)^2} - \frac{\tfrac1n R_j^2 \, \ee(\tilde{d}_j)}{(1+\tfrac1n R_j^2 \, \ee(\tilde{d}_j))^2}\right) Z_j \right|} =  \gO(LU \opnorm{B} t^{-1}) \,,
$$
where
$$
U := \sum_{j=1}^{n} \left( \tfrac1n R_j^2 \wedge \tfrac{1}{n^2} R_j^4 (\sqrt{p}+\sqrt{n}) \opnorm{\Sigma} t^{-1} \right) \,.
$$
\item
For any bounded random vectors $V_j$ and $W_j$ independent of $X_j$
such that $\|V_j\|_2 \leq L_1$ and $\|W_j\|_2 \leq L_2$, we~have
$$
\sum_{j=1}^{n} \tfrac1n R_j^2 \Exp{\left| V_j\trsp M_S^{-1} W_j - V_j\trsp M_j^{-1} W_j \right|} = \gO(L_1 L_2 \widetilde{U})
$$
and
$$
\sum_{j=1}^{n} \tfrac1n R_j^2 \Exp{\left| V_j\trsp M_S^{-1} B M_S^{-1} W_j - V_j\trsp M_j^{-1} B M_j^{-1} W_j \right|} = \gO(L_1 L_2 \widetilde{U} \opnorm{B} t^{-1}) \,,
$$
where
$$
\widetilde{U} := \sum_{j=1}^{n} \left( \tfrac1n R_j^2 t^{-1} \wedge \tfrac{1}{n^2} R_j^4 \opnorm{\Sigma} t^{-2} \right) \,.
$$
\end{enumerate}
\end{lemma}

\begin{proof}
Throughout this proof, $K$ denotes a numerical constant
which may change from step to step.

(a)
At least the first inequality is well known in random matrix theory
(see e.g. \cite{silversteinbai95}).
We include a proof for the sake of completeness.
Using \eqref{eq:rank1update}, we get
$$
  \left|\trace{\Sigma M_S^{-1} - \Sigma M_j^{-1}}\right|
= \left|\frac{\tfrac1n R_j^2 X_j\trsp M_j^{-1} \Sigma M_j^{-1} X_j}{1 + \tfrac1n R_j^2 X_j\trsp M_j^{-1} X_j}\right|
\leq \opnorm{M_j^{-1/2} \Sigma M_j^{-1/2}}
\leq \opnorm{\Sigma} t^{-1} \,.
$$
In fact, this continues to hold for a general square matrix $\Sigma$.
It therefore follows that
\begin{multline*}
     \left|\trace{\Sigma M_S^{-1} B M_S^{-1} - \Sigma M_j^{-1} B M_j^{-1}}\right| \\
\leq \left|\trace{\Sigma (M_S^{-1}-M_j^{-1}) B M_S^{-1}}\right| + \left|\trace{\Sigma M_j^{-1} B (M_S^{-1}-M_j^{-1}}\right|
\leq 2 \opnorm{B} \opnorm{\Sigma} t^{-2} \,.
\end{multline*}

(b)
This is a simple consequence of Azuma's inequality (see e.g.\@ Lemma~4.1 in Ledoux).
We follow the proof of Lemma~6 in \cite{nekCorrEllipD} :
For $j=0,\hdots,n$, let $\mathcal{F}_j$ denote the $\sigma$-field
generated by $X_1,\hdots,X_j$. \linebreak[2] Then, using part (a),
we have
\begin{multline*}
  \left|\Exp{\trace{\Sigma M_S^{-1}} | \mathcal{F}_j} - \Exp{\trace{\Sigma M_S^{-1}} | \mathcal{F}_{j-1}}\right| \\
= \left|\Exp{\trace{\Sigma (M_S^{-1}-M_j^{-1})} | \mathcal{F}_j} - \Exp{\trace{\Sigma (M_S^{-1}-M_j^{-1})} | \mathcal{F}_{j-1}}\right|
\leq 2 \opnorm{\Sigma} t^{-1} \,,
\end{multline*}
so, by Azuma's inequality, we get
$$
\Pr(|\trace{\Sigma M_S^{-1}} - \Exp{\trace {\Sigma M_S^{-1}}}| > u) \leq 2 \exp (- u^2 / 8n \opnorm{\Sigma}^2 t^{-2})
$$
for all $u > 0$. Since $\Exp{|Z|^r} = \int_0^\infty ru^{r-1} \Pr(|Z| > u) \ du$
for any real random variable $Z$, the first inequality follows easily.
The second inequality is derived similarly.

(c)
Recall that from \eqref{eq:ch-bounds-1}, we have the simple estimate
$$
\left|\frac{1}{1+\tfrac1n R_j^2 \, q_j} - \frac{1}{1+\tfrac1n R_j^2 \, \ee(d_j)}\right| \leq 1 \,.
$$
Using \eqref{eq:moment-xbx} and part (b), we also have the estimate
\begin{multline*}
\Exp{\left|\frac{1}{1+\tfrac1n R_j^2 \, q_j} - \frac{1}{1+\tfrac1n R_j^2 \, \ee(d_j)}\right|} \leq \tfrac{1}{n} R_j^2 \, \Exp{\left| q_j - \Exp{d_j} \right|} \\[+3pt] \leq \tfrac{1}{n} R_j^2 \, \Exp{\left| q_j - d_j \right|} + \tfrac{1}{n} R_j^2 \, \Exp{\left| d_j - \Exp{d_j} \right|} \leq K \, \tfrac{1}{n} R_j^2 (\sqrt{p} + \sqrt{n}) \opnorm{\Sigma} t^{-1} \,.
\end{multline*}
It follows that
\begin{align}
\label{eq:betab1bound}
\Exp{\left|\frac{1}{1+\tfrac1n R_j^2 \, q_j} - \frac{1}{1+\tfrac1n R_j^2 \, \ee(d_j)}\right|} \leq K \left( 1 \wedge \tfrac{1}{n} R_j^2 (\sqrt{p} + \sqrt{n}) \opnorm{\Sigma} t^{-1} \right) \,,
\end{align}
and the first inequality is proved.
For the second inequality, first observe that from
\eqref{eq:ch-bounds-1} and \eqref{eq:ch-bounds-2},
we have the simple estimate
$$
\left|\frac{\tfrac1n R_j^2 \, \tilde{q}_j}{(1+\tfrac1n R_j^2 \, q_j)^2} - \frac{\tfrac1n R_j^2 \, \ee(\tilde{d}_j)}{(1+\tfrac1n R_j^2 \, \ee(d_j))^2}\right| \leq \opnorm{B} t^{-1} \,.
$$
Moreover, writing
\begin{multline*}
  \frac{\tfrac{1}{n} R_j^2 \, \tilde{q}_j}{(1+\tfrac1n R_j^2 \, q_j)^2}
- \frac{\tfrac{1}{n} R_j^2 \, \ee(\tilde{d}_j)}{(1+\tfrac1n R_j^2 \, \ee(d_j))^2}
= \frac{\tfrac{1}{n} R_j^2 \, (\tilde{q}_j - \ee(\tilde{d}_j))}{(1+\tfrac1n R_j^2 \, q_j)^2 \, (1+\tfrac1n R_j^2 \, \ee(d_j))^2} \\[+3pt]
+ \frac{\tfrac{1}{n} R_j^2 \, \tilde{q}_j}{1+\tfrac1n R_j^2 \, q_j} \left(\frac{1}{1+\tfrac1n R_j^2 \, q_j} - \frac{1}{1+\tfrac1n R_j^2 \, \ee(d_j)}\right)
+ \frac{\tfrac{1}{n} R_j^2 \, \ee(\tilde{d}_j)}{1+\tfrac1n R_j^2 \, \ee(d_j)} \left(\frac{1}{1+\tfrac1n R_j^2 \, q_j} - \frac{1}{1+\tfrac1n R_j^2 \, \ee(d_j)}\right)
\end{multline*}
and using \eqref{eq:moment-xbx} and part (b), \eqref{eq:ch-bounds-1}, \eqref{eq:ch-bounds-2}
as well as \eqref{eq:betab1bound}, we get the estimate
\begin{multline*}
     \Exp{\left|\frac{\tfrac{1}{n} R_j^2 \, \tilde{q}_j}{(1+\tfrac1n R_j^2 \, q_j)^2} - \frac{\tfrac{1}{n} R_j^2 \, \ee(\tilde{d}_j)}{(1+\tfrac1n R_j^2 \, \ee(d_j))^2}\right|}
\leq \tfrac{1}{n} R_j^2 \Exp{\left|\tilde{q}_j - \ee(\tilde{d}_j)\right|} \\[+3pt] + 2 \opnorm{B} t^{-1} \Exp{\left|\frac{1}{1+\tfrac1n R_j^2 \, q_j} - \frac{1}{1+\tfrac1n R_j^2 \, \ee(d_j)}\right|}
\leq K \, \tfrac{1}{n} R_j^2 (\sqrt{p} + \sqrt{n}) \opnorm{B} \opnorm{\Sigma} t^{-2} \,.
\end{multline*}
It follows that
\begin{align}
\label{eq:betab2bound}
     \Exp{\left|\frac{\tfrac{1}{n} R_j^2 \, \tilde{q}_j}{(1+\tfrac1n R_j^2 \, q_j)^2} - \frac{\tfrac{1}{n} R_j^2 \, \ee(\tilde{d}_j)}{(1+\tfrac1n R_j^2 \, \ee(d_j))^2}\right|}
\leq K \opnorm{B} t^{-1} \left( 1 \wedge \tfrac{1}{n} R_j^2 (\sqrt{p}+\sqrt{n}) \opnorm{\Sigma} t^{-1} \right) \,.
\end{align}

(d)
Similar arguments as in part (c) show that
\begin{align*}
     \left(\Exp{\left|\frac{1}{1+\tfrac1n R_j^2 \, q_j} - \frac{1}{1+\tfrac1n R_j^2 \, \ee(d_j)}\right|^2}\right)^{1/2}
\leq K \left( 1 \wedge \tfrac{1}{n} R_j^2 (\sqrt{p} + \sqrt{n}) \opnorm{\Sigma} t^{-1} \right)
\end{align*}
and
\begin{align*}
     \left(\Exp{\left|\frac{\tfrac{1}{n} R_j^2 \, \tilde{q}_j}{(1+\tfrac1n R_j^2 \, q_j)^2} - \frac{\tfrac{1}{n} R_j^2 \, \ee(\tilde{d}_j)}{(1+\tfrac1n R_j^2 \, \ee(d_j))^2}\right|^2}\right)^{1/2}
\leq K \opnorm{B} t^{-1} \left( 1 \wedge \tfrac{1}{n} R_j^2 (\sqrt{p}+\sqrt{n}) \opnorm{\Sigma} t^{-1} \right) \,.
\end{align*}
Thus, the claim follows from Cauchy-Schwarz inequality.

(e)
On the one hand, we have the simple estimate
$$
\left| V_j\trsp (M_S^{-1} - M_j^{-1}) W_j \right| \leq 2L_1L_2 t^{-1} \,.
$$
On the other hand, using \eqref{eq:rank1update}, Cauchy-Schwarz inequality and \eqref{eq:moment-vx},
we have the estimate
$$
     \Exp{\left| V_j\trsp (M_S^{-1} - M_j^{-1}) W_j  \right|}
\leq \tfrac1n R_j^2 \Exp{\left| V_j\trsp M_j^{-1} X_j X_j\trsp M_j^{-1} W_j \right|}
\leq K \, \tfrac1n R_j^2 L_1 L_2 \opnorm{\Sigma} t^{-2} \,.
$$
Combining these estimates, it follows that
$$
\sum_{j=1}^{n} \tfrac1n R_j^2 \Exp{\left| V_j\trsp M_S^{-1} W_j -  V_j\trsp M_j^{-1}  W_j \right|} = \gO(L_1 L_2 \widetilde{U}) \,,
$$
which establishes the first part of (e).
For the second part of (e), write
\begin{multline*}
     \Exp{\left|V_j\trsp M_S^{-1} B M_S^{-1} W_j - V_j\trsp M_j^{-1}  B M_j^{-1}  W_j\right|} \\
\leq \Exp{\left|V_j\trsp (M_S^{-1} - M_j^{-1} ) B M_j^{-1} W_j\right|} + \Exp{\left|V_j\trsp M_j^{-1} B (M_S^{-1} - M_j^{-1} ) W_j\right|} \\ + \Exp{\left|V_j\trsp (M_S^{-1} - M_j^{-1} ) B (M_S^{-1} - M_j^{-1} ) W_j\right|} \,.
\end{multline*}
By the preceding estimates, the first two expectations are bounded
by $K \, \tfrac1n R_j^2 L_1 L_2 \opnorm{B} \opnorm{\Sigma} t^{-3}$ here.
For~the third expectation, we can use \eqref{eq:rank1update} and \eqref{eq:ch-bounds-2} to get
\begin{align*}
      |V_j\trsp (M_S^{-1} - M_j^{-1} ) B (M_S^{-1} - M_j^{-1} ) W_j|
&=    \frac{\tfrac{1}{n^2} R_j^4}{(1+\tfrac1n R_j^2 \, q_j)^2} \, | V_j M_j^{-1} X_j X_j' M_j^{-1} B M_j^{-1} X_j X_j' M_j^{-1} W_j| \\
&\leq \tfrac{1}{n} R_j^2 \, |V_j M_j^{-1} X_j X_j' M_j^{-1} W_j| \opnorm{B} t^{-1}
\end{align*}
and therefore, by Cauchy-Schwarz inequality and \eqref{eq:moment-vx},
$$
\Exp{|V_j\trsp (M_S^{-1} - M_j^{-1} ) B (M_S^{-1} - M_j^{-1} ) W_j|} \leq K \, \tfrac{1}{n} R_j^2 L_1 L_2 \opnorm{B} \opnorm{\Sigma} t^{-3} \,.
$$
Combining this with the simple estimate
$$
\left| V_j\trsp M_S^{-1} B M_S^{-1} W_j - V_j\trsp M_j^{-1} B M_j^{-1} W_j \right| \leq 2L_1L_2 \opnorm{B} t^{-2} \,,
$$
it follows that
$$
\sum_{j=1}^{n} \tfrac1n R_j^2 \Exp{\left| V_j\trsp M_S^{-1} B M_S^{-1} W_j - V_j\trsp M_j^{-1}  B M_j^{-1}  W_j \right|} = \gO(L_1 L_2 \widetilde{U} \opnorm{B} t^{-1}) \,.
$$

This completes the proof of the lemma.
\end{proof}

\pagebreak[2]
\medskip

To verify Heuristic \ref{heuristic:regularizedInverse},
we will prove the following result:

\begin{proposition}\label{proposition:regularizedInverse}
Suppose that the assumptions from the beginning of this subsection hold,
the ratio $p/n$ stays bounded, $\norm{v} = 1$ and
\begin{align}
\label{eq:rj-moments-1}
\sum_{j=1}^{n} \frac{R_j^2}{n} \opnorm{\Sigma} = \gO(1)
\qquad\text{and}\qquad
\sum_{j=1}^{n} \left( \frac{R_j^2}{n} \wedge \frac{R_j^4}{n^{3/2}} \opnorm{\Sigma} \right) \opnorm{\Sigma} = \lo(1)
\end{align}
as $n \to \infty$.
Then we have
\begin{align*}
\Exp{v\trsp (S+A)^{-1} v} - \Exp{v\trsp (\gamma(A)\Sigma+A)^{-1} v} \to 0 \,,
\end{align*}
where
$$
\gamma(A) := \frac{1}{n}\sum_{i=1}^n \frac{R_i^2}{1+\tfrac1n R_i^2 \trace{\Sigma M_S^{-1}}} \;.
$$
\end{proposition}

\begin{proof}
We first show that we may replace $\gamma(A)$ with the deterministic quantity
$$
\overline\gamma(A) := \sum_{i=1}^n \frac{\tfrac1n R_i^2}{1+\tfrac1n R_i^2 \Exp{\trace{\Sigma M_i^{-1}}}} \;.
$$
To this end, since $|v\trsp (A + \gamma(A)\Sigma)^{-1} v - v\trsp (A + \overline\gamma(A)\Sigma)^{-1} v| \leq t^{-2} |\gamma(A) - \overline\gamma(A)| \opnorm{\Sigma}$,
it suffices to show that
\begin{align}
\label{eq:replacegamma-1}
\Exp{|\gamma(A) - \overline\gamma(A)|} \opnorm{\Sigma} = \lo(1) \,.
\end{align}
But now,
\begin{align*}
      \Exp{|\gamma(A) - \overline\gamma(A)|}
&\leq \sum_{i=1}^n \Exp{\left|\frac{\tfrac{1}{n}R_i^2}{1+\tfrac1n R_i^2 \trace{\Sigma M_S^{-1}}} - \frac{\tfrac{1}{n}R_i^2}{1+\tfrac1n R_i^2 \Exp{\trace{\Sigma M_S^{-1}}}}\right|} \\
&\qquad + \sum_{i=1}^n \Exp{\left|\frac{\tfrac{1}{n}R_i^2}{1+\tfrac1n R_i^2 \Exp{\trace{\Sigma M_S^{-1}}}} - \frac{\tfrac{1}{n}R_i^2}{1+\tfrac1n R_i^2 \Exp{\trace{\Sigma M_i^{-1}}}}\right|} \\
&\leq K \left( \sum_{j=1}^{n} \Big( \tfrac1n R_j^2 \wedge \tfrac{1}{n^{3/2}} R_j^4 \opnorm{\Sigma} t^{-1} \Big)
    +\sum_{j=1}^{n} \Big( \tfrac1n R_j^2 \wedge \tfrac{1}{n^2} R_j^4 \opnorm{\Sigma} t^{-1} \Big) \right) \,,
\end{align*}
where the second step follows from similar arguments
as in the proof of Lemma \ref{lemma:replacement}\,(c)
(using Lemma~\ref{lemma:replacement} (b)~and~(a)).
Thus, \eqref{eq:replacegamma-1} follows from Assumption~\eqref{eq:rj-moments-1}.

We now proceed similarly as in \cite{silverstein95} and \cite{nekCorrEllipD}.
Using \eqref{eq:rank1update}, it is easy to check that
$M_S^{-1} X_j = (1+\tfrac1n R_j^2 \, q_j)^{-1} M_j^{-1} X_j$.
Thus, setting $T := \overline\gamma(A) \Sigma$, so that $M_T=A+\overline\gamma(A) \Sigma$, we~get
\begin{align*}
M_S^{-1} - M_T^{-1} = - M_S^{-1} (S - T) M_T^{-1}
&= - \sum_{i=1}^{n} \tfrac{1}{n} R_i^2 \, M_S^{-1} X_i X_i' M_T^{-1} + M_S^{-1} T M_T^{-1} \\
&= - \sum_{i=1}^{n} \frac{\tfrac{1}{n} R_i^2 \, M_i^{-1} X_i X_i' M_T^{-1}}{1 + \tfrac1n R_i^2 \, q_i} + M_S^{-1} T M_T^{-1} \\
&= - \sum_{i=1}^{n} \left( \frac{\tfrac{1}{n} R_i^2 \, M_i^{-1} X_i X_i' M_T^{-1}}{1 + \tfrac1n R_i^2 \, q_i} - \frac{\tfrac1n R_i^2 \, M_S^{-1} \Sigma M_T^{-1}}{1 + \tfrac1n R_i^2 \, \ee(d_i)} \right)
\end{align*}
and therefore
\begin{align}
\label{eq:msmtdiff}
\Exp{v' M_S^{-1} v} - \Exp{v' M_T^{-1} v}
&= - \sum_{i=1}^{n} \Exp{ \frac{\tfrac{1}{n} R_i^2 \, M_i^{-1} X_i X_i' M_T^{-1}}{1 + \tfrac1n R_i^2 \, q_i} - \frac{\tfrac1n R_i^2 \, M_S^{-1} \Sigma M_T^{-1}}{1 + \tfrac1n R_i^2 \, \ee(d_i)} } \,.
\end{align}
Now, using Lemma \ref{lemma:replacement}\,(d), the independence of $X_j$
and $X_1,\hdots,X_{j-1},X_{j+1},\hdots,X_n$, Lemma \ref{lemma:replacement}\,(e),
and Assumption \eqref{eq:rj-moments-1}, it follows that
\begin{align*}
   \sum_{i=1}^{n} \Exp{ \frac{\tfrac{1}{n} R_i^2 \, v' M_i^{-1} X_i X_i' M_T^{-1} v}{1+\tfrac1n R_i^2 \, q_i}}
&= \sum_{i=1}^{n} \Exp{ \frac{\tfrac{1}{n} R_i^2 \, v' M_i^{-1} X_i X_i' M_T^{-1} v}{1+\tfrac1n R_i^2 \,\ee(d_i)}} + \lo (1) \\
&= \sum_{i=1}^{n} \Exp{ \frac{\tfrac{1}{n} R_i^2 \, v' M_i^{-1} \Sigma M_T^{-1} v}{1+\tfrac1n R_i^2 \,\ee(d_i)}} + \lo (1) \\
&= \sum_{i=1}^{n} \Exp{ \frac{\tfrac{1}{n} R_i^2 \, v' M_S^{-1} \Sigma M_T^{-1} v}{1+\tfrac1n R_i^2 \,\ee(d_i)}} + \lo (1) \,.
\end{align*}
This completes the proof of Proposition \ref{proposition:regularizedInverse}.
\end{proof}

\pagebreak[2]
\medskip

To verify Heuristic \ref{heuristic:squareRegularizedInverse},
we will prove the following result:

\begin{proposition}\label{proposition:squareRegularizedInverse}
Suppose that the assumptions from the beginning of this subsection hold,
the ratio $p/n$ stays bounded, $\norm{v} = 1$ and
\begin{align}
\label{eq:rj-moments-2}
\sum_{j=1}^{n} \frac{R_j^2}{n} \opnorm{\Sigma} = \gO(1)
\qquad\text{and}\qquad
\sum_{j=1}^{n} \left( \frac{R_j^2}{n} \wedge \frac{R_j^4}{n^{3/2}} \opnorm{\Sigma} \right) \opnorm{B} \opnorm{\Sigma} = \lo(1)
\end{align}
as $n \to \infty$.
Then we have
\begin{align*}
\Exp{v\trsp (S+A)^{-1} B (S+A)^{-1} v} - \Exp{v\trsp (A+\gamma(A)\Sigma)^{-1} (B+\xi(A,B) \Sigma) (A+\gamma(A)\Sigma)^{-1} v} \to 0 \,,
\end{align*}
where $\gamma(A)$ is defined in Proposition \ref{proposition:regularizedInverse} and
$$
\xi(A,B) := \left[\frac{1}{n}\sum_{i=1}^n \frac{R_i^4}{\left(1+\tfrac1n R_i^2 \trace{\Sigma M_S^{-1}}\right)^2}\right]\frac{1}{n} \trace{\Sigma(S+A)^{-1}B (S+A)^{-1}}\;.
$$
\end{proposition}

\begin{proof}
Similarly as in the proof of Proposition \ref{proposition:regularizedInverse},
we first show that we may replace $\gamma(A)$ and $\xi(A,B)$
with the deterministic quantities $\overline\gamma(A)$ and $\overline\xi(A,B)$,
where $\overline\gamma(A)$ is defined in the proof of Proposition \ref{proposition:regularizedInverse} and
$$
\overline\xi(A,B) := \sum_{i=1}^n \frac{\tfrac{1}{n^2} R_i^4}{\left(1+\tfrac1n R_i^2 \Exp{\trace{\Sigma M_i^{-1}}}\right)^2} \, \Exp{\trace{\Sigma M_i^{-1}B M_i^{-1}}}\;.
$$
To begin with, similarly as in \eqref{eq:ch-bounds-2},
$\xi(A,B)$ and $\overline\xi(A,B)$ are bounded
by $\sum_{j=1}^{n} \tfrac1n R_j^2 \opnorm{B} t^{-1}$.
It therefore follows from Assumption \eqref{eq:rj-moments-2} that
$\xi(A,B) \Sigma$ and $\overline\xi(A,B) \Sigma$ are bounded
(in operator norm) by $K \opnorm{B}$, where $K$ is a constant.
Thus, using the decomposition
$$
(A_1 A_2 A_3 - B_1 B_2 B_3) = (A_1 - B_1) B_2 B_3 + A_1 (A_2 - B_2) B_3 + A_1 A_2 (A_3 - B_3) \,,
$$
we see that it suffices to check that
\begin{align}
\label{eq:replacegamma-2}
\Exp{|\gamma(A) - \overline\gamma(A)|} \opnorm{B} \opnorm{\Sigma} = o(1)
\quad\text{and}\quad
\Exp{|\xi(A,B) - \overline\xi(A,B)|} \opnorm{\Sigma} = o(1) \,.
\end{align}
The former bound is clear from \eqref{eq:replacegamma-1},
since we are assuming \eqref{eq:rj-moments-2} instead of \eqref{eq:rj-moments-1} now.
For the latter bound, let us note that
\begin{align*}
&\mskip24mu \Exp{|\xi(A,B) - \overline\xi(A,B)|} \\
&\leq \sum_{i=1}^n \Exp{\left|\frac{\tfrac{1}{n^2} R_i^4 \, \trace{\Sigma M_S^{-1}B M_S^{-1}}}{\left(1+\tfrac1n R_i^2 \, \trace{\Sigma M_S^{-1}}\right)^2} - \frac{\tfrac{1}{n^2} R_i^4 \, \Exp{\trace{\Sigma M_S^{-1} B M_S^{-1}}}}{\left(1+\tfrac1n R_i^2 \, \Exp{\trace{\Sigma M_S^{-1}}}\right)^2} \right|} \\
&\qquad + \sum_{i=1}^n \Exp{\left|\frac{\tfrac{1}{n^2} R_i^4 \, \Exp{\trace{\Sigma M_S^{-1}B M_S^{-1}}}}{\left(1+\tfrac1n R_i^2 \, \Exp{\trace{\Sigma M_S^{-1}}}\right)^2} - \frac{\tfrac{1}{n^2} R_i^4 \, \Exp{\trace{\Sigma M_i^{-1} B M_i^{-1}}}}{\left(1+\tfrac1n R_i^2 \, \Exp{\trace{\Sigma M_i^{-1}}}\right)^2} \right|} \\
&\leq K \opnorm{B} t^{-1} \left( \sum_{j=1}^n \Big( \tfrac1n R_j^2 \wedge \tfrac{1}{n^{3/2}} R_j^4 \opnorm{\Sigma} t^{-1} \Big)
    + \sum_{j=1}^n \Big( \tfrac1n R_j^2 \wedge \tfrac{1}{n^{2}} R_j^4 \opnorm{\Sigma} t^{-1} \Big) \right) \,,
\end{align*}
where the second step follows from similar arguments
as in the proof of Lemma \ref{lemma:replacement}\,(c)
(using Lemma~\ref{lemma:replacement} (b)~and~(a)).
In view of Assumption~\eqref{eq:rj-moments-2},
this establishes \eqref{eq:replacegamma-2}.

Put $T := \overline\gamma(A) \Sigma$, $T(u) := \overline\gamma(A + uB) \Sigma$ ($u > 0$) and observe that
$$
\left. \frac{d}{du} ( S+A+uB )^{-1} \right|_{u=0} = - (S+A)^{-1} B (S+A)^{-1}
$$
and
$$
\left. \frac{d}{du} ( T(u)+A+uB )^{-1} \right|_{u=0} = - (T+A)^{-1} \left( B + \overline\xi(A,B) \Sigma \right) (T+A)^{-1} \,.
$$
Thus, replacing $A$ with $A + uB$ in \eqref{eq:msmtdiff} and
calculating the derivative with respect to $u$ at $u = 0$,
we get
\begin{align}
- \ee (v' M_S^{-1} B M_S^{-1} v) + \ee (v' M_T^{-1} (B + \overline\xi(A,B) \Sigma) M_T^{-1} v)
= D_1 + D_2 + D_3 \,,
\end{align}
where
\begin{align*}
D_1 &= - \sum_{i=1}^{n} \Exp{ \frac{\tfrac{1}{n^2} R_i^4 \, \tilde{q}_i}{(1+\tfrac1n R_i^2 \, q_i)^2} \, v' M_i^{-1} X_i X_i' M_T^{-1} v - \frac{\tfrac1n R_i^4 \, \ee(\tilde{d}_i)}{(1+\tfrac1n R_i^2 \,\ee(d_i))^2} \, v' M_S^{-1} \Sigma M_T^{-1} v } \,, \\
D_2 &= + \sum_{i=1}^{n} \Exp{ \frac{\tfrac{1}{n} R_i^2}{1+\tfrac1n R_i^2 \, q_i} \, v' M_i^{-1} B M_i^{-1} X_i X_i' M_T^{-1} v - \frac{\tfrac{1}{n} R_i^2}{1+\tfrac1n R_i^2 \,\ee(d_i)} \, v' M_S^{-1} B M_S^{-1} \Sigma M_T^{-1} v } \,, \\
D_3 &= + \sum_{i=1}^{n} \ee\Bigg( \frac{\tfrac{1}{n} R_i^2}{1+\tfrac1n R_i^2 \, q_i} \, v' M_i^{-1} X_i X_i' M_T^{-1} (B+\overline\xi(A,B)\Sigma) M_T^{-1} v \\&\qquad\qquad\qquad\qquad\qquad\qquad\qquad\qquad - \frac{\tfrac{1}{n} R_i^2}{1+\tfrac1n R_i^2 \,\ee(d_i)} \, v' M_S^{-1} \Sigma M_T^{-1} (B + \overline\xi(A,B) \Sigma) M_T^{-1} v \Bigg) \,.
\end{align*}
Similarly as in the proof of Proposition \ref{proposition:regularizedInverse},
the idea is to show that for $i=1,2,3,$ $D_i \to 0$ as $n \to \infty$.
Making appropriate use of Lemma \ref{lemma:replacement},
this follows by essentially the same calculation
as in the proof of Proposition \ref{proposition:regularizedInverse},
so~that we go fast over the details.

For the first difference, we use Lemma \ref{lemma:replacement}\,(d),
the independence of $X_j$ and $X_1,\hdots,X_{j-1},X_{j+1},\hdots,X_n$,
and Lemma \ref{lemma:replacement}\,(e) to obtain
\begin{align*}
  &\sum_{i=1}^{n} \Exp{ \frac{\tfrac{1}{n^2} R_i^4 \, \tilde{q}_i}{(1+\tfrac1n R_i^2 \, q_i)^2} \, v' M_i^{-1} X_i X_i' M_T^{-1} v } \\
= &\sum_{i=1}^{n} \Exp{ \frac{\tfrac{1}{n^2} R_i^4 \, \ee(\tilde{d}_i)}{(1+\tfrac1n R_i^2 \,\ee(d_i))^2} \, v' M_i^{-1} X_i X_i' M_T^{-1} v } + \lo(1) \\
= &\sum_{i=1}^{n} \Exp{ \frac{\tfrac{1}{n^2} R_i^4 \, \ee(\tilde{d}_i)}{(1+\tfrac1n R_i^2 \,\ee(d_i))^2} \, v' M_i^{-1} \Sigma M_T^{-1} v } + \lo(1) \\
= &\sum_{i=1}^{n} \Exp{ \frac{\tfrac{1}{n^2} R_i^4 \, \ee(\tilde{d}_i)}{(1+\tfrac1n R_i^2 \,\ee(d_i))^2} \, v' M_S^{-1} \Sigma M_T^{-1} v } + \lo(1) \,.
\end{align*}
(In the final step, we have also used
\eqref{eq:ch-bounds-1} and \eqref{eq:ch-bounds-2}
to see that the fraction is bounded by $\tfrac1n R_i^2 \opnorm{B} t^{-1}$.)
\linebreak[2]
For the second and third difference, it follows by similar arguments that
\begin{align*}
  &\sum_{i=1}^{n} \Exp{ \frac{\tfrac{1}{n} R_i^2}{1+\tfrac1n R_i^2 \, q_i} \, v' M_i^{-1} B M_i^{-1} X_i X_i' M_T^{-1} v } \\
= &\sum_{i=1}^{n} \Exp{ \frac{\tfrac{1}{n} R_i^2}{1+\tfrac1n R_i^2 \,\ee(d_i)} \, v' M_i^{-1} B M_i^{-1} X_i X_i' M_T^{-1} v } + \lo(1) \\
= &\sum_{i=1}^{n} \Exp{ \frac{\tfrac{1}{n} R_i^2}{1+\tfrac1n R_i^2 \,\ee(d_i)} \, v' M_i^{-1} B M_i^{-1} \Sigma M_T^{-1} v } + \lo(1) \\
= &\sum_{i=1}^{n} \Exp{ \frac{\tfrac{1}{n} R_i^2}{1+\tfrac1n R_i^2 \,\ee(d_i)} \, v' M_S^{-1} B M_S^{-1} \Sigma M_T^{-1} v } + \lo(1)
\end{align*}
and
\begin{align*}
  &\sum_{i=1}^{n} \Exp{ \frac{\tfrac{1}{n} R_i^2}{1+\tfrac1n R_i^2 \, q_i} \, v' M_i^{-1} X_i X_i' M_T^{-1} (B+\overline\xi(A,B)\Sigma) M_T^{-1} v } \\
= &\sum_{i=1}^{n} \Exp{ \frac{\tfrac{1}{n} R_i^2}{1+\tfrac1n R_i^2 \,\ee(d_i)} \, v' M_i^{-1} X_i X_i' M_T^{-1} (B+\overline\xi(A,B)\Sigma) M_T^{-1} v } + \lo(1) \\
= &\sum_{i=1}^{n} \Exp{ \frac{\tfrac{1}{n} R_i^2}{1+\tfrac1n R_i^2 \,\ee(d_i)} \, v' M_i^{-1} \Sigma M_T^{-1} (B+\overline\xi(A,B)\Sigma) M_T^{-1} v } + \lo(1) \\
= &\sum_{i=1}^{n} \Exp{ \frac{\tfrac{1}{n} R_i^2}{1+\tfrac1n R_i^2 \,\ee(d_i)} \, v' M_S^{-1} \Sigma M_T^{-1} (B + \overline\xi(A,B) \Sigma) M_T^{-1} v } + \lo(1) \,.
\end{align*}
This concludes the proof.
\end{proof}

\subsection{On the Rate of Convergence}

Suppose that the constants $b_L(4;X_i)$ and $b_{Q_2}(2;X_i)/n$
from \eqref{B-L} and \eqref{B-Q} are uniformly bounded.
Then our results show that if the $R_i$ are also uniformly bounded,
we have, for instance,
$$
\Exp{|g(\alpha;X)-\Exp{g(\alpha;X)}|^2} = \gO(n^{-1})
$$
and
$$
|\Exp{g(\alpha;X)-g(\alpha;Y)}| = \gO(n^{-1/2}) \,.
$$
More generally, this still holds
if the $R_i$ are given by i.i.d.\@ random variables
with finite $4$th moments.
For some applications (e.g.\@ to the field of finance),
it may be helpful to have the same results under the weaker assumption
that the $R_i$ are given by i.i.d.\@ random variables
with finite second moments only.
Recall that this is the minimal reasonable assumption,
for if the $R_i$ do not have second moments,
the covariance matrix of the vectors $R_i X_i$ is not defined.
In this section we sketch how to derive results under this minimal assumption.

However, to derive our results, we need somewhat stronger conditions
on the covariance matrices $\Sigma_i$, the regularizing matrix $A$ and
the distributions of the random variables $X_i$. More precisely,
we will work under the following \emph{additional} assumptions:
\begin{itemize}
\setlength{\itemsep}{0.1pt}
\item
We have $p/n \geq c_\varrho$ for some $c_\varrho > 0$.
\item
We have $\tfrac{1}{n} \sum_{i=1}^{n} R_i^2 \leq C_R$ for some $C_R < \infty$.
\item
We have $\tfrac1p \trace A \leq C_A$ for some $C_A < \infty$.
\item
We have $\tfrac1p \trace \Sigma_i \leq C_\Sigma$ for some $C_\Sigma < \infty$.
\item
There exist $c_\Sigma > 0$ and $\varepsilon \in (0,1)$
such that the number of eigenvalues of $\Sigma_i$
which are less than $\leq c_\Sigma$ is less than $p\varepsilon$.
\item
We have $\opnorm{\Sigma_i} \leq C_\Sigma$,
and the constants $b_L(8;X_i)$ and $b_{Q_2}(4;X_i)/n^2$
are uniformly bounded.
\end{itemize}
Let us mention that the very last assumption could be weakened
in exchange for a worse rate of convergence in the following results.
For instance, we could easily allow for a bound of the order
$\gO(\log n)$ or $\gO(n^\delta)$ (with $\delta > 0$ sufficiently small).

As it is our main intention here to give an idea of what is possible,
we concentrate on one particular case
and present results for the quadratic form
$g(\alpha) := n^{-1} \alpha\trsp D X (X\trsp D^2 X/n+A)^{-1} X\trsp D \alpha$
(from Section~\ref{subsubsec:quadFormsInverseMeanAndPopMean}) only.

Our results rely on the observation that
(i) normalized traces of regularized inverses of random matrices
are typically strongly concentrated
and
(ii) under the assumptions stated above,
$\Exp{\trace{\Sigma_i M_i^{-1}}}$ is of the order $n$.
Let us provide precise formulations:

The first observation has already been used several times in this paper
(see also \cite{nekCorrEllipD}),
for instance in the proof of~Lemma \ref{lemma:replacement}\,(b),
where it is stated that 
$$
\Proba{\left|\tfrac{1}{p} \trace {\Sigma M^{-1}} - \Exp{\tfrac{1}{p} \trace {\Sigma M^{-1}}}\right| \geq u} \leq 2 \exp(-u^2 p^2 t^2 / 8 n \opnorm{\Sigma}^2) \,.
$$
for any $u > 0$.
For the second observation, we show the following lemma.

\begin{lemma}
Under the afore-mentioned assumptions, we have
$$
\Exp{\tfrac{1}{n} \trace {\Sigma_i {M}_i^{-1}}} \geq c \,,
$$
where $c = c(c_\varrho,C_R,C_A,C_\Sigma,c_\Sigma,\varepsilon)$.
\end{lemma}

In the following proof, if $A,B$ are any matrices and $x \in \mathbb{R}_{+}$,
we call~$B$ (by slight abuse of terminology) a rank $x$ modification of $A$
if $\text{rank}(A-B) \leq x$, and if $M$ is any symmetric matrix,
we let $\lambda_1(M),\hdots,\lambda_p(M)$ be the eigenvalues $M$.

\begin{proof}
We have
$$
  \Exp{\trace{M_i}}
= \Exp{\trace{A + \textstyle\sum_{j \ne i} \tfrac1n R_j^2 X_j X_j\trsp}}
= \trace {A} + \sum_{j \ne i} \tfrac1n R_j^2 \trace{\Sigma_j}
\leq p (C_A + C_R C_\Sigma) \,.
$$
Thus, by Markov's inequality, it follows that
with $C := 4(C_A + C_R C_\Sigma) / (1-\varepsilon)$,
we have
$$
     \Proba{\tfrac{1}{p} \textstyle\sum_{j=1}^{p} \pmb{1}_{\{ \lambda_j(M_i) \geq C \}} \geq \frac{1-\varepsilon}{2}}
\leq \Proba{\tfrac{1}{p} \trace{M_i} \geq 2(C_A+C_RC_\Sigma)}
\leq \dfrac{\Exp{\trace{M_i}}}{2p(C_A+C_RC_\Sigma)}
\leq \tfrac{1}{2} \,.
$$
Consider the set $G$ where
$\tfrac{1}{p} \sum_{j=1}^{p} \pmb{1}_{\{ \lambda_j(M_i) \geq C \}} \leq \frac{1-\varepsilon}{2}$,
so $\Proba{G} \geq \tfrac12$.
Then, by spectral calculus, there exists a positive-definite
rank $p(\frac{1-\varepsilon}{2})$ modification $\tilde{M_i}$ of $M_i$
such~that $\lambda_j(\tilde{M_i}) \leq C$ for all $j=1,\hdots,p$.
Similarly, there exists a positive-definite
rank $p\varepsilon$ modification $\tilde{\Sigma_i}$ of $\Sigma_i$
such~that $\lambda_j(\tilde{\Sigma_i}) \geq c_\Sigma$ for all $j=1,\hdots,p$.
It~follows that
$\lambda_j(\tilde{\Sigma}_i^{1/2} \tilde{M}_i^{-1} \tilde{\Sigma}_i^{1/2}) \geq c_\Sigma/C$
for all $j=1,\hdots,p$.
Indeed, for any vector $x$ of norm $1$,
$$
x\trsp \tilde{\Sigma}_i^{1/2} \tilde{M}_i^{-1} \tilde{\Sigma}_i^{1/2} x
\geq x\trsp \tilde{\Sigma_i} x / C
\geq c_\Sigma / C \,,
$$
and thus $\tilde{\Sigma}_i^{1/2} \tilde{M}_i^{-1} \tilde{\Sigma}_i^{1/2} \succcurlyeq c_\Sigma / C$.
By Theorem A.43 in \cite{baisilverstein10},
it further follows that at~least $p(\frac{1-\varepsilon}{2})$ eigenvalues
of $\Sigma_i M_i^{-1}$ are $\geq c_\Sigma/C$.
Thus, as $\Sigma_i^{1/2} M_i^{-1} \Sigma_i^{1/2} \succcurlyeq 0$,
we have shown that on the~set $G$,
$$
\tfrac1p \trace{\Sigma_i M_i^{-1}} \geq \tfrac{1-\varepsilon}{2} c_\Sigma/C \,,
$$
Since $\Proba{G} \geq \tfrac12$ and $\Sigma_i^{1/2} M_i^{-1} \Sigma_i^{1/2} \succcurlyeq 0$,
we may conclude that
$$
\tfrac{1}{n} \Exp{\trace {\Sigma_i M_i^{-1}}} \geq \tfrac{1-\varepsilon}{4} c_\varrho c_\Sigma / C =: c \,.
$$
\end{proof}

\begin{corollary}
We have $\Proba{\tfrac1n \trace{\Sigma_i M_i^{-1}} \leq \tfrac12 c} \leq C_0 \exp(- c_0 n)$.
\end{corollary}

Let us now investigate the implications of these observations
for our results concerning $g(\alpha,X)$:
Recall from the proof of Theorem \ref{thm:ThirdFormIsDeterministic}
that (with the notation there)
$$
\Exp{|g(\alpha;X) - \Exp{g(\alpha;X)}|^2} \leq \sum_{j=1}^{n} \Exp{|T -  T_i|^2} \,,
$$
where
\begin{multline*}
T-T_i=\alpha_i^2 \frac{R_i^2}{n} \frac{q_i(X_i)-d_i}{(1+\frac{R_i^2}{n}q_i)(1+\frac{R_i^2}{n}d_i)}+\frac{1}{1+\frac{R_i^2}{n}d_i}(R_i^2/n\zeta_i^2-2\alpha_iR_i/\sqrt{n} \zeta_i)\\
+\frac{\tfrac1nR_i^2(d_i-q_i(X_i))}{(1+\frac{R_i^2}{n}q_i(X_i))(1+\frac{R_i^2}{n}d_i)}(R_i^2/n\zeta_i^2-2\alpha_iR_i/\sqrt{n} \zeta_i)\;.
\end{multline*}

Now, on the set $G_i := \{ \tfrac1n d_i \geq \tfrac12 c \}$
(which has probability $1-o(1)$ by the preceding corollary),
$$
|T-T_i|\leq \frac{2}{cn} \alpha_i^2 |q_i(X_i)-d_i|+\left(\frac{2}{c}+1\right)\left(1/n\zeta_i^2+2\alpha_i/\sqrt{n}|\zeta_i|\right)\\
+\frac{2}{cn}|d_i-q_i(X_i)|\left(R_i^2/n\zeta_i^2+2\alpha_iR_i/\sqrt{n}|\zeta_i|\right)\;.
$$
Using \eqref{B-L} and \eqref{B-Q} as well as Cauchy-Schwarz inequality,
it follows that
\begin{multline*}
\Expj{i}{|T-T_i|^2 \, \bm{1}_{G_i}}
\leq
K(c) \bigg(\alpha_i^4 \frac{1}{n^2}b_{Q_2}(2;X_i) \frac{1}{t^2}+\frac{1}{n^2} \frac{b_L(4;X_i)}{t^2}+\alpha_i^2 \frac{1}{n} \frac{b_L(2;X_i)}{t}\\
+\frac{1}{n^2}\frac{R_i^4}{n^2} \frac{\sqrt{b_{Q_2}(4;X_i)}}{t^2} \frac{\sqrt{b_L(8;X_i)}}{t^2}+\frac{1}{n^2}\alpha_i^2 \frac{R_i^2}{n} \frac{\sqrt{b_{Q_2}(4;X_i)}}{t^2} \frac{\sqrt{b_L(4;X_i)}}{t}\bigg)\,,
\end{multline*}
where $K(c)$ denotes a numerical constant which depends on $c$.
Since the right-hand side is deterministic,
the same bound holds for the unconditional expectation.

On the complementary set $G_i^C$, we can use
the fact that $|T - T_i| \leq 1 + \alpha_i^2$ to obtain
$$
\Exp{|T-T_i|^2 \, \bm{1}_{G_i^C}}
\leq
K \Proba{G_i^C}
\leq
K C_0 \exp(-c_0n) \,.
$$
Summing over $i=1,\hdots,n$ and recalling our assumptions, we conclude that
$$
\sum_{i=1}^{n} \Exp{g(\alpha;X) - \Exp{g(\alpha;X)}}^2
=
\gO(n^{-1}) \,.
$$
Similar considerations can be made for the Lindeberg approach,
with the result that
$$
|\Exp{g(\alpha;X) - g(\alpha;Y)}|
=
\gO(n^{-1/2}) \,.
$$

\section{Relevance to statistical problems}
As discussed in the introduction, many quantities of statistical interest can be analyzed using our results. We will find deterministic equivalents for them. To keep the presentation readable for readers interested more in the applications than in the theory, we do not repeat the assumptions of our theorems. So all our statements should be understood as being prefaced: ``assuming that the technical conditions led our earlier in the paper are satisfied, we have...".

What the reader should essentially know is that shrinking the sample covariance matrix to a deterministic matrix $A$ has the effect of essentially shrinking a scaled version of the population covariance to the same matrix $A$. The damping factor depends on $A$ and $\Sigma$ and is estimable. When the mean is also estimated, the results of Subsection \ref{subsec:quadFormsInverseAndMean} need to be applied.

Our results show the remarkable robustness of random matrix results - we need very little control over the particulars of the data distributions - though they highlight their sensitivity to geometric assumptions. We now give a few examples where these computations are relevant and shed light on statistical matters.

\subsection{Estimation issues}
\subsubsection{Estimation of $v\trsp (\Sigma+A)^{-1}v$ when $\Sigma$ is not observed directly}
The motivation for this kind of question comes from understanding the population behavior of certain statistical procedures from observed data and hence deriving benchmarks as to how well a procedure could do. This could be used in evaluating a kind of regret, directly from the data.

Recall our general setting, namely we observe
$$
\ellipticalData_i=\mu+R_i X_i\;,
$$
where $R_i$ are possibly random and $X_i$ are random with distributions satisfying ``our usual assumptions" (see Subsection \ref{subsec:SetupOfStudy}). In particular, $X_i$'s have mean 0. Recall also the notation ${\cal S}=\frac{1}{n}\sum_{i=1}^n R_i^2 X_i X_i\trsp$ and $\muTilde=\frac{1}{n}\sum_{i=1}^n R_i X_i$.

We have shown that we can find a deterministic equivalent to $v\trsp ({\cal S}+A)^{-1}v$, namely,
$$
v\trsp({\cal S}+A)^{-1}v\simeq v\trsp (\gamma(A)\Sigma+A)^{-1} v\;.
$$
We first note that since $\SigmaHat={\cal S}-\muTilde\muTilde\trsp$,
$$
v\trsp (\SigmaHat+A)^{-1} v=v\trsp({\cal S}+A)^{-1}v +\frac{(v\trsp ({\cal S}+A)^{-1} \muTilde)^2}{1-\muTilde\trsp({\cal S}+A)^{-1}\muTilde}
\simeq v\trsp({\cal S}+A)^{-1}v\simeq v\trsp (\gamma(A)\Sigma+A)^{-1}v \;,
$$
as we have seen that $v\trsp ({\cal S}+A)^{-1} \muTilde\simeq 0$.

Now in certain situation, for instance to when we want to estimate the optimal risk of certain statistical procedures, we will need to estimate $v\trsp (\Sigma+A)^{-1}v$. We now sketch how to come up with an estimator of this quantity.

Let $t$ be a real in $\mathbb{R}_+$. We clearly have
$$
v\trsp (\SigmaHat+tA)^{-1} v\simeq v\trsp (\gamma(tA)\Sigma+tA)^{-1}v=\frac{v\trsp (\Sigma+tA/\gamma(tA))^{-1}v}{\gamma(tA)}\;.
$$

Now recall that
$$
\gamma(A)=\frac{1}{n}\sum_{i=1}^n \frac{R_i^2}{1+\frac{R_i^2}{n}\trace{\Sigma({\cal S}+A)^{-1}}}=
\frac{1}{n}\sum_{i=1}^n \frac{R_i^2}{1+R_i^2\alpha(A)}\;,
$$
and under our assumptions, $\gamma(A)$ has an asymptotically deterministic equivalent. Note that under concentration assumptions on $X_i$'s,
$$
\frac{X_i\trsp ({\cal S}_i+A)^{-1}X_i}{n}\simeq \frac{\trace{\Sigma ({\cal S}+A)^{-1}}}{n}\;,
$$
and using rank-1 update,
$$
\frac{R_i^2}{n}X_i\trsp ({\cal S}_i+A)^{-1} X_i=\frac{R_i^2}{n}\frac{X_i\trsp ({\cal S}+A)^{-1}X_i}{1-\frac{R_i^2}{n} X_i\trsp ({\cal S}+A)^{-1}X_i}\;,
$$
so we need only invert $({\cal S}+A)^{-1}$ once to compute efficiently all the terms we are interested in. (Of course in practice, we do not have access to $R_i X_i$, so we will use $Y_i-\muHat=R_i X_i-\muTilde$. Because $\muTilde\trsp ({\cal S}+A)^{-1}\muTilde$ is of order 1 and we will be dividing everything by $n$, we can neglect this term in this discussion. The same applies to terms of the form $\muTilde\trsp ({\cal S}+A-\muTilde\muTilde\trsp)^{-1}X_i$).

So we can now estimate $R_i^2 \alpha(A)$, and using the fact that
$$
\gamma(A)=\frac{1}{\alpha(A)}(1-\frac{1}{n}\sum_{i=1}^n \frac{1}{1+R_i^2\alpha(A)})\;,
$$
we can also estimate $\gamma(A)$.

So to estimate $v\trsp (\Sigma+A)^{-1} v$, all we need to do is find $t$ such that
$$
\frac{\gamma(t_0 A)}{t_0}=1\;.
$$
We will now show that $\gamma(tA)/t$ is decreasing; hence a simple dichotomous search will yield a fast algorithm for finding this $t_0$.

We note that
$$
\frac{\gamma(t A)}{t}=\frac{1}{n}\sum_{i=1}^n \frac{R_i^2}{t+\frac{R_i^2}{n}\trace{\Sigma({\cal S}/t+A)^{-1}}}\;.
$$
Now $({\cal S}/t+A)^{-1}$ is clearly increasing in the Loewner order, and hence so is $\trace{\Sigma({\cal S}/t+A)^{-1}}$ since we are dealing with positive semi-definite matrices. Therefore,
$$
\frac{\gamma(tA)}{t} \text{ is decreasing }\;.
$$
We note that its limit is 0 at infinity and infinity at 0. Hence the equation
$$
\frac{\gamma(tA)}{t}=1 \text{ has a unique solution, } t_0\;.
$$

We now have found an estimator of $v\trsp (\Sigma +A)^{-1} v$, since
$$
v\trsp (\Sigma +A)^{-1} v \simeq t_0 v\trsp (\SigmaHat+t_0 A)^{-1} v\;.
$$

\subsection{Classification}
Random matrix techniques offer us insights into the behavior of standard methods in high-dimension. Our work could be helpful in tuning regularization parameters, picking penalties etc... because we are able to predict performance of the methods, provided our assumptions are met. It is nonetheless clear that sometimes (actually many times), some of the quantities we are considering
could be evaluated by leave-one out methods, which can be implemented efficiently because of rank-1 updates. In that case, our analysis has the merit of explaining the behavior of the techniques considered, something that alternative estimators (such as cross-validation) do not do.

A standard technique in classification is linear discriminant analysis. Some analysis in the high-dimensional context has already been done (\cite{bl03}), in a somewhat different direction. Here our aim is to explain what creates problems with LDA in high-dimension, even in the Gaussian case, and discuss briefly the behavior of Regularized discriminant analysis (RDA) proposed in \cite{FriedmanRDA89}.

\subsubsection{A preliminary remark}
\newcommand{\pTilde}{\tilde{p}}
In the classification context (see details below), we will often be faced with a situation where a (regularized) covariance matrix is a pooled estimator of covariance computed from two groups, i.e
$$
\SigmaHat=p_1\SigmaHat_1+p_2\SigmaHat_2\;.
$$

In our context we will assume that the observations in each group have the same mean $\mu_i$, where $\mu_i$ may depend on $i=1,2$. Assuming that the data is of the form
$$
\ellipticalData_k=\mu_i + R_k X_k\;,
$$
where $X_k$ has mean 0, we have for instance
$$
\SigmaHat_1=\frac{1}{N_1} \sum_{k=1}^{N_1} R_k^2 X_k X_k\trsp-\muTilde_1\muTilde_1\trsp \;,
$$
where $X_k$ have mean 0, and
$$
\muTilde_1=\frac{1}{N_1} \sum_{i=1}^{N_1} R_k X_k\;.
$$

We will naturally encounter forms of the type
$$
(\muHat_2-\muHat_1)\trsp (\SigmaHat+A)^{-1}(\muHat_2-\muHat_1)
$$
and we now explain how to find deterministic equivalents for the limiting behavior of these forms. We note that $\muHat_i=\mu_i+\muTilde_i$, so we will have to work out three quantities:
$$
(\mu_2-\mu_1)\trsp (\SigmaHat+A)^{-1}(\mu_2-\mu_1)\;, (\mu_2-\mu_1)\trsp (\SigmaHat+A)^{-1}(\muTilde_2-\muTilde_1) \text{ and } (\muTilde_2-\muTilde_1)\trsp (\SigmaHat+A)^{-1}(\muTilde_2-\muTilde_1)\;.
$$
The first one is simple as it involves a shrunken matrix and deterministic vectors. The other two are a bit more subtle, since $\SigmaHat$ and $\muTilde_i$'s interact (the Gaussian case being an exception for obvious reasons).

We call
$$
{\cal S}_1=\frac{1}{N_1-1} \sum_{k=1}^{N_1} R_k^2 X_k X_k\trsp \text{ and } {\cal S}_2=\frac{1}{N_2-1} \sum_{k=N_1+1}^{N_1+N_2} R_k^2 X_k X_k\trsp\;,
$$
and, if $p_i=(N_i-1)/(N_1+N_2-2)$,
$$
{\cal S}=p_1{\cal S}_1+p_2{\cal S}_2\;.
$$

We note that, more generally we will have, for some $\pTilde_i$ (for instance $\pTilde_i=p_i N_i/(N_i-1)$ if we wish to preserve unbiasedness),
$$
\SigmaHat+A={\cal S}+A-\pTilde_1\muTilde_1 \muTilde_1\trsp -\pTilde_2 \muTilde_2\muTilde_2\trsp\;,
$$
where
$$
{\cal S}=\frac{1}{N_1+N_2-2} \sum_{i=1}^{N_1+N_2} R_i^2 X_iX_i\trsp\;.
$$

$\bullet$ \textbf{On $\bm{(\muTilde_2-\muTilde_1)\trsp (\SigmaHat+A)^{-1}(\muTilde_2-\muTilde_1)}$}\\
Using a rank-1 update formula (we could also use a more general version of the Sherman-Woodbury-Morrison formula), we have
$$
\muTilde_2 \trsp (\SigmaHat+A)^{-1}= \frac{\muTilde_2\trsp ({\cal S}+A-\pTilde_1\muTilde_1\muTilde_1\trsp)^{-1}}{1-\pTilde_2 \muTilde_2\trsp ({\cal S}+A-\pTilde_1\muTilde_1\muTilde_1\trsp)^{-1} \muTilde_2}\;.
$$
Therefore, we have in particular,
$$
\muTilde_2\trsp (\SigmaHat+A)^{-1}\muTilde_2=\frac{1}{\pTilde_2}\left(\frac{1}{1-\pTilde_2 \muTilde_2\trsp ({\cal S}+A-\pTilde_1 \muTilde_1\muTilde_1\trsp)^{-1} \muTilde_2}-1\right)\;.
$$
We also see by the same token that
$$
\muTilde_2\trsp ({\cal S}+A-\pTilde_1\muTilde_1\muTilde_1\trsp)^{-1} \muTilde_2=\muTilde_2\trsp ({\cal S}+A)^{-1} \muTilde_2+\pTilde_1 \frac{(\muTilde_2\trsp ({\cal S}+A)^{-1} \muTilde_1)^2}{1-\pTilde_1 \muTilde_1\trsp ({\cal S}+A)^{-1}\muTilde_1}\;.
$$

Now recall that (see Subsubsection \ref{subsec:quadFormsInverseAndMean})
$$
\muTilde_2\trsp(\SigmaHat+A)^{-1}\muTilde_1\simeq 0\;.
$$

So we conclude that
$$
\boxed{
\muTilde_2\trsp (\SigmaHat+A)^{-1}\muTilde_2\simeq \frac{\muTilde_2\trsp ({\cal S}+A)^{-1}\muTilde_2}{1-\pTilde_2\muTilde_2\trsp ({\cal S}+A)^{-1}\muTilde_2}\;.
}
$$
Naturally, our work in Subsubsection \ref{subsec:quadFormsInverseAndMean} allows us to find a deterministic equivalent to
$$
\muTilde_2\trsp ({\cal S}+A)^{-1}\muTilde_2
$$
and so from then we get a deterministic equivalent to $\muTilde_2\trsp (\SigmaHat+A)^{-1}\muTilde_2$. Of course, a similar analysis carries through for
$\muTilde_1\trsp ({\cal S}+A)^{-1}\muTilde_1$. To be more precise, if we call $1_{G_i}$ the vector that has 1 if $\ellipticalData_k$ is in group $i$ and 0 otherwise, we see that the $\alpha$ that corresponds to $\muTilde_2$ is
$$
\alpha=\frac{\sqrt{n}}{N_2} 1_{G_2}\;,
$$
and we can apply our formulas.

We also need to consider
$$
\muTilde_2\trsp ({\cal S}+A)^{-1}\muTilde_1\;.
$$
Using the rank-1 update formula, we have
$$
\muTilde_2\trsp ({\cal S}+A)^{-1}\muTilde_1=\frac{\muTilde_2\trsp ({\cal S}+A-\pTilde_1\muTilde_1\muTilde_1\trsp)^{-1}\muTilde_1}{1-\pTilde_2 \muTilde_2\trsp ({\cal S}+A-\pTilde_1\muTilde_1\muTilde_1\trsp)^{-1} \muTilde_2}\;.
$$
We have already worked out an approximation to the denominator. Now for the numerator, we have obviously
$$
\muTilde_2\trsp ({\cal S}+A-\pTilde_1\muTilde_1\muTilde_1\trsp)^{-1}\muTilde_1=\frac{\muTilde_1\trsp ({\cal S}+A)^{-1}\muTilde_2}{1-\pTilde_1 \muTilde_1\trsp ({\cal S}+A)^{-1}\muTilde_1}\;.
$$
Hence, again, we see that in the asymptotic limit we consider,
$$
\muTilde_2\trsp ({\cal S}+A-\pTilde_1\muTilde_1\muTilde_1\trsp)^{-1}\muTilde_1\simeq 0\;.
$$

So we conclude that
$$
\boxed{
\muTilde_2\trsp (\SigmaHat+A)^{-1} \muTilde_1\simeq 0\;.
}
$$

$\bullet$ \textbf{On $\bm{(\muTilde_2-\muTilde_1)\trsp (\SigmaHat+A)^{-1}\mu}$}\\
The idea is here again to use our rank-1 update formulas. We have
$$
\muTilde_2 \trsp (\SigmaHat+A)^{-1}\mu=\frac{\muTilde_2\trsp ({\cal S}+A-\pTilde_1\muTilde_1\muTilde_1\trsp)^{-1}\mu}{1-\pTilde_2 \muTilde_2\trsp ({\cal S}+A-\pTilde_1\muTilde_1\muTilde_1\trsp)^{-1} \muTilde_2} \;.
$$
We also have
$$
\muTilde_2\trsp ({\cal S}+A-\pTilde_1\muTilde_1\muTilde_1\trsp)^{-1}\mu=\muTilde_2\trsp ({\cal S}+A)^{-1}\mu +\pTilde_1 \frac{\muTilde_2\trsp ({\cal S}+A)^{-1}\muTilde_1 \muTilde_1\trsp ({\cal S}+A)^{-1}\mu}{1-\pTilde_1\muTilde_1\trsp({\cal S}+A)^{-1}\muTilde_1}\;.
$$
So we conclude that if, for instance $\norm{\mu}$ stays bounded,
$$
\boxed{
\muTilde_2 \trsp (\SigmaHat+A)^{-1}\mu\simeq 0\;.
}
$$

We now have all the elements needed to get an asymptotically deterministic approximation to
$$
(\muHat_2-\muHat_1)\trsp (\SigmaHat+A)^{-1}(\muHat_2-\muHat_1)\;.
$$

\subsubsection{LDA: Gaussian case}

We recall the (optimal) setup. Suppose we have two groups (or classes). The observations can come from group 1 or group 2. In both groups they are ${\cal N}(\mu_{1,2},\Sigma)$. The probability of belonging to group 1 is $\pi_1$. The question is now given an observation, how should it be classified?

It is easy and standard to find the optimal rule in the population. Namely, by doing likelihood computations, one quickly realizes that the optimal classification rule is (\cite{htf09}): classify an observation as belonging to Group 2 if
$$
x\trsp \Sigma^{-1}(\mu_2-\mu_1)\geq \frac{1}{2} (\mu_2-\mu_1)\trsp \Sigma^{-1} (\mu_2+\mu_1)+\log(\pi_1/\pi_2)\;.
$$

Naturally, in practice, $\Sigma$ and $\mu_1$ and $\mu_2$ need to be estimated. A natural solution is to use the training data (which is labeled, i.e we know to which class each observation belongs) to estimate $\mu_1$ and $\mu_2$ and then use a pooled estimate of covariance for $\Sigma$.

In somewhat more details, if we have $N_1$ observations that belong to class 1 in our training set, and $N_2$ that belong to class 2, let us denote by $\muHat_1$ and $\muHat_2$ the sample mean of the observations in group 1 and group 2. If $\SigmaHat_1$ and $\SigmaHat_2$ are the sample covariance in each of these groups, then our estimate of $\Sigma$ is
$$
\SigmaHat=\frac{1}{N_1+N_2-2} \left((N_1-1) \SigmaHat_1+(N_2-1) \SigmaHat_2\right)\;.
$$
(We will assume in the following discussion that $p\leq N_1+N_2-2$ so $\SigmaHat$ is invertible.)
It is now natural to ask the following questions:
\begin{enumerate}
\item how does naive LDA perform?
\item how suboptimal is the naive threshold?
\item is it possible to estimate the minimal misclassification rate, even if we cannot find the optimal direction on which to project a new observation?
\end{enumerate}

Naturally, when a Gaussian vector is projected on a direction $d$, its distribution becomes ${\cal N}(\mu\trsp d, d\trsp \Sigma d)$. If our decision rule is to classify $x$ to Group 2 if $x\trsp d>t$, it is clear that the misclassification rate is, if $\mu_1(d)=\mu_1\trsp d$, $\mu_2(d)=\mu_2\trsp d$ and $\sigma^2(d)=d\trsp \Sigma d$,
$$
\pi_1(1-\Phi(\frac{t-\mu_1(d)}{\sigma}))+\pi_2 \Phi(\frac{t-\mu_2(d)}{\sigma(d)})\;.
$$
A simple computation therefore shows that the optimal threshold is
$$
t^*=\frac{\sigma^2(d)}{\mu_2(d)-\mu_1(d)}\log(\pi_1/\pi_2)+\frac{\mu_2(d)+\mu_1(d)}{2}\;.
$$
Hence we have
$$
\frac{t^*-\mu_{1,2}}{\sigma}=\pm\frac{\mu_2-\mu_1}{2\sigma}+\frac{\sigma}{\mu_2-\mu_1} \log(\frac{\pi_1}{\pi_2})\;.
$$
We can therefore compute the optimal misclassification rate as
$$
\pi_1(1-\Phi(\frac{\mu_2-\mu_1}{2\sigma}+\frac{\sigma}{\mu_2-\mu_1} \log(\frac{\pi_1}{\pi_2}))+\pi_2 \Phi(-\frac{\mu_2-\mu_1}{2\sigma}+\frac{\sigma}{\mu_2-\mu_1} \log(\frac{\pi_1}{\pi_2}))\;.
$$
Note that in LDA in the population, we have $\mu_2(d)-\mu_1(d)=(\mu_2-\mu_1)\trsp \Sigma^{-1}(\mu_2-\mu_1)=\sigma^{2}$. Hence some simplifications ensue; in particular, the optimal misclassification rate is, if $\sigma$ is the Mahalanobis distance between $\mu_2$ and $\mu_1$,
$$
\pi_1-\pi_1\Phi(\frac{\sigma}{2}+\frac{1}{\sigma}\log(\pi_1/\pi_2))+\pi_2\Phi(-\frac{\sigma}{2}+\frac{1}{\sigma}\log(\pi_1/\pi_2))\;.
$$

Hence, our problems reduce to:
\begin{enumerate}
\item Estimate the Mahalanobis distance between $\mu_1$ and $\mu_2$ so we can compute the optimal misclassification rate for the problem
\item Estimate $t^*$ from the data to obtain a procedure that outperforms the naive procedure.
\end{enumerate}
We note that it is good practice to do cross-validation to estimate $t^*$ - and this has been recognized by practitioners, see \cite{htf09}. However, even when the data is Gaussian, as we show below, a correction to the naive empirical threshold is needed in high-dimension.

$\bullet$ \textbf{Estimation of }$\mathbf{t^*}$.
When $d=\SigmaHat^{-1} (\muHat_2-\muHat_1)$, we have
\begin{align*}
\sigma^2(d)&=(\muHat_2-\muHat_1)\trsp \SigmaHat^{-1} \Sigma \SigmaHat^{-1}(\muHat_2-\muHat_1)\;,\\
\mu_i(d)&=\mu_i\trsp \SigmaHat^{-1} (\muHat_2-\muHat_1)\;.
\end{align*}

In the Gaussian case, using properties of Wishart matrices (the interested reader is also refered to \cite{nekMarkoRealizedRisk} for similar computations, but going beyond the Wishart case), we see that, if $\rho=p/N$,
$$
\sigma^2(d)\simeq (\mu_2-\mu_1)\trsp \Sigma^{-1} (\mu_2-\mu_1)\frac{1}{(1-\rho)^3}+\frac{1}{(1-\rho)^3} (\frac{p}{N_1}+\frac{p}{N_2})
$$
On the other hand,
$$
\mu_i(d)\simeq \frac{1}{1-\rho}\mu_i\trsp \Sigma^{-1}(\mu_2-\mu_1)\;.
$$

Now from the data we can get an estimate of $(\muHat_2-\muHat_1)\trsp \SigmaHat^{-1}(\muHat_2-\muHat_1)$. A simple computation, based on properties of Wishart matrices (see e.g \cite{nekMarkoRisk} for full details) gives:
$$
(\muHat_2-\muHat_1)\trsp \SigmaHat^{-1}(\muHat_2-\muHat_1)\simeq \frac{1}{1-\rho}\left[(\mu_2-\mu_1)\trsp \Sigma^{-1} (\mu_2-\mu_1) +\frac{p}{N_1}+\frac{p}{N_2}\right]\simeq (1-\rho)^2 \sigma^2(d)\;.
$$
On the other hand,
\begin{align*}
\muHat_2\trsp \SigmaHat^{-1} (\muHat_2-\muHat_1)& \simeq \frac{1}{1-\rho}\left[\mu_2\trsp \Sigma^{-1}(\mu_2-\mu_1)+\frac{p}{N_2}\right]\;, \\
\muHat_1\trsp \SigmaHat^{-1} (\muHat_2-\muHat_1)& \simeq \frac{1}{1-\rho}\left[\mu_1\trsp \Sigma^{-1}(\mu_2-\mu_1)-\frac{p}{N_1}\right]\;.
\end{align*}
So we can estimate $\mu_i(d)$ by
$$
\mu_i(d)\simeq  \muHat_i\trsp \SigmaHat^{-1} (\muHat_2-\muHat_1)\pm\frac{p}{N_i(1-\rho)}\;,
$$
where $\pm$ is 1 for $i=1$ and $\pm=-1$ for $i=2$.

We can now estimate $t^*$ by putting together all these estimators. (We note that we could also do this by using estimate of $\sigma^{2}(d)$ and $\mu_i(d)$ based on leave-one out procedures. However, the advantage of the procedure proposed here is that the amount of extra computations is extremely small, since the corrections are known in closed form.)

On the other hand, it is clear that the naive threshold value is (in general) suboptimal. As a matter of fact, it is approximately
$$
t_{\text{naive}}\simeq \frac{1}{2}\left[\frac{1}{1-\rho}(\mu_2-\mu_1)\trsp \Sigma^{-1}(\mu_2+\mu_1)+\frac{1}{1-\rho} \left[\frac{p}{N_2}-\frac{p}{N_1}\right]\right]+\log(\pi_1/\pi_2)\;.
$$
On the other hand, if $\text{maha}=(\mu_2-\mu_1)\trsp \Sigma^{-1}(\mu_2-\mu_1)$,
$$
t^*\simeq \frac{1}{2}\frac{1}{1-\rho}(\mu_2-\mu_1)\trsp \Sigma^{-1}(\mu_2+\mu_1)+\log(\pi_1/\pi_2)\frac{1}{(1-\rho)^2}\left[1+\left(\frac{p}{N_1}+\frac{p}{N_2}\right)\frac{1}{
\text{maha}}\right]\;.
$$
Let us further remark that when $N_1=N_2$, because $\log(\pi_1/\pi_2)=0$, our correction returns exactly the naive threshold, and hence will not yield improvements. On the other hand, in this situation, the naive threshold is close to optimal and our analysis shows that further numerical investigation of a good threshold is not needed.

In other respects, it is rather easy to estimate $(\mu_2-\mu_1)\trsp \Sigma^{-1} (\mu_2-\mu_1)$, and hence get the optimal misclassification rate for any classification procedure, in the case where the data is truly Gaussian. Note that this is not available by using cross-validation.

Hence, beside shedding light on the potential (limited) problems of LDA in high-dimension, the computations we showed can be used to establish a benchmark for how well a classification procedure can perform and perhaps helps the user in choosing something better than LDA - or convincing her that LDA (perhaps corrected) in her context is performing quite well and close to the optimum.

\subsubsection{``LDA": elliptical case}
We are now interested in finding a reasonable classification procedure for elliptical data in high-dimension. We will see that the results obtained in this paper are relevant to shed light on their behavior.

We consider the case here where $R_i$'s have a smooth density. The data is modeled as
$$
\ellipticalData_i=\mu_{1,2}+R_i X_i\;.
$$
We will focus on the case $X_i\sim {\cal N}(0,\Sigma)$, though some of the computations could be carried in a more complex situation. Let us call $f$ the density of $R$.
The density of
$$
\ellipticalData=\mu_{1,2}+R_i X_i\;,
$$
is, since it is a continuous scale mixture of normal,
$$
\phi(y;\mu)=\int f(r) r^{-p} \frac{\exp\left(-\frac{(y-\mu)\trsp \Sigma{-1} (y-\mu)}{2r^2}\right)}{\sqrt{\det(2\pi \Sigma)}}dr\;.
$$
Hence, it is difficult to get an exactly optimal classification rule by using a likelihood method. Nonetheless, we can apply Laplace's method to approximate this integral.

We now recall the model from which $\ellipticalData$ is generated and we see that $\frac{(\ellipticalData-\mu)\trsp \Sigma^{-1} (\ellipticalData-\mu)}{p}$ is concentrated around $R^2$ if $(\mu_{1,2}-\mu)\Sigma^{-1} (\mu_{1,2}-\mu)=\gO(1)$.

We are going to make the assumption that $(\mu_2-\mu_1)\trsp \Sigma^{-1} (\mu_2-\mu_1)=\gO(1)$. Calling, for $y$ a dummy variable assumed to take values only where $\ellipticalData$ concentrates,
$$
\alpha_p(i)=\frac{(y-\mu_i)\trsp \Sigma^{-1}(y-\mu_i)}{p}\;,
$$
we see that $\alpha_p(i)=\gO(1)$ (indeed $\alpha_p(i)\simeq R^2$; see the remark on $\ellipticalData$ above) and
$$
\left|\alpha_p(1)-\alpha_p(2)\right|=\gO(1/p)\;.
$$
Hence applying Laplace's method, we see that
$$
\phi(y;\mu_i)\sim f(\sqrt{\alpha_p(i)})\exp(-p/2(\log(\alpha_p(i))+1))\sqrt{\pi \alpha_p(i)/p}\;.
$$
Hence, under our assumptions, if $\Delta=p (\alpha_p(2)-\alpha_p(1))$ (which is of order 1),
$$
\frac{\phi(y;\mu_1)}{\phi(y;\mu_2)}\simeq \exp(\Delta/(2\alpha_p(1)))\;.
$$
Now $-p\Delta=2(\mu_2-\mu_1)\trsp \Sigma^{-1}y+\mu_1\trsp \Sigma^{-1}\mu_1-\mu_2\trsp \Sigma^{-1}\mu_2$. Hence, if the prior probabilities are $\pi_1$ and $\pi_2$, a reasonable rule for classification appears to be: classify in group 2 if, for a new observation $y$,
\begin{equation}\label{eq:LDARuleModifiedElliptical}
y\trsp \Sigma^{-1}(\mu_2-\mu_1)\geq \alpha_p(1) \log\left(\frac{\pi_1}{\pi_2}\right)+\frac{(\mu_2-\mu_1)\trsp \Sigma^{-1}(\mu_2+\mu_1)}{2}\;.
\end{equation}

Here, in what is perhaps a surprise, we see that in high-dimension, in the class of elliptical distribution a procedure similar to LDA seems quite reasonable.

Under our assumptions, it should be noted that in high-dimension, $\alpha_p(1)\simeq R_i^2$ (for new data generated according to our model). Therefore this rule consistent with LDA, since for Gaussian data $R_i^2=1$. Now, if $\norm{\mu_2-\mu_1}^2\ll \trace{\Sigma}$, we see that
$$
\frac{\norm{\ellipticalData-\mu_i}^2}{\trace{\Sigma}}\simeq R_i^2\;,
$$
hence, the rule is approximately implementable - though situations where $R_i$ has very heavy tails  are likely to be very hard on these approximations.

Now in the elliptical case, we know (see \cite{nekMarkoRisk}) that there exists $\limitScaling$ such that if $\ellipticalData$ is independent of $\SigmaHat$, $\muHat_1$ and $\muHat_2$,
$$
\ellipticalData\trsp \SigmaHat^{-1}(\muHat_2-\muHat_1)\simeq \limitScaling \ellipticalData\trsp \Sigma^{-1}(\mu_2-\mu_1)
$$
and we can also find an approximation of
$$
\frac{(\mu_2-\mu_1)\trsp \Sigma^{-1}(\mu_2+\mu_1)}{2}
$$
through appropriate corrections, the key computations having been carried out in \cite{nekMarkoRisk}. 
Hence, we can design a classification rule by using (nearly) unbiased estimators of the quantities on both sides of Equation \eqref{eq:LDARuleModifiedElliptical}. This could naturally also be done using leave-one out procedures, though these procedures would not explain what is happening.

$\bullet$\textbf{On changing estimators of covariance}\\ One advantage of the analyses we have carried out is that they reveal (somewhat explicitly) the role played by the $R_i$'s. Since those are essentially estimable (for instance in the Gaussian, and in general as soon as we have measure concentration), we could also envision different weighting schemes, in particular putting all of them to 1 (which is extremely natural from a convexity standpoint), which amounts to using estimators which are similar in spirit to Tyler's estimator (see \cite{TylerRobustEstimateScatterAoS87} and \cite{nekMarkoRisk} for more details.) Because the paper is already quite long, we will not seek an optimal procedure here, but our various estimates (here and in \cite{nekMarkoRisk}, \cite{nekMarkoRealizedRisk}) can in principle be used to assess difference in performance between these estimators of covariance for the statistical tasks at hand.

$\bullet$ \textbf{Computing the misclassification rate in the elliptical setting}\\
Suppose we now use a simple threshold rule, similar to LDA, to classify. Though this is suboptimal, understanding the behavior of this simple rule is interesting, and helps shed light on various procedures and their robustness.

So suppose we classify an observation $x$ to Group 2 if $x\trsp v>t$. Suppose that $x$ is elliptical and call $f$ the density of the $R$. A computation similar to the ones carried before shows that the misclassification rate is
$$
\pi_1 \int f(r) \Phi\left(\frac{\mu_1-t}{\sqrt{v\trsp \Sigma v}r}\right)dr+\pi_2 \int f(r) \Phi\left(\frac{t-\mu_2}{\sqrt{v\trsp \Sigma v}r}\right) dr\;.
$$
Since we are able to estimate $\mu_i$'s, and $v\trsp \Sigma v$, as well as (at least coarsely) the density $f$ - since we can estimate the $R_i$'s, we can find the optimal threshold $t^*$. This gives a principled alternative to cross-validation in this case (though leave-one-out techniques could also be used).

\subsubsection{RDA}
In \cite{FriedmanRDA89}, partly motivated by questions having to do with the variability of LDA procedures (in particular when $\Sigma$ is ill-conditioned), it was proposed to replace $\SigmaHat$ by
$$
\SigmaTilde=(1-w)\SigmaHat+w A\;,
$$
where $A$ is a matrix towards which $\Sigma$ is shrunken. The computations done in the first part of the paper allow us to measure the performance of RDA in our asymptotic context.

Our results show that when $w$ varies from $0$ to $1$, up to a computable scaling factor, forms of the type $v\trsp \SigmaTilde v$ cover the range of  $v\trsp [(1-\lambda)\Sigma+ \lambda A]^{-1}v$, for $\lambda$ varying from 0 to 1, though of course $\lambda$ is very different from $w$ (and $\lambda$ depends on the ellipticity of the data). This property is something that is not immediately obvious in high-dimension. This is valid much beyond the Gaussian design case, as we have shown.

Let us now illustrate this in the Gaussian case. In this case, we know how to pick the optimal threshold at given $w$ and can compute the misclassification rate of the corresponding procedure. Our results also show that the naive threshold is suboptimal, and suggests corrections, though those can also be found using leave-one out procedures that do not rely on our understanding of the phenomena. (This is fairly similar to our more detailed LDA discussion.)

Our computations also show that one should probably not use 5 or 10 fold cross-validation methods in high-dimension, since it affects that the ratio $p/n$, which is key in determining and getting optimal performance.

Here again, a rigorous study of the impact of $R_i$ on the quality of classification and the potential benefits of using robust estimate of scatter is now feasible but we postpone it to other investigations because of the length of this paper.

\subsection{Optimization problems}
Suppose we consider the optimization problem
$$
\left\{
\begin{array}{c}
\min_w w\trsp \Sigma w\\
\text{ subject to } V'w=U
\end{array}
\right.\;,
$$
where $V$ is a $p\times k$ matrix of constraints, and $U$ is a $k\times 1$ vector of values for those constraints.
This is a canonical problem in portfolio optimization (see \cite{MeucciBook}, \cite{Markowitz1952}).
Under minimal invertibility conditions, the solution is
$$
\wOpt=\Sigma^{-1}V M^{-1} U\;,
$$
where $M=V\trsp\Sigma^{-1}V$.

Suppose that we estimate $\Sigma$ by $\SigmaTilde=\lambda \SigmaHat+A$ and suppose that $V$ contains a constraint involving $\mu$, which is not known and needs to be estimated. Call $\wHat$ the corresponding solution and $\MHat=\VHat\trsp f(\SigmaHat)^{-1} \VHat$. Then our estimates allow us to get a deterministic equivalent to the naive estimate of the risk, namely, $\wHat\trsp \SigmaTilde \wHat=U\trsp \MHat U$ as well as the true risk of our allocation, i.e $\wHat\trsp \Sigma \wHat$, at least when the number of constraints is fixed.

Let us now be a bit more specific. Suppose $\SigmaTilde=\SigmaHat+A$ (scalar constants can easily be dealt with), that the number of constraints is fixed and $V$ contains only fixed constraints (i.e nothing needs to be estimated, and in particular not the mean - this is for instance the case when users in Finance perform minimum variance optimization, without regards for expected returns). Then,
$$
\MHat\simeq V\trsp (\gamma(A)\Sigma+A)^{-1} V=\MTilde_A,
$$
so we get as deterministic equivalent of the naive risk
$$
U\trsp (V\trsp (\gamma(A)\Sigma+A)^{-1} V)^{-1} U\;.
$$
The interpretation of this result is that the shrinkage procedure essentially produces an estimator which is a dampen shrinkage estimator, the damping factor being $\gamma(A)$.

To compute the realized risk, all one needs to do is look at $U\trsp \MHat^{-1} V\trsp (\SigmaHat+A)^{-1} \Sigma (\SigmaHat+A)^{-1} V \MHat^{-1} U$. To understand this, we can just rely on the results of Heuristic \ref{heuristic:squareRegularizedInverse}, with $B=\Sigma$. It should be noted that
$$
V\trsp (\SigmaHat+A)^{-1} \Sigma (\SigmaHat+A)^{-1} V \simeq [1+\xi(A,\Sigma)] V\trsp \Sigma V=[1+\xi(A,\Sigma)]M\;.
$$
Hence,
$$
U\trsp \MHat^{-1} V\trsp (\SigmaHat+A)^{-1} \Sigma (\SigmaHat+A)^{-1} V \MHat^{-1} U\simeq [1+\xi(A,\Sigma)] U\trsp \MTilde_A^{-1} M \MTilde_A^{-1} U\;.
$$

The situation where $V$ involves $\mu$ and is replaced by $\muHat$ in $\VHat$ can be investigated using our results on quadratic forms in $DX (X\trsp D^2 X+A)^{-1} X\trsp D$ and the other results we developed in the paper specifically for this task.

Finally, to the reader who might wonder why the study of  $M^{-1}\Sigma_{\eps}M^{-1}$ is potentially useful, even in the setting where $\ellipticalData_i$ are  i.i.d and hence have the same covariance $\Sigma$, let us give a ``practical" example: it is sometimes the case that in the context of portfolio optimization, one uses log-returns instead of returns to find the portfolio weights. This is found to be natural when the stock prices follow geometric brownian motions, as in the Black-Scholes model. But clearly, in that setting of log-normal prices,  the risk exposure should be computed using the covariance of the returns and not that of the log returns - two matrices that are in general different. (Note that our results (and our work on log-normal distributions) also give risk predictions when using returns instead of log returns when working with log-normal data.)

\subsection{Ridge regression}
Suppose we consider ridge regression with a general quadratic penalty (a.k.a Tikhonov regularization). Then $\betaHat$ is found by solving
$$
\betaRidge=\argmin_{\beta} \norm{Y-\frac{1}{\sqrt{n}}X\beta}_2^2+\lambda \beta\trsp \Gamma \beta\;,
$$
where $Y$ is our response, $X$ is the design matrix and $\Gamma$ is a psd matrix. It is easy to verify that
$$
\betaRidge=\frac{1}{\sqrt{n}}(\frac{1}{n}X\trsp X+\lambda \Gamma)^{-1}X\trsp Y\;.
$$
Suppose that $Y=\frac{1}{\sqrt{n}}[X\beta_0+\eps]$. Then,
$$
\betaRidge=(\frac{X\trsp X}{n}+\lambda \Gamma)^{-1}(\frac{X\trsp X}{n} \beta_0+\frac{X\trsp}{n} \eps)\;.
$$
Hence,
$$
\betaRidge-\beta_0=-\lambda (\frac{X\trsp X}{n}+\lambda \Gamma)^{-1}\Gamma \beta_0+(\frac{X\trsp X}{n}+\lambda \Gamma)^{-1}\frac{X\trsp}{n} \eps\;.
$$
The situation where the design is random can now be studied with our tools, provided the assumptions of our theorems are satisfied.

For instance if $\eps$ has covariance $\Sigma_{\eps}$ and mean 0, we have
$$
\Exp{\norm{\betaRidge-\beta_0}_2^2|X}=\lambda^2 \beta_0\trsp \Gamma\trsp (\frac{X\trsp X}{n}+\lambda \Gamma)^{-2} \Gamma \beta_0+\frac{1}{n}\trace{(\frac{X\trsp X}{n}+\lambda \Gamma)^{-1}\frac{X\trsp \Sigma_{\eps}X}{n} (\frac{X\trsp X}{n}+\lambda \Gamma)^{-1}}\;.
$$
The first quantity can be analyzed using our results in this paper. The second one is comparatively simpler and comes out of random matrix arguments. For instance, when $\Sigma_{\eps}=\id_n$, we see that we are left with
$$
\trace{(\frac{X\trsp X}{n}+\lambda \Gamma)^{-1}\frac{X\trsp X}{n} (\frac{X\trsp X}{n}+\lambda \Gamma)^{-1}}=\trace{(\frac{X\trsp X}{n}+\lambda \Gamma)^{-1}}-\lambda\trace{\Gamma(\frac{X\trsp X}{n}+\lambda \Gamma)^{-2}}\;,
$$
and these quantities can be analyzed using standard results on Stieltjes transforms (as well as the derivation trick we use repeatedly in this paper).

We also note that if $X$ has a symmetric distribution (we could relax of course this assumption with some work done along the lines of what is done in the paper),
$$
\frac{1}{n}\Exp{\trace{(\frac{X\trsp X}{n}+\lambda \Gamma)^{-1}\frac{X\trsp X}{n} (\frac{X\trsp X}{n}+\lambda \Gamma)^{-1}}}=\Exp{\frac{1\trsp}{\sqrt{n}} \frac{X}{\sqrt{n}} (\frac{X\trsp X}{n}+\lambda \Gamma)^{-2}\frac{X}{\sqrt{n}}\trsp \frac{1}{\sqrt{n}}}\;,
$$
and we can therefore use the work done in \ref{subsubsec:formsinDXMinvSigmaEpsMinvXtrspD}. A similar argument would hold if $\Sigma_{\eps}$ were diagonal, with $1$ replaced by $u$, with $u_i^2=\Sigma_{\eps}(i,i)$.

The arguments presented in this paper can also be used to understand the quantities $\norm{\betaRidge-\beta_0}_2^2$ directly, before taking expectation, if for instance we have a bound (with high-probability) on $\norm{\eps}$.

Our concentration arguments also allow us to show that
$$
\frac{1}{p}\Exp{\norm{\betaRidge(X,Y)-\beta_0}_2^2}
$$
has the same limit as the conditional version.

\section{Conclusion}

Our study aimed at showing that the tools of random matrix theory could be used to further our understanding  of various statistical procedures based on shrinkage estimators of covariance. Despite the great recent interest in $l_1$-type regularizations, these more classical methods are still very useful and very much in use, which is why we undertook the task of explaining what they actually did (at least asymptotically) in high-dimension. We also note that our study has moved us now quite far away from ``linear" models for the data and we have obtained results for distributions with genuinely non-linear structures, something that is very much needed to understand various practical applications.

We have both shown what we think is a great distributional robustness of random matrix based results in this context and a great geometric fragility of those models: distributional assumptions are largely irrelevant as long as they have the same geometric implications for the data; when two models yield a different geometry, the limiting approximations can change completely. Hence it seems to us that our study highlights a basic applied fact: namely users of random matrix results should run diagnostic tests before they apply (or rely on) results obtained in Gaussian or Gaussian like situations (which are the only ones covered by the ``classical" random matrix models). For otherwise, if there is e.g correlation between our $n$ observations, or if the geometry of the dataset does not conform to ``i.i.d Gaussian" geometry, naive random matrix predictions will prove unhelpful and uninformative at best.

On a technical note, our results are quite general, thanks in large part to the approach we used, which does not require us to compute the limit (or deterministic equivalent) of various quantities  to show it is the same when our data come from a wide class of possible distributions.  It should be noted that our results encompass many distributions for which natural questions in random matrix theory (such as behavior of largest and smallest eigenvalues) have not yet been settled or even investigated. In the future, it might also be of interest to look into more general estimates of covariance, namely matrix functions of the (shrunken) covariance matrix, i.e estimates that apply a certain fixed function to the eigenvalues of the shrunken matrix and leave the eigenvectors as is. This seems very approachable by our methods, using Cauchy's formula for instance, but because this might be considered a bit less central to multivariate statistics  we postpone a rigorous study of these questions to a possible future paper.

\appendix
\begin{center}
\textbf{\textsc{APPENDIX}}
\end{center}
\vspace{1cm}
\renewcommand{\thesection}{\Alph{section}}
\renewcommand{\theequation}{\Alph{section}-\arabic{equation}}
\renewcommand{\thelemma}{\Alph{section}-\arabic{lemma}}
\renewcommand{\thecorollary}{\Alph{section}-\arabic{corollary}}
\renewcommand{\thesubsection}{\Alph{section}-\arabic{subsection}}
\renewcommand{\thesubsubsection}{\Alph{section}-\arabic{subsection}.\arabic{subsubsection}}
  % redefine the command that creates the equation no.
\setcounter{equation}{0}  % reset counter
\setcounter{lemma}{0}
\setcounter{corollary}{0}
\setcounter{section}{0}
\section{A remark on robustness of spectral distributions}
This technical appendix is not directly related to the rest of the paper but shows how the methods we used can be utilized to analyze the robustness of another quantity of interest in random matrix theory, namely the spectral distribution of the matrix. (We put the result here because it fits our theme of robustness and is interesting but of course does not warrant its own paper.)

In \cite{nekCorrEllipD}, we investigated the robustness properties of generalizations of the \mapa equation and showed that it held under mild concentration requirements on the data.

As a first step we showed that we could use Azuma's inequality to control the fluctuations of the Stieltjes transform for a very broad class of distributions. Now to show robustness, all we have to do is show that the expectation of the Stieltjes transform is the same for all the models we consider. In \cite{nekCorrEllipD}, we limited ourselves to models for which the data $\ellipticalData_i$ had the same covariance for all $i$. We can now use similar ideas to the ones we have developed in this paper to do it in a more general case. We call
$$
{\cal S}_X=\frac{1}{n}\sum_{i=1}^n R_i^2 X_i X_i\trsp
$$
and the corresponding Stieltjes transform (for ${\cal S}_X+A$)
$$
m_{p,X}(z)=\frac{1}{p}\trace{({\cal S}_X+A-z\id)^{-1}}\;,
$$
where $A$ is a (deterministic) psd matrix, $z\in \mathbb{C}^+$ and $\imag{z}=v>0$. We call $u=\myreal{z}$.

We have the following theorem.

\begin{theorem}
Under the usual assumptions of this paper (see Subsection \ref{subsec:SetupOfStudy}), assuming that the $R_i$'s are deterministic, and $\{X_i\}_{i=1}^n$ and $\{Y_i\}_{i=1}^n$ have mean 0 and are such that $\scov{X_i}=\scov{Y_i}$, we have, for any fixed $z \in \mathbb{C}^+$,
$$
|\Exp{m_{p,X}(z)-m_{p,Y}(z)}|\leq \frac{1}{p} \sum_{i=1}^n \left(\frac{K|z|}{v^3} \frac{R_i^2}{n} \left[b_{Q_2}(1;X_i)+b_{Q_2}(1;Y_i)\right]\right) \wedge \frac{2}{v}\;.
$$
\end{theorem}
This extends some of the results of \cite{nekCorrEllipD}, since under various concentration assumptions we will be able to control $b_{Q_2}(1;Y_i)$ and $b_{Q_2}(1;X_i)$ (recall that when $X_i$ are Gaussian with covariance bounded in operator norm, $b_{Q_2}(1;X_i)$ is of order $\sqrt{p}$). Note once again that the models considered here are richer than the ones considered in \cite{nekCorrEllipD}. The main difference with the results of \cite{nekCorrEllipD} is that this new theorem covers cases where we cannot describe the limit, whereas in \cite{nekCorrEllipD} we described the limit ``explicitly".

We refer the reader to \cite{nekCorrEllipD} (or \cite{baisilverstein10}) for details explaining why showing a.s convergence of the Stieltjes transform at each $z$ (and a mass preservation condition) gives a.s weak convergence of the spectral distribution. Essentially our theorem says that the existence of a limit needs to be checked only in the Gaussian case and that such a result would transfer over to more general distributions for which we control $b_{Q_2}(1;Y_i)$.

\begin{proof}
We go quick on the details of the proof because we have done many similar ones in the paper. We take a Lindeberg approach, naturally. It is clear that if $B_j=\frac{1}{n}\sum_{k=1}^{j-1} R^2_k X_k X_k\trsp+
\frac{1}{n}\sum_{k=j+1}^{n} R^2_k Y_k Y_k\trsp+A$ (with obvious adjustments mentioned in the paper for $j=1$ and $j=n$), all we have to do is understand
$$
\frac{1}{p}\Exp{\trace{(B_j+\frac{R_j^2}{n}X_jX_j\trsp-z\id)^{-1}}-\trace{(B_j+\frac{R_j^2}{n}Y_iY_j\trsp-z\id)^{-1}}}\;.
$$
Let us call $B_j(z)=B_j-z\id$. Note that $B_j$ is psd.
By standard rank-1 updates arguments, we have
\begin{align*}
\Delta_j&=\trace{(B_j(z)+\frac{R_j^2}{n}X_jX_j\trsp)^{-1}}-\trace{(B_j(z)+\frac{R_j^2}{n}Y_jY_j\trsp)^{-1}}\\
&=
\frac{R_j^2}{n}\left[-\frac{X_j\trsp B_j^{-2}(z)X_j}{1+\frac{R_j^2}{n}X_j\trsp B_j^{-1}(z)X_j}+\frac{Y_j\trsp B_j^{-2}(z)Y_j}{1+\frac{R_j^2}{n}Y_j\trsp B_j^{-1}(z)Y_j}\right]
\end{align*}
Let us call $d_j(z)=\trace{B_j(z)^{-1}\Sigma_j}$, where $\Sigma_j$ is the covariance of $X_j$ and $Y_j$. Clearly, since $d_j(z)$ is independent of $X_j$ and $Y_j$,
$$
\Exp{\frac{X_j\trsp B_j^{-2}(z)X_j}{1+\frac{R_j^2}{n}d_j(z)}}=\Exp{\frac{Y_j\trsp B_j^{-2}(z)Y_j}{1+\frac{R_j^2}{n}d_j(z)}}=\Exp{\frac{\trace{B_j^{-2}(z)\Sigma_j}}{1+\frac{R_j^2}{n}d_j(z)}}\;.
$$
So to control $|\Exp{\Delta_j}|$, all we have to do is control, if we call $q_j(z)=X_j\trsp B_j^{-1}(z)X_j$,
$$
\left|\Exp{\frac{R_j^2}{n}\frac{X_j\trsp B_j^{-2}(z)X_j}{1+\frac{R_j^2}{n}d_j(z)}-\frac{R_j^2}{n}\frac{X_j\trsp B_j^{-2}(z)X_j}{1+\frac{R_j^2}{n}q_j(z)}}\right|\;.
$$
The quantity inside the expectation can be rewritten
$$
\Omega_j=\frac{R_j^2}{n}X_j\trsp B_j^{-2}(z)X_j\frac{R_j^2}{n}\frac{q_j(z)-d_j(z)}{(1+\frac{R_j^2}{n}d_j(z))(1+\frac{R_j^2}{n}q_j(z))}\;.
$$
Lemma 2.6 in \cite{silversteinbai95} shows that
$$
\left|\frac{R_j^2}{n}\frac{X_j\trsp B_j^{-2}(z)X_j}{1+\frac{R_j^2}{n}q_j(z)}\right|\leq \frac{1}{v}
$$
Hence, $|\Delta_j|\leq 2/v$ and
$$
|\Exp{\Omega_j}|\leq \frac{1}{v} \frac{R_j^2}{n} \Exp{\frac{|d_j(z)-q_j(z)|}{|1+\frac{R_j^2}{n} d_j(z)|}}\;.
$$
By writing $B_j^{-1}(z)$ in terms of its eigenvalues and eigenvectors, we note that $\imag{z\trace{B_j^{-1}(z)\Sigma_j}}\geq 0$ because $B_j$ and $\Sigma_j$ are psd and $z\in \mathbb{C}^+$ (alternatively, $\imag{zB_j^{-1}(z)}$ is psd). Therefore $\imag{zd_j(z)}\geq 0$. Hence,
$$
\frac{1}{|z(1+\frac{R_j^2}{n}d_j(z))|}\leq \frac{1}{v}
$$
So finally,
$$
|\Exp{\Omega_j}|\leq \frac{|z|}{v^2} \frac{R_j^2}{n} \Exp{|d_j(z)-q_j(z)|}\;.
$$
We now have to analyze $d_j(z)-q_j(z)$. We notice that
$$
q_j(z)-d_j(z)=X_j\trsp M_1 X_j + i X_j\trsp M_2 X_j-\Expj{j}{X_j\trsp M_1 X_j + i X_j\trsp M_2 X_j}\;,
$$
where if $\alpha_k$'s are the eigenvectors of $B_j$ and $\lambda_k$ its eigenvalues, we have
$$
\myreal{B_j^{-1}(z)}=M_1=\sum_{k=1}^p \frac{\lambda_k-u}{(\lambda_k-u)^2+v^2} \alpha_k \alpha_k\trsp
$$
and
$$
\imag{B_j^{-1}(z)}=M_2=\sum_{k=1}^p \frac{v}{(\lambda_k-u)^2+v^2} \alpha_k \alpha_k\trsp \;.
$$
$M_1$ can be written as $M_1=M_{1,+}-M_{1,-}$, where $M_{1,+}$ is formed by keeping the non-negative eigenvalues of $M_1$ and replacing the negative ones by 0. Of course, $M_{1,+}$ and $M_{1,-}$ are psd (technically we should index them by $u$, but we do not do it to alleviate the notation).
We now remark that $M_{1,+}$, $M_{1,-}$ and $M_2$ are psd with $\norm{M_{1,\pm}}\leq 1/v$ and $\norm{M_2}\leq 1/v$.
We can therefore conclude, using the fact that $|z|\leq |\myreal{z}|+|\imag{z}|$ as well as the fact that $M_1$ and $M_2$ are independent of $X_j$ that
$$
\Exp{|d_j(z)-q_j(z)|}\leq \frac{K}{v} b_{Q_2}(1;X_j)\;.
$$
Putting everything together we obtain the result announced in the theorem.
\end{proof}

\bibliographystyle{C:/NekTexAuxiliaries/Bibliography/annstats}
\bibliography{C:/NekTexAuxiliaries/Bibliography/research,C:/NekTexAuxiliaries/Bibliography/Bioresearch}

\end{document}